\documentclass[10pt,compress]{article}

\usepackage{amsfonts,amssymb,amsmath,amsthm,indentfirst}
\usepackage{fontenc,enumerate}
\usepackage{graphics,cite}
\usepackage{color,graphicx}
\usepackage[dvipsnames]{xcolor}
\usepackage[colorlinks, linkcolor=red, citecolor=Green, urlcolor=OliveGreen, pagebackref, hypertexnames=false]{hyperref}

\usepackage{latexsym,mathrsfs,enumerate,verbatim,bm}

\usepackage[light]{antpolt}    \usepackage[T1]{fontenc}

\usepackage{wasysym}
\usepackage{graphics}

\usepackage{color,graphicx}
\newcommand{\Sch}{\mathcal{S}}
\newcommand{\cA}{\mathcal{A}}
\newcommand{\cC}{\mathcal{C}}

\DeclareMathAlphabet{\mathpzc}{OT1}{pzc}{m}{it}

\newcommand{\sett}[1]{\left\{   #1   \right\}}
\newcommand{\norm}[1]{\left\|   #1   \right\|}

\usepackage{xcolor}
\usepackage[pagewise,mathlines]{lineno}

\newcommand{\vertiii}[1]{{\left\vert\kern-0.25ex\left\vert\kern-0.25ex\left\vert #1 
		\right\vert\kern-0.25ex\right\vert\kern-0.25ex\right\vert}}
\numberwithin{equation}{section}
\newtheorem{theorem}{\quad Theorem}[section]
\newtheorem{lemma}[theorem]{\quad Lemma}
\newtheorem{corollary}[theorem]{\quad Corollary}
\newtheorem{remark}[theorem]{\quad Remark}

\newtheorem{proposition}[theorem]{\quad Proposition}

\newcommand{\R}{\mathbb{R}}

\newcommand{\Nn}{\mathcal N}   % L^B-mass functional
\newcommand{\G}{\mathcal G} 
\newcommand{\dd}{{\rm d}}

\newcommand{\N}{\mathbb{N}}

\newcommand{\bb}{\mathbb{B}}
\allowdisplaybreaks
\author{  Amin Esfahani\footnote{Department of Mathematics, Nazarbayev University, Astana 010000, Kazakhstan\newline   E-mail: amin.esfahani@nu.edu.kz, achenef.tesfahun@nu.edu.kz}~     }

\title{ Weighted Gagliardo–Nirenberg Inequalities of Grushin Type: Symmetries, Compactness, and Optimizers
	}

\date{}

\begin{document}
 
	\maketitle
	\begin{abstract}
		
 In this paper, we explore a family of weighted Gagliardo–Nirenberg interpolation inequalities associated with Grushin-type operators on $\mathbb{R}^n \times \mathbb{R}^m$. Specifically, we examine how to control the weighted norm $\||x|^\beta u \|_{L^B}$ by balancing fractional derivatives $D_x^s u$ with weighted derivatives $|x|^b D_y^r u$. Using two-parameter scaling, spatial translations, and Littlewood–Paley decompositions, we map out exact parameter conditions and demonstrate that the scaling exponents are both necessary and optimal. For multi-parameter weighted extensions, we leverage Khintchine-type random sums to show why $Q \ge 2$ is strictly required.When $x$ and $y$ are radially symmetric, two key advantages emerge: the allowable parameter range for the weights expands, and compact embeddings of the energy space into $L^p$, $2 < p < B$, are completely restored. We establish this compactness through two complementary tools—geometric covering arguments and a generalized three-norm version of Strauss's lemma that yields explicit pointwise weighted decay estimates. Finally, in the scale-invariant case  $\beta=0$, a concentration-compactness analysis proves that the inequality's sharp constant is achieved by an optimal function, with minimizing sequences remaining precompact up to the natural three-parameter group of symmetries.
			\end{abstract}

	%\end{frontmatter}
 
%\tableofcontents
 \section{Introduction}
 
 Interpolation inequalities of Gagliardo--Nirenberg and
 Caffarelli--Kohn--Nirenberg type are basic tools in the analysis of
 nonlinear PDEs and in the study of sharp constants and extremal functions
 \cite{Gagliardo,Nirenberg,CKN,Weinstein83,ChouChu93,CatrinaWang01,FelliSchneider03,DEL16}.
 Weighted $L^2$-type interpolation (Hardy--Sobolev, Stein--Weiss, Pitt and
 Heisenberg--Pauli--Weyl inequalities) is closely tied to the
 uncertainty principle  \cite{SteinWeiss58,Herbst77,Beckner95,Beckner08}.
 In this paper we study an inequality in which the \emph{anisotropy} of the
 variables is built in: for $(x,y)\in\mathbb R^n\times\mathbb R^m$ we control a
 weighted $L^B$-norm of $u$ by a derivative of order $s$ in $x$ and a
 \emph{weighted} derivative of order $r$ in $y$,
 \begin{equation}\label{intro:main}
 	\bigl\||x|^{\beta}u\bigr\|_{L^B(\mathbb R^{n+m})}
 	\le C\,\bigl\||x|^{b}D_y^{r}u\bigr\|_{L^2}^{\tau}\,
 	\|D_x^{s}u\|_{L^2}^{1-\tau}.
 \end{equation}

 See also \cite{liyan}. For $s=r=1$ and $b=\alpha\ge 0$ the right-hand side is the interpolation
 of the two terms of the Dirichlet form
 $\int|\nabla_xu|^2+|x|^{2\alpha}|\nabla_yu|^2$ of the Grushin operator
 $-\Delta_x-|x|^{2\alpha}\Delta_y$, which degenerates on $\{x=0\}$
 \cite{Grushin70,Grushin71}. This operator and its sub-Riemannian
 relatives have been studied extensively: hypoellipticity and regularity
 \cite{Grushin70,FranchiLanconelli83,Garofalo93}, Sobolev inequalities,
 sharp constants and critical equations
 \cite{Beckner01,GarofaloVassilev01,MontiMorbidelli06,Monti06,DSWZ22},
 semilinear problems \cite{KogojLanconelli12,AlvesHolanda23}, and
 control/observability \cite{BurqSun22}. Replacing $\nabla_x$ and
 $|x|^b\nabla_y$ by $D_x^s$ and $|x|^bD_y^r$ gives a fractional,
 two-parameter family which contains the Grushin Sobolev inequality as
 the case $\beta=0$ and exhibits the genuinely anisotropic
 scaling $u\mapsto u(\lambda x,\mu y)$ of the problem.

 Theorem~\ref{thm:main-w} proves \eqref{intro:main} under explicit
 conditions \textup{(H0)--(H4)} on $(m,n,s,r,b,\beta)$. The exponents are
 \[
 B=\frac{2\,(rn+m(b+s))}{m(b+s)+r(n-2s-2\beta)},\qquad
 \tau=\frac{m(s+\beta)}{rn+m(b+s)},
 \]
 and for $\beta=0$ one recovers the Sobolev-type inequality with the
 ``homogeneous dimension'' $Q_0=n+m(b+s)/r$; for $s=r=1$, $b=\alpha$,
 this is the classical Grushin exponent $2Q_0/(Q_0-2)$ with
 $Q_0=n+(1+\alpha)m$. Proposition~\ref{prop:necessity} shows that $B$ and
 $\tau$ are forced by the two-parameter scaling and that the condition
 $\beta\le\tau b$ and the strict inequality $\tau<1$ are necessary, so
 the parameter range cannot be improved without symmetry. The proof is
 short and modular: a Gagliardo--Nirenberg inequality in $y$ on each
 slice $\{x\}\times\mathbb R^m$, H\"older's inequality in $x$, and a
 Hardy--Sobolev inequality in $x$ applied to the profile
 $U_2(x)=\|u(x,\cdot)\|_{L^2_y}$, whose $\dot H^s$-norm is controlled by
 $\|D^s_xu\|_{L^2}$ (Lemma~\ref{lem:diam}). For $\beta=0$ we also give
 a second proof (Section \ref{sec-secondproof}) based on a Littlewood--Paley decomposition
 in $y$, which isolates the scale-invariant estimate on each frequency
 shell.

 In Section \ref{sec-extension} we consider a more general multi-parameter family
 $\|u\|_{L^Q}\le C\prod_jA_j(u)^{\theta_j}$ with coordinate-wise weights
 and orders. We do not prove sufficiency; instead we determine all
 necessary conditions that follow from scaling, modulation, translation
 to infinity and a Khintchine-type averaging argument (Theorem~\ref{s2:thm:main}).
 This identifies, in particular, the forced values of $Q$ and $\theta$
 and shows that $Q\ge2$.

 Radial symmetry typically enlarges the admissible range of weights
 in Hardy--Sobolev and Caffarelli--Kohn--Nirenberg type inequalities
 \cite{Strauss77,Lions82,SuWangWillem07,DrelichmanDuran11,DeNapoli14,DDD12}.
 In Section \ref{sec-radialgrushin} we show that radial symmetry \emph{in $x$} replaces the
 translation-obstruction condition $\beta\le\tau b$ by the weaker
 condition \textup{(R1)}--\textup{(R2)}, using the radial weighted
 Sobolev inequality of Drelichman and Dur\'an in the $x$-step of the
 proof above, and that this is sharp for $s<\tfrac12$. Symmetry in $y$
 alone gives no improvement, and adding it to symmetry in $x$ gives none
 either (Corollary~\ref{s3cor:opt}).

 Radial symmetry also restores compactness, in the spirit of Strauss and
 Lions \cite{Strauss77,Lions82,BerestyckiLions83,SickelSkrzypczak00,ChoOzawa09}.
 We prove that the embedding of the energy space
 $X=\{u:\ \|u\|_2+\|D_x^su\|_2+\||x|^bD_y^ru\|_2<\infty\}$ into $L^p$,
 $2<p<B$, is compact on functions radial in $x$ and $y$
 (Sections \ref{sec-compact} and \ref{sec-compact-struss}), by two different mechanisms: a covering of annuli by
 rotated pieces (Section \ref{sec-compact}), and a three-norm version of Strauss' lemma
 (Sections \ref{sec-strauss} and \ref{sec-compact-struss}) that yields in addition pointwise weighted decay
 estimates (Theorem~\ref{thm:main-r}). These embeddings are the natural
 input for existence of ground states and standing waves for
 Grushin-type nonlinear Schr\"odinger equations
 \cite{BerestyckiLions83,AlvesHolanda23}.

 Finally, for $\beta=0$ the inequality is invariant under a
 three-parameter group of dilations and amplitude changes, and the
 embedding is not compact for any symmetry that fails to break this
 group. In Section \ref{sec-minimizers} we prove that the sharp constant is attained, and
 that minimizing sequences are precompact modulo this group, via a
 concentration-compactness argument
 \cite{Lions84a,Lions84b,Lions85,BrezisLieb83}. The main difficulties
 are: (i) the anisotropic dilations leave a one-parameter subgroup
 of the triple of norms fixed, which is used to normalize the
 concentration scale; (ii) the weight $|x|^b$ breaks translation
 invariance in $x$, so the concentration point cannot be moved to the
 origin and must be shown to stay bounded; (iii) for fractional $s,r$
 the energies are nonlocal, so we localize using a quantitative
 splitting (Lemma~\ref{em-lem:split}) in the spirit of
 \cite{DietzeNam25}.
 
 \paragraph{Organization.}
 Section \ref{sec:statement} contains the main inequality and its optimality. Section \ref{sec-secondproof}
 gives a second proof for $\beta=0$. Section \ref{sec-extension} treats the multi-weight
 problem. Sections \ref{sec-compact}-\ref{sec-compact-struss} prove compact embeddings. Section \ref{sec-radialgrushin}
 establishes the radially improved inequality and its sharpness.
 Section \ref{sec-minimizers} proves existence of optimizers.
 \medskip
 
 \noindent\textbf{Notation.} Fourier transform, $D^\sigma$, $D^s_x$, $D^r_y$
 and function classes are as in Section \ref{sec:statement}; $C$ denotes constants that may
 change from line to line.
	
	\section{Main inequality}\label{sec:statement}
	%==========================================================================
	
	Throughout this section, $m,n\ge 1$ are integers, the variable in $(x,y)\in\R^{n}\times\R^{m}$ is written $(x,y)$, and the exponents $s\in(0,1]$, $r>0$ and the weights $b,\beta\in\R$ are fixed.   For a function $u$ of $(x,y)$ the operators $D_x^{s}$ and $D_y^{r}$ are the multipliers with symbols $|\xi|^s$ and $|\eta|^r$ with respect to the partial Fourier transform in $x$ and in $y$, respectively.

The following theorem establishes the main mixed-norm fractional inequality, provided the dimensional parameters and weights satisfy a specific compatibility regime.	
	\begin{theorem}\label{thm:main-w}
		Let $m,n\ge1$, $s\in(0,1]$, $r>0$ and $b,\beta\in\R$, and assume the following five conditions:
		\begin{gather}
			2s<n, \tag{H0}\label{H0}\\
			s+\beta>0, \tag{H1}\label{H1}\\
			rn+m(b+s)>m(s+\beta)
			%\quad\big(\text{equivalently } \beta<b+\tfrac{rn}{m}\big), 
			\tag{H2}\label{H2}\\
			rn+m(b+s)>2r(s+\beta), \tag{H3}\label{H3}\\
			\beta\,(rn+ms)\le ms b%\quad\big(\text{equivalently } \beta\le\tau b\big). 
			.\tag{H4}\label{H4}
		\end{gather}
		Then, there is a constant $C=C(m,n,s,r,b,\beta)>0$ such that
		\begin{equation}\label{eq:main}
			\norm{|x|^{\beta}u}_{L^{B}(\R^{n+m})} \leq  C\norm{|x|^{b}D_y^{r}u}_{L^2(\R^{n+m})}^{\tau}\norm{D_x^{s}u}_{L^2(\R^{n+m})}^{1-\tau}
			\qquad\text{for all }u\in\Sch(\R^{n+m}),
		\end{equation}
		where
		\[
		B=\dfrac{2\,(rn+m(b+s))}{m(b+s)+r(n-2s-2\beta)},\qquad \tau=\frac{m(s+\beta)}{rn+m(b+s)}
		\]
	\end{theorem}

	To clarify the scope and natural boundaries of this result, it is helpful to unpack the algebraic relationships governing these parameters.
	\begin{remark}\label{rem:hyp}
		Before discussing the hypotheses, we record the elementary consequences of \eqref{H1}--\eqref{H4} for the parameters. Let, $\Gamma= rn+m(b+s)$ and $\delta=s+\beta$. Then, $\tau=\frac{m\varrho} \Gamma =\frac{m\delta}{r}$  and $B=\frac{2\Gamma}{\Gamma-2r\varrho}=\frac{2}{1-2\delta}.$
		By \eqref{H1} and \eqref{H2} we have $\Gamma>m\varrho>0$, hence $\Gamma>0$ and $\tau=m\varrho/\Gamma\in(0,1)$. Since $\delta=r\tau/m>0$ and \eqref{H3} is equivalent to $\delta<\tfrac12$, we get $0<\delta<\tfrac12$ and $B=2/(1-2\delta)\in(2,\infty)$. Finally, $\beta\le\tau b$ means $\beta\Gamma\le m(s+\beta)b$, that is $\beta(rn+mb+ms)\le msb+m\beta b$, which is exactly \eqref{H4}.
		Condition \eqref{H3} is automatic from \eqref{H2} as soon as $m\ge 2r$, in particular whenever $m\ge2$ and $r\le1$; it matters only for $m=1$ and $r>\tfrac12$. For $\beta=0$ the conditions reduce to $b\ge0$ together with \eqref{H0} and $\Gamma>2rs$, and the theorem gives $\|u\|_{B}\lesssim\||x|^bD_y^ru\|_2^{\tau}\|D_x^su\|_2^{1-\tau}$ with $\tau=ms/(rn+m(b+s))$. Condition \eqref{H0} is only used in Lemma~\ref{lem:HS} below; it holds for every $s\in(0,1]$ when $n\ge3$, for $s<1$ when $n=2$, and for $s<\tfrac12$ when $n=1$. The proof uses only the algebraic facts $B\in(2,\infty)$, $\tau\in(0,1)$, $\beta\le\tau b$ and \eqref{H0}; it therefore also applies in any other sign configuration of $(\varrho,\Gamma)$ for which these four facts hold. If $b\le -n/2$ then the first term on the RHS of \eqref{eq:main} becomes unbounded for generic $u$, and \eqref{eq:main} is trivially true for such $u$; the proof below does not require $b>-n/2$.
	\end{remark}
	
	The conditions imposed above are not merely artifacts of the proof. The exponents $B,\tau$ and the weight constraint \eqref{H4} cannot be improved, as formalized in the following necessity result.
	
	\begin{proposition}[Necessity]\label{prop:necessity}
		Let $B\in[1,\infty)$ and $\tau>0$ , and suppose that
		\begin{equation}\label{eq:ineqK}
			\||x|^\beta u\|_{B}\lesssim  \cA(u)^{\tau}\,\norm{D_x^{s}u}_{L^2(\R^{n+m})}^{1-\tau}
		\end{equation}
		holds for all $u\in C_c^\infty(\R^{n+m})$ whose support is contained in $\{x\ne0\}$. Then the following three statements hold.
		\begin{gather}
			\frac{m}{B}=\frac m2-\tau r \quad\text{and}\quad \frac nB=\frac n2-\big(s+\beta-\tau(b+s)\big); \label{eq:scal}\\
			\beta\le\tau b; \label{eq:transl}\\
			\tau<1. \label{eq:tau1}
		\end{gather}
	\end{proposition}

	    To build the proof of Theorem~\ref{thm:main-w}, we isolate three key ingredients. The first is a fractional Gagliardo-Nirenberg interpolation inequality that will be applied to the $y$-variable.
	\begin{lemma} \label{lem:GN}
		Let $d\ge1$, $\sigma>0$ and $p\in(2,\infty)$ with $\frac1p>\frac12-\frac\sigma d$. Put $\nu=\frac d\sigma\big(\frac12-\frac1p\big)$. Then $\nu\in(0,1)$ and there is a constant $C$ such that
		\[
		\|f\|_{L^p(\R^d)}\le C\,\|D^{\sigma}f\|_{L^2(\R^d)}^{\nu}\,\|f\|_{L^2(\R^d)}^{1-\nu}\qquad\text{for all }f\in\Sch(\R^d).
		\]
	\end{lemma}
	
	Our second tool is a fractional Hardy-Sobolev inequality, which will manage the singular weights appearing in the $x$-variable.

	\begin{lemma} \label{lem:HS}
		Let $m\ge1$, $0<r<\frac m2$, $0\le a\le r$ and $\frac1g=\frac12-\frac{r-a}{m}$. Then there is a constant $C$ such that
		\[
		\big\||x|^{-a}f\big\|_{L^{g}(\R^m)}\le C\,\|D^{r}f\|_{L^2(\R^m)}\qquad\text{for all }f\in H^{r}(\R^m).
		\]
	\end{lemma}
	Finally, we require a structural lemma ensuring that taking the $L^2$-norm in the $y$-variable preserves fractional Sobolev regularity in the $x$-variable.

	\begin{lemma}\label{lem:diam}
		Let $0<s\le1$, $u\in\Sch(\R^{n+m})$, and $U_2(x)=\|u(x,\cdot)\|_{L^2(\R^m)}$. Then $U_2\in H^s(\R^n)$ and $\|D^sU_2\|_{L^2(\R^n)}\le\|D_x^su\|_{L^2(\R^{n+m})}$.
	\end{lemma}
	
	\begin{proof}
		Note irst that $U_2\in L^2(\R^n)$ with $\|U_2\|_2=\|u\|_2$.
		
		Assume first $0<s<1$. For $f\in L^2(\R^n)$ put $\Phi_s(f)=\iint|f(x+h)-f(x)|^2|h|^{-n-2s}\,dx\,dh\in[0,\infty]$. By the Plancherel identity and the Tonelli theorem,
		\[
		\Phi_s(f)=\int|\widehat f(\xi)|^2\Big(\int|e^{ih\cdot\xi}-1|^2|h|^{-n-2s}dh\Big)d\xi=\kappa_{n,s}\int|\xi|^{2s}|\widehat f(\xi)|^2d\xi=\kappa_{n,s}\|D^sf\|_2^2,
		\]
		where the inner integral equals $\kappa_{n,s}|\xi|^{2s}$ by rotation invariance and the substitution $h\mapsto h/|\xi|$, and $\kappa_{n,s}=\int|e^{ih_1}-1|^2|h|^{-n-2s}dh\in(0,\infty)$ is finite because $|e^{ih_1}-1|^2\lesssim\min(|h|^2,1)$ and $0<s<1$. Applying this identity to $u(\cdot,y)$ for each fixed $y$, integrating in $y$, and using the Plancherel identity in the $x$-variable (for each fixed $y$) together with the Tonelli theorem, we get
		\[
		\kappa_{n,s}\|D_x^su\|_2^2=\iint\big\|u(x+h,\cdot)-u(x,\cdot)\big\|_{L^2(\R^m)}^2|h|^{-n-2s}\,dx\,dh.
		\]
		By the triangle inequality in $L^2(\R^m)$, $|U_2(x+h)-U_2(x)|\le\|u(x+h,\cdot)-u(x,\cdot)\|_{L^2(\R^m)}$, hence $\kappa_{n,s}\|D^sU_2\|_2^2=\Phi_s(U_2)\le\kappa_{n,s}\|D_x^su\|_2^2$.
		
		Assume now $s=1$. Since $u\in\Sch$, the map $x\mapsto u(x,\cdot)\in L^2(\R^m)$ is $C^1$ with bounded derivative. By the triangle inequality, $U_2$ is Lipschitz, so it is differentiable almost everywhere by Rademacher's theorem. At a point of differentiability and for a unit vector $e$,
		\[
		\begin{split}
			|\partial_eU_2(x)|&=
			\lim_{t\to0}\frac{|U_2(x+te)-U_2(x)|}{|t|}\le\lim_{t\to0}\norm{\frac{u(x+te,\cdot)-u(x,\cdot)}{t}}_{L^2(\R^m)}\\&
			=\norm{\partial_eu(x,\cdot)}_{L^2(\R^m)}\le\norm{|\nabla_xu (x,\cdot)}_{L^2(\R^m)}.
		\end{split}
		\]
		Taking the supremum over $e$ gives $|\nabla U_2(x)|\le\||\nabla_xu|(x,\cdot)\|_{L^2(\R^m)}$, and integration in $x$ yields $\|\nabla U_2\|_2\le\|\nabla_xu\|_2$. A Lipschitz function in $L^2$ with gradient in $L^2$ lies in $H^1$, and $\|D^1f\|_2=\|\nabla f\|_2$ by the Plancherel identity.
	\end{proof}
	
	%==========================================================================
	
	%==========================================================================

 Equipped with these three lemmas, we are now ready to assemble the proof of the main inequality.
 
	\begin{proof}[Proof of Theorem \ref{thm:main-w}]
		
		By the discussion after Theorem~\ref{thm:main-w} we have $0<\delta<\frac12$, $0<\tau<1$ and $\beta\le\tau b$. Define
		\begin{equation}\label{eq:defag}
			a=\frac{\tau b-\beta}{1-\tau},\qquad \frac1g=\frac12-\frac{\delta}{1-\tau}.
		\end{equation}
		We claim that
		\begin{equation}\label{eq:claims}
			\begin{gathered}
				\beta=\tau b-(1-\tau)a,\qquad s-a=\frac{n\delta}{1-\tau},\qquad \frac1B=\frac\tau2+\frac{1-\tau}{g},\\
				\frac1g=\frac12-\frac{s-a}{n},\qquad 0\le a<s,\qquad 2<g<\infty.
			\end{gathered}
		\end{equation}
		The first identity is the definition of $a$. For the second, using $\varrho=s+\beta$, $\Gamma-m(b+s)=rn$ and $\delta=r\varrho/\Gamma$,
		\[
		(1-\tau)(s-a)=s(1-\tau)-\tau b+\beta=\varrho-\tau(b+s)=\varrho\,\frac{\Gamma-m(b+s)}{\Gamma}=\frac{\varrho\, rn}{\Gamma}=n\delta .
		\]
		For the third, recall $\frac1B=\frac12-\delta$; then
		$\frac\tau2+\frac{1-\tau}g=\frac\tau2+\frac{1-\tau}2-\delta=\frac12-\delta=\frac1B .$
		The fourth identity is the second one inserted in \eqref{eq:defag}. Finally $a\ge0$ is the hypothesis $\beta\le\tau b$ divided by $1-\tau>0$; $s-a=n\delta/(1-\tau)>0$ gives $a<s$ and $g>2$; and $a\ge 0$ together with \eqref{H0} gives $s-a\le s<\frac n2$, hence $\frac1g>0$ and $g<\infty$. In particular $(a,g)$ satisfy the hypotheses of Lemma~\ref{lem:HS} with the present $s$ and $n$, because $0\le a\le s<\frac n2$.
		
		Now, for $x\in\R^n$ put %\cA(u)=\norm{|x|^{b}D_y^{r}u}_{L^2(\R^{n+m})},\qquad
		% \cC(u)=\norm{D_x^{s}u}_{L^2(\R^{n+m})} .
		\[
		F(x)=\|D_y^ru(x,\cdot)\|_{L^2(\R^m)},\qquad U_2(x)=\|u(x,\cdot)\|_{L^2(\R^m)},\qquad U_B(x)=\|u(x,\cdot)\|_{L^B(\R^m)} .
		\]
		By the Tonelli theorem, $\||x|^bF\|_{L^2(\R^n)}=\norm{|x|^{b}D_y^{r}u}_{L^2(\R^{n+m})}$ and $\||x|^\beta U_B\|_{L^B(\R^n)}=\||x|^\beta u\|_{L^B(\R^{n+m})}$. For fixed $x$ the function $u(x,\cdot)$ belongs to $\Sch(\R^m)$, and we apply Lemma~\ref{lem:GN} with $d=m$, $\sigma=r$ and $p=B$. The hypotheses are satisfied: $B\in(2,\infty)$, and $\frac1B>\frac12-\frac rm$ is equivalent to $\delta<\frac rm$, that is to $\tau=\frac{m\delta}{r}<1$. The corresponding exponent is $\nu=\frac mr(\frac12-\frac1B)=\frac{m\delta}r=\tau$. Hence
		\begin{equation}\label{eq:step1}
			U_B(x)\le C\,F(x)^{\tau}\,U_2(x)^{1-\tau}\qquad\text{for every }x\in\R^n .
		\end{equation}
		
		\noindent Multiplying \eqref{eq:step1} by $|x|^\beta$ and using $\beta=\tau b-(1-\tau)a$ from \eqref{eq:claims}, we obtain for $x\ne0$
		\[
		|x|^{\beta}U_B(x)\le C\,\big(|x|^{b}F(x)\big)^{\tau}\,\big(|x|^{-a}U_2(x)\big)^{1-\tau}.
		\]
		By the third identity in \eqref{eq:claims}, $\frac1B=\frac{\tau}{2}+\frac{1-\tau}{g}$, so Hölder's inequality with exponents $\frac2\tau$ and $\frac{g}{1-\tau}$ gives
		\begin{equation}\label{eq:step2}
			\||x|^{\beta}u\|_{L^B(\R^{n+m})}\le C\,\||x|^bF\|_{L^2(\R^n)}^{\tau}\,\||x|^{-a}U_2\|_{L^g(\R^n)}^{1-\tau}=C\,\cA^{\tau}\,\||x|^{-a}U_2\|_{L^g(\R^n)}^{1-\tau}.
		\end{equation}
		
		\noindent By Lemma~\ref{lem:diam}, $U_2\in H^s(\R^n)$ and $\|D^sU_2\|_2\le\norm{D_x^{s}u}_{L^2(\R^{n+m})}$. By Lemma~\ref{lem:HS}, applicable by \eqref{eq:claims} and \eqref{H0},
		\begin{equation}\label{eq:step3}
			\||x|^{-a}U_2\|_{L^g(\R^n)}\le C\,\|D^sU_2\|_{L^2(\R^n)}\lesssim  \norm{D_x^{s}u}_{L^2(\R^{n+m})} .
		\end{equation}
		Inserting \eqref{eq:step3} into \eqref{eq:step2} gives   \eqref{eq:main}. \end{proof}

	%==========================================================================
	\medskip
	    We conclude this section by demonstrating that the parameter constraints are strictly necessary, relying on standard scaling, translation, and modulation arguments.
	
	\begin{proof}[Proof of Proposition~\ref{prop:necessity}]%\label{sec:nec}
		%==========================================================================
		
		Let $u\in C_c^\infty(\R^{n+m})$, $u\ne0$, with $\operatorname{supp}u\subset\{x\ne0\}$ and  set
		\begin{equation}\label{eq:defAC}
			\cA(u)=\big\||x|^{b}D_y^{r}u\big\|_{L^2(\R^{n+m})},\qquad
			\cC(u)=\big\|D_x^{s}u\big\|_{L^2(\R^{n+m})},
		\end{equation}
		so that $\||x|^\beta u\|_B$, $\cA(u)$ and $\cC(u)$ are finite and positive for all real $b,\beta$.
		%\norm{|x|^{b}D_y^{r}u}_{L^2(\R^{n+m})}
		
		{The scaling relations \eqref{eq:scal}.}
		For $\lambda,\mu>0$ let $u_{\lambda,\mu}(x,y)=u(\lambda x,\mu y)$; its support is again contained in $\{x\ne0\}$. A change of variables gives
		\begin{gather*}
			\||x|^\beta u_{\lambda,\mu}\|_B=\lambda^{-\beta-\frac nB}\mu^{-\frac mB}\||x|^\beta u\|_B,\qquad
			\norm{|x|^{b}D_y^{r}u_{\lambda,\mu}}_{L^2(\R^{n+m})}
			=\lambda^{-b-\frac n2}\mu^{r-\frac m2}\norm{|x|^{b}D_y^{r}u}_{L^2(\R^{n+m})},\\
			\norm{D_x^{s}u_{\lambda,\mu}}_{L^2(\R^{n+m})}=\lambda^{s-\frac n2}\mu^{-\frac m2}\norm{D_x^{s}u}_{L^2(\R^{n+m})}.
		\end{gather*}
		Applying \eqref{eq:ineqK} to $u_{\lambda,\mu}$ and dividing by the left-hand side we obtain, for all $\lambda,\mu>0$,
		\[
		1\le K'\,\lambda^{\,e_1+\beta+\frac nB}\,\mu^{\,e_2+\frac mB},\qquad e_1=-\tau\big(b+\tfrac n2\big)+(1-\tau)\big(s-\tfrac n2\big),\quad e_2=\tau r-\tfrac m2,
		\]
		with $K'=K\norm{|x|^{b}D_y^{r}u}_{L^2(\R^{n+m})}^\tau\norm{D_x^{s}u}_{L^2(\R^{n+m})}^{1-\tau}/\||x|^\beta u\|_B$ independent of $\lambda,\mu$. Letting $\lambda\to0$ and $\lambda\to\infty$ (with $\mu=1$) forces $e_1+\beta+\frac nB=0$, and letting $\mu\to0$ and $\mu\to\infty$ forces $e_2+\frac mB=0$. The second equation is $\frac mB=\frac m2-\tau r$. Using $e_1=-\tau b+(1-\tau)s-\frac n2$, the first equation reads $\frac nB=\frac n2-\big(s+\beta-\tau(b+s)\big)$, which is \eqref{eq:scal}.
		
		If $\Gamma\ne0$, set $\delta'=\frac12-\frac1B$. The two relations say $\delta'=\tau r/m$ and $n\delta'=\varrho-\tau(b+s)$. Substituting the first into the second gives $\tau(rn/m+b+s)=\varrho$, that is $\tau=m\varrho/\Gamma$, and then $\delta'=r\varrho/\Gamma=\delta$, so $B$ and $\tau$ are derived. The converse is the same computation read backwards.
		
		{The condition $\beta\le\tau b$.}
		Let $\rho_0>0$ be such that $\operatorname{supp}u\subset\{|x|\le\rho_0\}$ and, for $x_0\in\R^n$ with $|x_0|\ge2\rho_0$, let $u^{x_0}(x,y)=u(x-x_0,y)$. On the support of $u^{x_0}$ we have $\frac12|x_0|\le|x|\le\frac32|x_0|$, so that $\||x|^\beta u^{x_0}\|_B\ge c_1|x_0|^\beta\|u\|_B$ and $\cA(u^{x_0})\le c_2|x_0|^b\norm{|x|^{b}D_y^{r}u}_{L^2(\R^{n+m})}$ with constants $c_1,c_2>0$ depending only on $b,\beta$. Since $D_x^s$ commutes with translations, $\cC(u^{x_0})=\norm{D_x^{s}u}_{L^2(\R^{n+m})}$. Because $\tau>0$, inequality \eqref{eq:ineqK} gives $|x_0|^{\beta}\le C'|x_0|^{\tau b}$ for all $|x_0|\ge2\rho_0$, which is possible only if $\beta\le\tau b$.
		
		{The condition $\tau<1$.}
		We argue by contradiction, assuming $\tau\ge1$, and use product functions $u(x,y)=\phi(x)\psi(y)$ with $\phi,\psi$ smooth and compactly supported, $\operatorname{supp}\phi\subset\{x\ne0\}$.
		
		Suppose first $\tau>1$ and let $u_\lambda(x,y)=e^{i\lambda x_1}\phi(x)\psi(y)$ for $\lambda\ge1$. Then $|u_\lambda|=|u_1|$ and $D_y^ru_\lambda=e^{i\lambda x_1}\phi\,D_y^r\psi$, so $\||x|^\beta u_\lambda\|_B$ and $\cA(u_\lambda)$ do not depend on $\lambda$, whereas
		\[
		\cC(u_\lambda)^2=\|\psi\|_2^2\int|\xi+\lambda e_1|^{2s}|\widehat\phi(\xi)|^2d\xi\ \ge\ \|\psi\|_2^2\Big(\frac\lambda2\Big)^{2s}\int_{|\xi|\le\lambda/2}|\widehat\phi|^2d\xi\ \longrightarrow\ \infty\quad(\lambda\to\infty).
		\]
		Since $1-\tau<0$, the right-hand side of \eqref{eq:ineqK} tends to $0$ while the left-hand side is a fixed positive number, a contradiction.
		
		Suppose now $\tau=1$. By \eqref{eq:scal}, $\frac m B=\frac m2-r$, so $B>2$. Fix $x_0\ne0$, let $\phi_\varepsilon(x)=\phi((x-x_0)/\varepsilon)$ with $\phi\in C_c^\infty(B(0,1))$, $\phi\ne0$, and $\varepsilon<|x_0|/2$, and put $u_\varepsilon=\phi_\varepsilon\otimes\psi$. Then, for $\varepsilon\to0$,
		\[
		\||x|^\beta u_\varepsilon\|_B\ \ge\ c\,\varepsilon^{n/B}\|\phi\|_B\|\psi\|_B,\qquad \cA(u_\varepsilon)\le c'\,\varepsilon^{n/2}\|\phi\|_2\|D^r\psi\|_2,
		\]
		and for $\tau=1$ inequality \eqref{eq:ineqK} would give $\varepsilon^{n/B}\lesssim\varepsilon^{n/2}$ for all small $\varepsilon$, which is false because $B>2$.
	\end{proof}

	%%%%%%%%%%%%%%%%%%%%%%%%%%%%%%%%
 In this section, we present an alternative proof of the main inequality. While the previous approach relied on Gagliardo--Nirenberg and Hardy--Sobolev inequalities, this method is based on a Littlewood--Paley decomposition in the $y$-variable. This provides an independent and constructive perspective on the fractional bounds.
 
	\section{ {Second proof}}\label{sec-secondproof}
	
	\begin{theorem}
		Let $x\in\mathbb R^n$, $y\in\mathbb R^m$ ($n,m\ge1$), and $N=n+m$.   For $u\in\mathcal S(\mathbb R^N)$,
		$$\|u\|_{L^q(\mathbb R^N)}\le C\,\|D_x^su\|_{L^2}^{1-\tau}\||x|^\gamma D_y^ru\|_{L^2}^{\tau},$$
		where
		$$ \tau:= \frac{ms}{nr+m(\gamma+s)}\in(0,1).$$
		and
		$$q =\frac{2\,(nr+m(\gamma+s))}{m(\gamma+s)+nr-2rs}.$$
	\end{theorem}

	\begin{proof}
		Set $A=\|D_x^su\|_{L^2}$, $B=\||x|^\gamma D_y^ru\|_{L^2}$, $\alpha:=\frac nQ\in(0,1)$ and
		$$  Q = \frac{nr+m(\gamma+s)}{r}.$$
		Step 1.  Choose $\varphi\in C_c^\infty(\mathbb R^m)$ with $0\le\varphi\le1$, $\operatorname{supp}\varphi\subset\{\frac12<|\eta|<2\}$, and $\sum_{j\in\mathbb Z}\varphi(2^{-j}\eta)=1$ for $\eta\ne0$. Let $u_j$ be the multiplier in $y$ with symbol $\varphi(2^{-j}\eta)$. Then $u_j\in\mathcal S(\mathbb R^N)$, and for $2\le q<\infty$ and $f\in\mathcal S(\mathbb R^m)$ the Littlewood–Paley theorem gives

		$$
		\|f\|_{L^q(\mathbb R^m)}\le C\Big\|\Big(\sum_j|P_jf|^2\Big)^{1/2}\Big\|_{L^q(\mathbb R^m)}.
		$$
		Also $\sum_j\varphi^2\le\sum_j\varphi=1$.
		Take $\tilde\varphi\in C_c^\infty$ with $\tilde\varphi=1$ on $\operatorname{supp}\varphi$, and let $k=\mathcal F^{-1}\tilde\varphi$, $k_j=2^{jm}k(2^j\cdot)$. Then $u_j(x,\cdot)=k_j*_yu_j(x,\cdot)$. Young's inequality with $1+\frac1q=\frac1\rho+\frac12$, together with $\|k_j\|_\rho=2^{jm(1-1/\rho)}\|k\|_\rho$, gives

		$$\|u_j(x,\cdot)\|_{L^q_y}\le C\,2^{jm(\frac12-\frac1q)}\,g_j(x),\qquad g_j(x):=\|u_j(x,\cdot)\|_{L^2_y}.$$
		
		Integrating the $q$-th power in $x$ and squaring:

		\begin{equation}
			\|u_j\|_{L^q}^2\le C\,2^{jm(1-\frac2q)}\|g_j\|_{L^q(\mathbb R^n)}^2. \label{eq3.1}
		\end{equation}
		
		For $f\in L^1\cap L^2(\mathbb R^n)$ with $\|D^sf\|_2<\infty$,
		it is known that
		
		\begin{equation}
			\|f\|_{L^q}\le C_0\|D^sf\|_2^{\alpha}\|f\|_2^{1-\alpha}. \label{eq3.2}
		\end{equation}
		
		Step 2. We show that $\|D^sg\|_2\le\|D_x^su\|_2$.  Let $u\in\mathcal S(\mathbb R^N)$ and $g(x)=\|u(x,\cdot)\|_{L^2_y}$. Then $g$ is continuous, rapidly decaying, and

		$$|g(x)-g(x')|\le\|u(x,\cdot)-u(x',\cdot)\|_{L^2_y}.$$
		
		Case $0<s<1$. By the Plancherel identity in $x$, $\int|e^{ih\cdot\xi}-1|^2|h|^{-n-2s}dh=c_{n,s}|\xi|^{2s}$ with $c_{n,s}\in(0,\infty)$. Hence for scalar $f$,

		$$\iint\frac{|f(x)-f(x')|^2}{|x-x'|^{n+2s}}dxdx'=c_{n,s}\|D^sf\|_2^2.$$
		
		Applying this to $g$, then to $u(\cdot,y)$ for each $y$ and integrating in $y$ (Tonelli's theorem):

		$$c_{n,s}\|D^sg\|_2^2\le\iint\frac{\|u(x,\cdot)-u(x',\cdot)\|_{L^2_y}^2}{|x-x'|^{n+2s}}dxdx'=c_{n,s}\|D_x^su\|_{L^2(\mathbb R^N)}^2 .$$
		
		Case $s=1$.  $g$ is locally Lipschitz. At a point of differentiability, take $e=\nabla g/|\nabla g|$. Then

		$$|\nabla g|=|\partial_eg|\le\|\partial_eu(x,\cdot)\|_{L^2_y}\le\||\nabla_xu(x,\cdot)|\|_{L^2_y},$$
		
		and so $\|\nabla g\|_2\le\|\nabla_xu\|_2$.
		
		To successfully bound the individual pieces resulting from the Littlewood--Paley decomposition, we require an intermediate interpolation inequality. The following lemma achieves this by controlling the $L^q$-norm of a function in terms of its fractional Sobolev norm and a weighted $L^2$-norm, which will subsequently allow us to piece the decomposition back together.
		
		\begin{lemma}\label{lem2.2}
			Let $g\ge0$ be continuous and rapidly decaying on $\mathbb R^n$, with $\|D^sg\|_2<\infty$. Put $\mathsf N=\|D^sg\|_2^2$ and $\mathsf P=\int|x|^{2\gamma}g^2dx$. Then
					$ \|g\|_{L^q}^2\le C\,\mathsf N^{1-\tau}\,\mathsf P^{\tau}.$ 
		\end{lemma}
		
		\begin{proof}[Proof of Lemma]
			If $\mathsf N=0$ or $\mathsf P=0$, then $g=0$. Assume both are positive.
			
			For $R>0$, the Hölder inequality on $\bb_R$ and \eqref{eq3.2} give

			$$\int_{|x|<R}g^2\le|\bb_R|^{1-\frac2q}\|g\|_q^2\le C\,R^{n(1-\frac2q)}\,\mathsf N^{\alpha}\|g\|_2^{2(1-\alpha)}.$$
			
			By the assumptions, $n(1-\frac2q)=2\alpha s$, so this equals $C(R^{2s}\mathsf N)^\alpha(\|g\|_2^2)^{1-\alpha}$. On $\{|x|\ge R\}$ we have $\int g^2\le R^{-2\gamma}\mathsf P$. Hence

			$$\|g\|_2^2\le C(R^{2s}\mathsf N)^\alpha(\|g\|_2^2)^{1-\alpha}+R^{-2\gamma}\mathsf P .$$
			
			By weighted Young, $C X^\alpha Y^{1-\alpha}\le\frac12Y+C'X$. Since $\|g\|_2<\infty$, we can absorb:

			$$\|g\|_2^2\le C'' R^{2s}\mathsf N+2R^{-2\gamma}\mathsf P.$$
			
			\noindent By choose $R^{2s+2\gamma}=\mathsf P/\mathsf N$, we obtain

			\begin{equation}
				\|g\|_2^2\le C\,\mathsf P^{\frac s{s+\gamma}}\mathsf N^{\frac \gamma{s+\gamma}}.\label{eq4.1}
			\end{equation}
			Insert \eqref{eq4.1} into \eqref{eq3.2} squared:

			$$\|g\|_q^2\le C\,\mathsf N^{\alpha}\Big(\mathsf P^{\frac s{s+\gamma}}\mathsf N^{\frac \gamma{s+\gamma}}\Big)^{1-\alpha}=C\,\mathsf N^{\alpha+\frac{\gamma(1-\alpha)}{s+\gamma}}\mathsf P^{\frac{s(1-\alpha)}{s+\gamma}}=C\,\mathsf N^{1-\tau}\mathsf P^{\tau}$$.
			
		\end{proof}

		Step 3. Now, let $A_j=\|D_x^su_j\|_2$ and $B_j=\||x|^\gamma D_y^ru_j\|_2$. We apply   Lemma \ref{lem2.2} to $g_j$.
		From the above considerations, $\mathsf N_j:=\|D^sg_j\|_2^2\le A_j^2$.
		By the Plancherel identity in $y$ and $|\eta|>2^{j-1}$ on $\operatorname{supp}\hat u_j$,
		
		$$\mathsf P_j:=\int|x|^{2\gamma}g_j^2=\||x|^\gamma u_j\|_2^2\le 2^{2r}2^{-2jr}B_j^2.$$
		By \eqref{eq3.1},

		$$\|u_j\|_q^2\le C\,2^{jm(1-\frac2q)}\,(A_j^2)^{1-\tau}\big(2^{-2jr}B_j^2\big)^{\tau}.$$
		
		\noindent
		The power of $2^j$ is $m(1-\frac2q)-2r\tau=\frac{2ms}{Q}-\frac{2ms}{Q}=0$   the definition of $\tau$. So

		\begin{equation}
			\ \|u_j\|_{L^q}^2\le C\,(A_j^2)^{1-\tau}(B_j^2)^{\tau}\  \label{eq5.1}
		\end{equation}
		that is the scale-invariant bound for each piece.

		By Step 1 in $y$ for each fixed $x$, integrating in $x$, and then the Minkowski inequality in $L^{q/2}$ ($q/2\ge1$):

		$$\|u\|_q^2\le C\Big\|\sum_j|u_j|^2\Big\|_{L^{q/2}}\le C\sum_j\|u_j\|_q^2.$$
		
		By \eqref{eq5.1} and the Hölder inequality  in $j$ with exponents $\frac1{1-\tau},\frac1\tau$, we obtain:

		$$\|u\|_q^2\le C\sum_j(A_j^2)^{1-\tau}(B_j^2)^\tau\le C\Big(\sum_jA_j^2\Big)^{1-\tau}\Big(\sum_jB_j^2\Big)^{\tau}.$$
		
		Since $\sum_j\varphi(2^{-j}\eta)^2\le1$, the Plancherel identity gives

		$$\sum_jA_j^2=\frac1{(2\pi)^N}\iint|\xi|^{2s}\sum_j\varphi^2(2^{-j}\eta)\,|\mathcal Fu|^2\le A^2,$$
		
		and likewise $\sum_jB_j^2\le B^2$. Therefore

		$$\|u\|_{L^q(\mathbb R^N)}\le C\,\|D_x^su\|_2^{\,1-\tau}\,\||x|^\gamma D_y^ru\|_2^{\,\tau}\le C\big(\|D_x^su\|_2+\||x|^\gamma D_y^ru\|_2\big).$$

	\end{proof}

	\section{An extension}\label{sec-extension}

	Let $n\ge 1$. We are given exponents $s_1,\dots,s_n>0$ and vectors $b_j=(b_{j1},\dots,b_{jn})\in\R^n$ for $1\le j\le n$, and we write $B=(b_{jk})_{j,k=1}^n$. For $x\in\R^n$ and $1\le j\le n$ put
	\[
	w_j(x):=|x^{b_j}|=\prod_{l=1}^n|x_l|^{b_{jl}} .
	\]
	Our Fourier transform is $\widehat u(\xi)=\int_{\R^n}u(x)e^{-ix\cdot\xi}\,dx$. The operator $D_{x_j}^{s_j}$ is the Fourier multiplier with symbol $|\xi_j|^{s_j}$, that is, $\widehat{D^{s_j}_{x_j}u}(\xi)=|\xi_j|^{s_j}\widehat u(\xi)$. For a function $u$ we set
	\[
	A_j(u):=\big\|\,w_j\,D_{x_j}^{s_j}u\,\big\|_{L^2(\R^n)},\qquad 1\le j\le n .
	\]
	The inequality under study is
	\begin{equation}\label{s2:eq:main}
		\|u\|_{L^Q(\R^n)}\le C\prod_{j=1}^n A_j(u)^{\theta_j},\qquad \theta_1+\dots+\theta_n=1,
	\end{equation}
	with $0<Q<\infty$ and $\theta=(\theta_1,\dots,\theta_n)\in\R^n$, where $C$ does not depend on $u$. We use the abbreviations
	\[
	\delta:=\frac12-\frac1Q,\qquad c_k:=\sum_{j=1}^n\theta_jb_{jk}\quad(1\le k\le n),
	\]
	and the matrix $M:={\rm diag}(s_1,\dots,s_n)-B^{\mathsf T}$, so that $M_{kk}=s_k-b_{kk}$ and $M_{kj}=-b_{jk}$ for $k\neq j$.
	
	\medskip
	\noindent\textbf{Standing assumption (H).} $s_j>0$ and $b_{jk}>-\tfrac12$ for all $j,k$.
	
	\medskip
	Inequality \eqref{s2:eq:main} has to be understood on a class of functions for which both sides make sense. We take the following class, which is stable under all the transformations used below. Let $\Omega:=\{\xi\in\R^n:\ \xi_1\xi_2\cdots\xi_n\neq0\}$ and
	\[
	\mathcal{S}_0:=\{u\in\mathcal{S}(\R^n):\ \widehat u\in C_c^\infty(\Omega)\}.
	\]
	We say that \eqref{s2:eq:main} \emph{holds} if it holds for every $u\in\mathcal{S}_0$.
	
	\begin{remark}
		If \eqref{s2:eq:main} is assumed on any larger class (for instance on all Schwartz functions for which the right-hand side is finite), then it holds in particular on $\mathcal{S}_0$. Hence every necessary condition proved below for $\mathcal{S}_0$ is a necessary condition for \eqref{s2:eq:main} on that larger class as well.
	\end{remark}
	
	For $g\in\mathcal{S}(\R^n)$ and an integer $M_0\ge0$ put $\|g\|_{(M_0)}:=\sup_{x}\prod_{l=1}^n(1+|x_l|)^{M_0}|g(x)|$; this is finite, and it is dominated by a Schwartz seminorm.

	To ensure the operators are well-defined, we first establish a weighted $L^2$ bound for Schwartz functions subject to a decay condition.

	\begin{lemma}[Weighted $L^2$ bound]\label{s2:lem:weighted}
		Let $\beta\in\R^n$ with $\beta_l>-\frac12$ for all $l$, and let $M_0$ be an integer with $M_0>\frac12+\max_l\beta_l$. There is a constant $C_{\beta,M_0}$ such that for all $g\in\mathcal{S}(\R^n)$
		\[
		\Big\|\prod_{l=1}^n|x_l|^{\beta_l}\,g\Big\|_{L^2(\R^n)}\le C_{\beta,M_0}\,\|g\|_{(M_0)} .
		\]
	\end{lemma}
	
	\begin{proof}
		By the definition of $\|g\|_{(M_0)}$ and Tonelli's theorem,
		\[
		\int_{\R^n}\prod_l|x_l|^{2\beta_l}|g(x)|^2dx
		\le
		\|g\|_{(M_0)}^2
		\prod_{l=1}^n
		\int_{\R}|t|^{2\beta_l}(1+|t|)^{-2M_0}\,dt .
		\]
		Each factor is finite, because $2\beta_l>-1$ makes the singularity at $t=0$ integrable and $2M_0-2\beta_l>1$ makes the tail at infinity integrable.
	\end{proof}

	Using this bound, we can verify that the test class $\mathcal{S}_0$ is stable under both dilations and translations, making it a suitable space for our analysis.

	\begin{lemma}[Properties of $\mathcal{S}_0$]\label{s2:lem:class}
		Assume (H) and let $u\in\mathcal{S}_0$.
		\begin{enumerate}
			\item[(a)] $D_{x_j}^{s_j}u\in\mathcal{S}(\R^n)$ for every $j$, and $0\le A_j(u)<\infty$. If $u\ne0$, then $A_j(u)>0$ for every $j$.
			
			\item[(b)] $\mathcal{S}_0$ is a complex vector space, and it is invariant under the dilations
			\[
			u_\lambda(x):=u(\lambda_1x_1,\dots,\lambda_nx_n),
			\qquad
			\lambda\in(0,\infty)^n,
			\]
			and under translations $u(\cdot-a)$, $a\in\R^n$.
		\end{enumerate}
	\end{lemma}
	
	\begin{proof}
		(a) The symbol $|\xi_j|^{s_j}$ is smooth on $\Omega$ and $\widehat u$ has compact support in $\Omega$. Hence $|\xi_j|^{s_j}\widehat u\in C_c^\infty(\Omega)$, and so $D^{s_j}_{x_j}u\in\mathcal{S}(\R^n)$. Finiteness of $A_j(u)$ follows from Lemma~\ref{s2:lem:weighted} with $\beta=b_j$, which is allowed by (H). If $u\neq0$ then $\widehat u\not\equiv0$ on a set where $|\xi_j|^{s_j}>0$, so $D_{x_j}^{s_j}u\ne0$. Since $w_j>0$ almost everywhere, $A_j(u)>0$.
		
		(b) Sums of functions in $C_c^\infty(\Omega)$ stay in $C_c^\infty(\Omega)$. We have
		\[
		\widehat{u_\lambda}(\xi)
		=
		\big(\prod_l\lambda_l\big)^{-1}
		\widehat u(\xi_1/\lambda_1,\dots,\xi_n/\lambda_n),
		\]
		whose support is the image of $\operatorname{supp}\widehat u$ under a diagonal map with positive entries and thus again a compact subset of $\Omega$. Finally, translation by $a$ multiplies $\widehat u$ by the smooth unimodular function $e^{-ia\cdot\xi}$ and does not change the support.
	\end{proof}
	
	\subsection{Scaling and the relation between the parameters}\label{s2:sec:scaling}

	We begin our parameter analysis by recording how the relevant quantities scale under anisotropic dilations.

	\begin{lemma}[Behaviour under anisotropic dilations]\label{s2:lem:dilation}
		Let $u\in\mathcal{S}_0$, $\lambda\in(0,\infty)^n$, $\Lambda:=\lambda_1\cdots\lambda_n$, and $u_\lambda(x)=u(\lambda_1x_1,\dots,\lambda_nx_n)$. Then
		\[
		\|u_\lambda\|_{L^Q}
		=
		\Lambda^{-1/Q}\|u\|_{L^Q},
		\qquad
		A_j(u_\lambda)
		=
		\lambda_j^{s_j}
		\prod_{k=1}^n\lambda_k^{-b_{jk}}
		\Lambda^{-1/2}A_j(u)
		\quad(1\le j\le n).
		\]
	\end{lemma}
	
	\begin{proof}
		The first identity is the change of variables $y=\lambda x$, $dx=\Lambda^{-1}dy$. For the second,
		\[
		\widehat{u_\lambda}(\xi)=\Lambda^{-1}\widehat u(\xi/\lambda),
		\]
		so the multiplier $|\xi_j|^{s_j}$ gives
		$
		D_{x_j}^{s_j}u_\lambda
		=
		\lambda_j^{s_j}\,(D_{x_j}^{s_j}u)(\lambda\,\cdot).
		$
		Moreover
		\[
		w_j(x)
		=
		\prod_k|y_k/\lambda_k|^{b_{jk}}
		=
		\prod_k\lambda_k^{-b_{jk}}\,w_j(y)
		\]
		for $y=\lambda x$. Therefore
		\[
		A_j(u_\lambda)^2
		=
		\lambda_j^{2s_j}
		\prod_k\lambda_k^{-2b_{jk}}
		\int
		w_j(y)^2\,
		|(D_{x_j}^{s_j}u)(y)|^2\,
		\Lambda^{-1}dy ,
		\]
		which is the claim.
	\end{proof}

	By applying the main inequality to these dilated functions, we obtain the fundamental scaling identity connecting the parameters.
	\begin{proposition} \label{s2:prop:scaling}
		Assume (H) and that \eqref{s2:eq:main} holds. Then for every $k=1,\dots,n$
		\begin{equation}\label{s2:eq:scaling}
			\theta_ks_k-\sum_{j=1}^n\theta_jb_{jk}
			=
			\frac12-\frac1Q,
			\qquad\text{i.e.}\qquad
			\theta_ks_k-c_k=\delta .
		\end{equation}
		In matrix form,
		\[
		M\theta=\delta\,\mathbf{1}
		\]
		with $\mathbf{1}=(1,\dots,1)^{\mathsf T}$, and
		\[
		\mathbf{1}^{\mathsf T}\theta=1.
		\]
	\end{proposition}
	
	\begin{proof}
		Fix $k$ and $u\in\mathcal{S}_0\setminus\{0\}$. For $t>0$ apply \eqref{s2:eq:main} to $u_\lambda$ with
	$
		\lambda=(1,\dots,1,t,1,\dots,1)
		$. Then $\Lambda=t$ and, by Lemma~\ref{s2:lem:dilation},
		$
		A_j(u_\lambda)
		=
		t^{s_j\delta_{jk}-b_{jk}-1/2}A_j(u),
	$
		where $\delta_{jk}$ is the Kronecker symbol. Using $\sum_j\theta_j=1$ we get
		\[
		t^{-1/Q}\|u\|_{L^Q}
		\le
		C\,t^{\theta_ks_k-c_k-\frac12}
		\prod_jA_j(u)^{\theta_j},
		\]
		that is,
		\[
		t^{\delta-(\theta_ks_k-c_k)}
		\,\|u\|_{L^Q}
		\le
		C\prod_jA_j(u)^{\theta_j}
		\qquad\text{for all }t>0 .
		\]
		By Lemma~\ref{s2:lem:class}, the right-hand side is a finite constant independent of $t$, and $\|u\|_{L^Q}>0$. If the exponent of $t$ were positive, letting $t\to\infty$ would give a contradiction; if it were negative, letting $t\to0^+$ would. Hence the exponent vanishes, which is \eqref{s2:eq:scaling}.
	\end{proof}

	Whenever the associated matrix $M$ is invertible, this scaling identity uniquely determines both the exponent $Q$ and the vector $\theta$.

	\begin{corollary} \label{s2:cor:Q}
		Assume (H) and \eqref{s2:eq:main}, and suppose that $M$ is invertible. Put
		\[
		z:=M^{-1}\mathbf{1}.
		\]
		Then $\mathbf{1}^{\mathsf T}z\neq0$ and
		\[
		\frac1Q
		=
		\frac12-\frac1{\mathbf{1}^{\mathsf T}M^{-1}\mathbf{1}},
		\qquad
		\theta
		=
		\frac{M^{-1}\mathbf{1}}{\mathbf{1}^{\mathsf T}M^{-1}\mathbf{1}},
		\qquad
		\text{i.e. }
		\theta_k=\frac{z_k}{z_1+\dots+z_n}.
		\]
	\end{corollary}
	
	\begin{proof}
		By \eqref{s2:eq:scaling},
	$
		\theta=\delta M^{-1}\mathbf{1}=\delta z.
		$
		Then
		$
		1=\mathbf{1}^{\mathsf T}\theta
		=
		\delta\,\mathbf{1}^{\mathsf T}z,
		$
		which forces $\mathbf{1}^{\mathsf T}z\ne0$ and
	$
		\delta=\frac1{\mathbf{1}^{\mathsf T}z}.
		$
	\end{proof}

	\begin{remark}[Special cases]
		\emph{(1) No weights} ($B=0$). Then $M={\rm diag}(s_k)$ and
		\[
		\frac1Q=\frac12-\big(\sum_k s_k^{-1}\big)^{-1},
		\qquad
		\theta_k=s_k^{-1}\big/\sum_ls_l^{-1},
		\]
		the anisotropic Sobolev embedding.
		
		\emph{(2) Diagonal weights} ($b_{jk}=b_j\delta_{jk}$, so $w_j=|x_j|^{b_j}$). Then $M={\rm diag}(s_k-b_k)$ and, if $s_k\ne b_k$ for all $k$,
		\[
		\theta_k=\frac{(s_k-b_k)^{-1}}{\sum_l(s_l-b_l)^{-1}},
		\qquad
		\frac1Q=\frac12-\Big(\sum_{l}(s_l-b_l)^{-1}\Big)^{-1}.
		\]
		
		\emph{(3) The same weight in all factors} ($b_j=b$ for all $j$). Then $c_k=b_k$ and \eqref{s2:eq:scaling} reads
		$
		\theta_ks_k=\delta+b_k.
		$
		Summing $\theta_k=(\delta+b_k)/s_k$ over $k$ gives
		\[
		\delta=\frac{1-\sum_kb_k/s_k}{\sum_k1/s_k}.
		\]
		
		\emph{(4) $n=1$.} Then $\theta_1=1$ and $\frac1Q=\frac12-s_1+b_1$.
	\end{remark}
	
	\subsection{Further necessary conditions.}\label{s2:sec:necessary}
	
Next, we demonstrate that the interpolation exponents $\theta_k$ must be non-negative by shifting the frequency support of our test functions.
	
	\begin{lemma}\label{s2:lem:theta}
		Assume (H) and \eqref{s2:eq:main}. Then $\theta_k\ge0$ for every $k$.
	\end{lemma}
	
	\begin{proof}
		Fix $k$ and choose $u_0\in\mathcal{S}_0\setminus\{0\}$ with
		$
		\operatorname{supp}\widehat u_0\subset(1,2)^n.
		$
		For $\tau>0$ let
	$
		u_\tau(x):=e^{i\tau x_k}u_0(x).
	$
		Then
		$
		\widehat{u_\tau}(\xi)=\widehat u_0(\xi-\tau e_k)
	$
		is supported in
	$
		\{\xi\in(1,2)^n+\tau e_k\}\subset\Omega,
		$
		so $u_\tau\in\mathcal{S}_0$, and $|u_\tau|=|u_0|$. In particular
	$
		\|u_\tau\|_{L^Q}=\|u_0\|_{L^Q}>0.
	$
				Let $j\neq k$. The multiplier $|\xi_j|^{s_j}$ does not see the shift in the $\xi_k$ variable, so
	$
		D_{x_j}^{s_j}u_\tau=e^{i\tau x_k}D_{x_j}^{s_j}u_0,
	$
		and
	$
		A_j(u_\tau)=A_j(u_0)\in(0,\infty).
	$
				For $j=k$ we have
	$
		D_{x_k}^{s_k}u_\tau=e^{i\tau x_k}h_\tau
	$
		with
		$
		\widehat{h_\tau}(\xi)=(\tau+\xi_k)^{s_k}\widehat u_0(\xi),
	$
		hence
		\[
		\tau^{-s_k}\widehat{h_\tau}(\xi)
		=
		(1+\xi_k/\tau)^{s_k}\widehat u_0(\xi)
		\longrightarrow
		\widehat u_0(\xi)
		\qquad
		(\tau\to\infty)
		\]
		in $C^\infty$, with supports in a fixed compact subset of $\Omega$. Consequently
	$
		\tau^{-s_k}h_\tau\to u_0
	$
		in $\mathcal{S}(\R^n)$, and in particular
	$
		\|\tau^{-s_k}h_\tau-u_0\|_{(M_0)}\to0.
	$
		By Lemma~\ref{s2:lem:weighted},
		\[
		\tau^{-s_k}A_k(u_\tau)
		=
		\|w_k\,\tau^{-s_k}h_\tau\|_{L^2}
		\longrightarrow
		\|w_ku_0\|_{L^2}
		=:a_*
		\in(0,\infty).
		\]
				Suppose now $\theta_k<0$. Applying \eqref{s2:eq:main} to $u_\tau$ gives
	$
		0<\|u_0\|_{L^Q}
		\le
		C\prod_{j\ne k}A_j(u_0)^{\theta_j}\cdot A_k(u_\tau)^{\theta_k},
		$
		and the right-hand side behaves like a constant times
	$
		\tau^{s_k\theta_k}a_*^{\theta_k}\to0,
	$
		a contradiction. Hence $\theta_k\ge0$.
	\end{proof}
	
	\subsection{Translation to infinity}
	
	To extract further constraints on the parameters, we require a uniform estimate for how the weighted norms behave under spatial translations.
	
	\begin{lemma} \label{s2:lem:translation}
		Assume (H). Let $u_0\in\mathcal{S}_0$, $1\le j,k\le n$, and
		$
		u_R:=u_0(\cdot-Re_k)
	$
		for $R\ge1$. Then there is $C_*$, independent of $R$, with
	$
		A_j(u_R)\le C_*\,R^{\,b_{jk}} .
	$
	\end{lemma}
	
	\begin{proof}
		Write
		\[
		g:=D_{x_j}^{s_j}u_0\in\mathcal{S}(\R^n),
		\qquad
		\beta:=b_{jk},
		\qquad
		\omega(x'):=\prod_{l\ne k}|x_l|^{2b_{jl}}.
		\]
		Since $D_{x_j}^{s_j}$ commutes with translations, a change of variables gives
		\[
		A_j(u_R)^2
		=
		\int_{\R^n}|x_k+R|^{2\beta}\,\omega(x')\,|g(x)|^2\,dx .
		\]
				We split the integral according to whether $|x_k|\le R/2$ or $|x_k|>R/2$.
		
		On $\{|x_k|\le R/2\}$ we have
		\[
		\frac R2\le|x_k+R|\le\frac{3R}2,
		\]
		hence
	$
		|x_k+R|^{2\beta}\le 2^{2|\beta|}R^{2\beta}.
	$
		The contribution is at most
	$
		2^{2|\beta|}R^{2\beta}\int\omega|g|^2,
	$
		and
	$
		\int\omega|g|^2<\infty
	$
		by Lemma~\ref{s2:lem:weighted}.
				On $\{|x_k|>R/2\}$ we distinguish two cases. If $\beta\ge0$, then
		$
		|x_k+R|\le3|x_k|,
	$
		so
		\[
		|x_k+R|^{2\beta}\le3^{2\beta}|x_k|^{2\beta}.
		\]
		The contribution is at most
		$
		3^{2\beta}\int|x_k|^{2\beta}\omega|g|^2,
		$
		a finite constant by Lemma~\ref{s2:lem:weighted}, and it is $\le C R^{2\beta}$ because $R\ge1$.
		
		If $\beta<0$, we use
		\[
		|g(x)|
		\le
		\|g\|_{(M_0)}
		\prod_l(1+|x_l|)^{-M_0}
		\]
		with $M_0\ge 2|\beta|$ large. Then the contribution is bounded by
		\[
		C\int_{|x_k|>R/2}
		|x_k+R|^{-2|\beta|}
		(1+|x_k|)^{-2M_0}
		\,dx_k
		\]
		\[
		\le
		C(1+R/2)^{-M_0}
		\int_{\R}
		|t+R|^{-2|\beta|}
		(1+|t|)^{-M_0}
		\,dt ,
		\]
		where we also integrated the remaining factors
		$
		\prod_{l\ne k}
		|x_l|^{2b_{jl}}
		(1+|x_l|)^{-2M_0},
	$
		which is finite.
		
		The last $t$-integral is bounded uniformly in $R$: the region $|t+R|\le1$ contributes at most
		$
		\int_{|s|\le1}|s|^{-2|\beta|}ds<\infty
	$
		because $2|\beta|<1$, and the region $|t+R|>1$ contributes at most
		$
		\int(1+|t|)^{-M_0}dt .
	$
		Since
		\[
		(1+R/2)^{-M_0}
		\le
		2^{M_0}R^{-M_0}
		\le
		2^{M_0}R^{2\beta}
		\]
		for $M_0\ge2|\beta|$ and $R\ge1$, the contribution is $\le CR^{2\beta}$.
		
		In all cases
	$
		A_j(u_R)^2\le C_*^2R^{2\beta}.
	$
	\end{proof}

	Testing the main inequality against these translated functions immediately yields a lower bound for the scaling variables.

	\begin{proposition} \label{s2:prop:translation}
		Assume (H) and \eqref{s2:eq:main}. Then
		$
		c_k=\sum_j\theta_jb_{jk}\ge0
	$
		for every $k$. By \eqref{s2:eq:scaling} this is equivalent to
		\[
		\theta_ks_k\ge\delta=\frac12-\frac1Q.
		\]
	\end{proposition}
	
	\begin{proof}
		By Lemma~\ref{s2:lem:theta}, $\theta_j\ge0$ for all $j$. Fix $k$ and $u_0\in\mathcal{S}_0\setminus\{0\}$, and let
	$
		u_R=u_0(\cdot-Re_k)\in\mathcal{S}_0.
	$
		Translation preserves the $L^Q$ norm, so by \eqref{s2:eq:main} and Lemma~\ref{s2:lem:translation}, for $R\ge1$,
		\[
		0<\|u_0\|_{L^Q}
		=
		\|u_R\|_{L^Q}
		\le
		C\prod_jA_j(u_R)^{\theta_j}
		\le
		C\prod_j(C_*R^{b_{jk}})^{\theta_j}
		=
		C'R^{c_k}.
		\]
		If $c_k<0$, the right-hand side tends to $0$ as $R\to\infty$, a contradiction. The equivalence with $\theta_ks_k\ge\delta$ is \eqref{s2:eq:scaling}.
	\end{proof}
	
	\begin{remark}
		Concentrating a bump of width $\varepsilon\to0$ at a point $x^0$ with all coordinates different from zero, or translating in a direction $a\in\R^n$ with several non-zero components, gives no new information. After undoing the dilation by Lemma~\ref{s2:lem:dilation} and using \eqref{s2:eq:scaling}, these tests reduce to inequalities of the form
		\[
		\prod_{k\in S}(\cdot)^{c_k}\gtrsim 1
		\]
		with a large quantity in the base, and these follow from $c_k\ge0$ for every $k$.
	\end{remark}
	
	\subsection{A hidden family of symmetries and the bound $Q\ge2$}
	
	To establish the restriction on the exponent $Q$, we must first identify a specific family of dilations that leaves both the left- and right-hand sides of the inequality completely invariant.

	\begin{lemma} \label{s2:lem:nu}
		Assume (H) and \eqref{s2:eq:scaling}. Then there is
		$
		\nu\in\R^n\setminus\{0\}
		$
		such that, with
		\[
		\lambda=(\mu^{\nu_1},\dots,\mu^{\nu_n}),
		\qquad
		\Lambda=\mu^{\nu_1+\dots+\nu_n},
		\qquad
		\mu>0,
		\]
		the functions $\Lambda^{1/Q}u(\lambda\,\cdot)$ satisfy
		$
		\big\|\Lambda^{1/Q}u(\lambda\,\cdot)\big\|_{L^Q}
		=
		\|u\|_{L^Q},
	$
		\[
		A_j\big(\Lambda^{1/Q}u(\lambda\,\cdot)\big)
		=
		A_j(u)
		\quad
		\text{for all }j
		\text{ and all }\mu>0 .
		\]
	\end{lemma}
	
	\begin{proof}
		By Lemma~\ref{s2:lem:dilation}, the factor in front of $A_j(u)$ is
		\[
		\Lambda^{1/Q}
		\lambda_j^{s_j}
		\prod_k\lambda_k^{-b_{jk}}
		\Lambda^{-1/2}
		=
		\mu^{(P\nu)_j},
		\qquad
		P:=M^{\mathsf T}-\delta\,\mathbf{1}\mathbf{1}^{\mathsf T},
		\]
		where we used $\sum_k\nu_k=\mathbf{1}^{\mathsf T}\nu$ and $M^{\mathsf T}={\rm diag}(s)-B$, so that
		\[
		(M^{\mathsf T}\nu)_j
		=
		s_j\nu_j-\sum_kb_{jk}\nu_k.
		\]
		Thus it suffices to find $\nu\ne0$ with $P\nu=0$. Now by \eqref{s2:eq:scaling} and $\mathbf{1}^{\mathsf T}\theta=1$,
		\[
		\theta^{\mathsf T}P
		=
		(M\theta)^{\mathsf T}
		-
		\delta(\theta^{\mathsf T}\mathbf{1})\mathbf{1}^{\mathsf T}
		=
		\delta\mathbf{1}^{\mathsf T}
		-
		\delta\mathbf{1}^{\mathsf T}
		=
		0.
		\]
		Since $\theta\ne0$, the matrix $P$ is singular, so a non-zero $\nu\in\ker P$ exists. The identity for the $L^Q$ norm is the first identity of Lemma~\ref{s2:lem:dilation}.
	\end{proof}
	
	We shall use Khintchine's inequality in the following form. Let $\varepsilon_1,\dots,\varepsilon_L$ be independent random signs. For $0<Q<2$ there is $c_Q>0$ such that for all complex numbers $a_1,\dots,a_L$
	\begin{equation}\label{s2:eq:khintchine}
		\mathbb{E}\Big|\sum_{m=1}^L\varepsilon_ma_m\Big|^Q
		\ge
		c_Q\Big(\sum_{m=1}^L|a_m|^2\Big)^{Q/2}.
	\end{equation}
	For real $a_m$ this is the classical Khintchine inequality. For complex $a_m$ it follows by applying the real case to the real and imaginary parts:
	\[
	|z|^Q
	\ge
	\frac12
	\big(
	|\operatorname{Re}z|^Q
	+
	|\operatorname{Im}z|^Q
	\big),
	\]
	and
$
	(x+y)^{Q/2}\le x^{Q/2}+y^{Q/2}
	$
	since $Q/2\le1$.

	By randomizing over these invariant dilations and applying Khintchine's inequality, we deduce that the target exponent $Q$ cannot fall below 2.
	\begin{proposition}[$Q\ge2$]\label{s2:prop:Q2}
		Assume (H) and \eqref{s2:eq:main}, with $0<Q<\infty$. Then $Q\ge2$.
	\end{proposition}
	
	\begin{proof}
		By Lemma~\ref{s2:lem:theta} and Proposition~\ref{s2:prop:scaling}, $\theta_j\ge0$ and \eqref{s2:eq:scaling} holds. Suppose, to get a contradiction, that $Q<2$. Take $\nu\ne0$ from Lemma~\ref{s2:lem:nu}, fix an index $k$ with $\nu_k\ne0$, and fix $u_0\in\mathcal{S}_0\setminus\{0\}$.
		
		\emph{Choice of the pieces.} As $r_0\downarrow0$ and $r_1\uparrow\infty$, the integrals
		\[
		\int_{\{r_0\le|y_k|\le r_1\}}|u_0|^Q\,dy
		\]
		increase to
		\[
		\int_{\{y_k\ne0\}}|u_0|^Qdy
		=
		\|u_0\|_{L^Q}^Q
		>0.
		\]
		Hence we can fix $0<r_0<r_1<\infty$ such that, with
		\[
		E_0:=\{y:\ r_0\le|y_k|\le r_1\},
		\]
		$
		\eta:=\int_{E_0}|u_0|^Q\,dy>0 .
		$
				Choose an integer $K\ge1$ with
	$
		2^{K|\nu_k|}>r_1/r_0.
	$
		For $m=1,\dots,L$ let
		\[
		\lambda^{(m)}
		:=
		(2^{Km\nu_1},\dots,2^{Km\nu_n}),
		\qquad
		\Lambda_m
		:=
		\prod_l\lambda^{(m)}_l,
		\]
		and
		$
		u_m(x)
		:=
		\Lambda_m^{1/Q}\,u_0(\lambda^{(m)}x)
		\in\mathcal{S}_0.
		$
		By Lemma~\ref{s2:lem:nu} (with $\mu=2^{Km}$),
		$
		\|u_m\|_{L^Q}
		=
		\|u_0\|_{L^Q}
		$
		and
	$
		A_j(u_m)=A_j(u_0)
	$
		for all $j$ and $m$.
		
		\emph{Disjointness.} Let
		$
		E_m:=\{x:\ \lambda^{(m)}x\in E_0\}.
		$
		If $x\in E_m\cap E_{m'}$ with $m\ne m'$, then
		\[
		2^{Km\nu_k}|x_k|
		\quad\text{and}\quad
		2^{Km'\nu_k}|x_k|
		\]
		both lie in $[r_0,r_1]$, so
		$
		2^{K|m-m'||\nu_k|}
		\le
		r_1/r_0,
	$
		which contradicts the choice of $K$. Thus the sets
	$
		E_1,\dots,E_L
	$
		are pairwise disjoint. By the change of variables $y=\lambda^{(m)}x$,
		\[
		\int_{E_m}|u_m|^Q\,dx
		=
		\Lambda_m
		\int_{E_m}|u_0(\lambda^{(m)}x)|^Q\,dx
		=
		\int_{E_0}|u_0|^Q\,dy
		=
		\eta .
		\]
		
		\emph{Random sums.} For a sign vector
		\[
		\varepsilon\in\{\pm1\}^L
		\]
		put
	$
		u_\varepsilon:=\sum_m\varepsilon_mu_m\in\mathcal{S}_0 .
		$
		Applying \eqref{s2:eq:main}, squaring, and averaging over $\varepsilon$ (uniform probability), we obtain
		\[
		\mathbb{E}\|u_\varepsilon\|_{L^Q}^2
		\le
		C^2\,\mathbb{E}\prod_jA_j(u_\varepsilon)^{2\theta_j}
		\le
		C^2\prod_j\big(\mathbb{E} A_j(u_\varepsilon)^2\big)^{\theta_j},
		\]
		where the second inequality is Hölder's inequality for the exponents $1/\theta_j$ (recall $\theta_j\ge0$ and $\sum\theta_j=1$; factors with $\theta_j=0$ are equal to $1$).
		
		Since
		$
		\mathbb{E}\,\varepsilon_m\varepsilon_{m'}
		=
		\delta_{mm'}
		$
		and $D_{x_j}^{s_j}$ is linear, Tonelli's theorem gives
		\[
		\mathbb{E} A_j(u_\varepsilon)^2
		=
		\int w_j^2\,
		\mathbb{E}
		\Big|
		\sum_m\varepsilon_mD_{x_j}^{s_j}u_m
		\Big|^2
		dx
		=
		\sum_{m=1}^LA_j(u_m)^2
		=
		L\,A_j(u_0)^2 .
		\]
		Therefore,
		\begin{equation}\label{s2:eq:upper}
			\mathbb{E}\|u_\varepsilon\|_{L^Q}^2
			\le
			C^2\,L\prod_jA_j(u_0)^{2\theta_j}.
		\end{equation}
		
		\emph{Lower bound.} Since $Q<2$, the function $t\mapsto t^{2/Q}$ is convex, so Jensen's inequality gives
		\[
		\mathbb{E}\|u_\varepsilon\|_{L^Q}^2
		=
		\mathbb{E}\big(\|u_\varepsilon\|_{L^Q}^Q\big)^{2/Q}
		\ge
		\big(\mathbb{E}\|u_\varepsilon\|^Q_{L^Q}\big)^{2/Q}.
		\]
		By  the Tonelli theorem and \eqref{s2:eq:khintchine} applied pointwise in $x$,
		\[
		\mathbb{E}\|u_\varepsilon\|_{L^Q}^Q
		=
		\int
		\mathbb{E}
		\Big|
		\sum_m\varepsilon_mu_m(x)
		\Big|^Q
		dx
		\ge
		c_Q
		\int
		\Big(\sum_m|u_m|^2\Big)^{Q/2}
		dx
		\]
		\[
		\ge
		c_Q
		\sum_{m=1}^L
		\int_{E_m}|u_m|^Q\,dx
		=
		c_Q\,L\,\eta ,
		\]
		where for the last inequality we used that the $E_m$ are disjoint and that
		\[
		\Big(\sum_{m'}|u_{m'}|^2\Big)^{Q/2}
		\ge
		|u_m|^Q
		\]
		on $E_m$. Combining this with \eqref{s2:eq:upper},
		\[
		(c_Q\,\eta\,L)^{2/Q}
		\le
		C^2L\prod_jA_j(u_0)^{2\theta_j}
		\qquad
		\text{for all }L\ge1 .
		\]
		Since $2/Q>1$, the left-hand side grows faster than the right-hand side as $L\to\infty$, a contradiction. Hence $Q\ge2$.
	\end{proof}
	
	\subsection{The main result}

	We are now ready to collect these disparate observations into our main structural theorem outlining all necessary conditions for the inequality to hold.
	\begin{theorem}[Necessary conditions]\label{s2:thm:main}
		Assume (H), let $1<Q<\infty$ and $\theta\in\R^n$ with $\theta_1+\dots+\theta_n=1$, and suppose that \eqref{s2:eq:main} holds for all $u\in\mathcal{S}_0$. Then the following hold.
		\begin{enumerate}
			\item[(i)] $\theta_k\ge0$ for every $k$.
			\item[(ii)] For every $k$,
			\[
			\theta_ks_k-\sum_{j=1}^n\theta_jb_{jk}
			=
			\frac12-\frac1Q;
			\]
			equivalently
		$
			M\theta=\delta\mathbf{1}.
			$
			
			\item[(iii)] $2\le Q<\infty$, that is,
		$
			0\le\delta<\frac12.
			$
			
			\item[(iv)] $c_k=\sum_{j}\theta_jb_{jk}\ge0$ for every $k$; equivalently
			$
			\theta_ks_k\ge\delta.
		$
		\end{enumerate}
		
		In particular,
		\begin{equation}\label{s2:eq:range}
			0\le c_k=\theta_ks_k-\delta\le\theta_ks_k
			\qquad
			(1\le k\le n).
		\end{equation}
	\end{theorem}
	
	\begin{proof}
		Part (i) is Lemma~\ref{s2:lem:theta}, part (ii) is Proposition~\ref{s2:prop:scaling}, part (iii) is Proposition~\ref{s2:prop:Q2} together with $Q<\infty$, and part (iv) is Proposition~\ref{s2:prop:translation}. Relation \eqref{s2:eq:range} follows from (ii), (iv) and $\delta\ge0$.
	\end{proof}

	 An immediate consequence of these conditions is that any vanishing exponent $\theta_k$ forces the inequality into the $Q=2$ regime.
	\begin{corollary}[Degenerate exponents]\label{s2:cor:degenerate}
		Under the assumptions of Theorem~\ref{s2:thm:main}, if $\theta_k=0$ for some $k$, then $Q=2$ and $c_k=0$. Equivalently, if $Q>2$, then $\theta_k>0$ for every $k$ and $0\le c_k<\theta_ks_k$.
	\end{corollary}
	
	\begin{proof}
		If $\theta_k=0$, then (ii) gives $c_k=-\delta$, while (iii) and (iv) give $\delta\ge0$ and $c_k\ge0$. Hence $\delta=0$, i.e.\ $Q=2$, and $c_k=0$. The second statement is the contrapositive, together with
	$
		c_k=\theta_ks_k-\delta<\theta_ks_k
		$
		for $\delta>0$.
	\end{proof}

	In the non-degenerate setting where $M$ is invertible, the necessary constraints can be explicitly rewritten in terms of the matrix parameters.

	\begin{corollary} \label{s2:cor:invertible}
		Assume in addition that $M$ is invertible and put
		\[
		z=M^{-1}\mathbf{1}\in\R^n.
		\]
		Then \eqref{s2:eq:main} can hold only if
		\begin{equation}\label{s2:eq:z}
			z_1+\dots+z_n>2
			\qquad\text{and}\qquad
			s_kz_k\ge1
			\quad
			(1\le k\le n),
		\end{equation}
		and in this case necessarily
		\[
		\frac1Q
		=
		\frac12-\frac1{z_1+\dots+z_n},
		\qquad
		\theta_k
		=
		\frac{z_k}{z_1+\dots+z_n}.
		\]
		In particular $Q>2$.
	\end{corollary}
	
	\begin{proof}
		By Corollary~\ref{s2:cor:Q},
		\[
		\sigma:=\mathbf{1}^{\mathsf T}z\ne0,
		\qquad
		\delta=1/\sigma,
		\qquad
		\theta=z/\sigma.
		\]
		Since $\delta\ge0$ by Theorem~\ref{s2:thm:main}(iii), in fact $\delta>0$ and $\sigma>0$; moreover
	$
		\delta<\frac12
	$
		gives
	$
		\sigma>2.
	$
		Now Theorem~\ref{s2:thm:main}(iv) states
	$
		\theta_ks_k\ge\delta,
	$
		i.e.
	$
		s_kz_k/\sigma\ge1/\sigma,
	$
		which is
	$
		s_kz_k\ge1
	$
		because $\sigma>0$. Then $z_k>0$ since $s_k>0$, so (i) holds automatically.
	\end{proof}

	These restrictions naturally simplify when the inequality involves purely diagonal weights.
	
	\begin{corollary} \label{s2:cor:diagonal}
		Let $w_j=|x_j|^{b_j}$, i.e.\ $b_{jk}=b_j\delta_{jk}$, and assume $s_k\ne b_k$ for all $k$. If \eqref{s2:eq:main} holds, then
		\[
		0\le b_k<s_k
		\quad
		(1\le k\le n),
		\qquad
		\sum_{k=1}^n\frac1{s_k-b_k}>2 ,
		\]
		and necessarily
		\[
		\theta_k=
		\dfrac{(s_k-b_k)^{-1}}
		{\sum_l(s_l-b_l)^{-1}},
		\qquad
		\dfrac1Q
		=
		\dfrac12-
		\Big(\sum_l(s_l-b_l)^{-1}\Big)^{-1}.
		\]
	\end{corollary}
	
	\begin{proof}
		Here
		 		$M={\rm diag}(s_k-b_k)$
	 $
		z_k=(s_k-b_k)^{-1}\ne0.
	$
		By \eqref{s2:eq:z},
	$
		\sum_kz_k>2
	$
		and
	$
		s_kz_k\ge1.
	$
		Since $s_k>0$, the second condition forces $z_k>0$, hence $s_k>b_k$. Finally
		$
		s_kz_k\ge1
	$
		reads
	$
		s_k\ge s_k-b_k,
	$
		that is,
	$
		b_k\ge0.
	$
	\end{proof}
	
	Finally, specializing the main theorem to one dimension recovers the exact parameter ranges required for the classical single-variable inequality.
	\begin{corollary}[The case $n=1$]\label{s2:cor:one}
		Let $n=1$, $s>0$, $b>-\frac12$ and
		$
		\frac1Q=\frac12-s+b.
		$
		If
		\[
		\|u\|_{L^Q}\le C\|\,|x|^bD^su\|_{L^2}
		\]
		on $\mathcal{S}_0$, then
		$
		b\ge0,
		$ $
		b\le s,
		$ $
		s-b<\tfrac12 ,
	$
		that is,
		$
		\max\{0,\,s-\tfrac12\}<b\le s
		$
		or
		$
		b=0\le s<\tfrac12.
		$
		The endpoint $b=s$ corresponds to $Q=2$.
	\end{corollary}
	
	\begin{proof}
		Here
		$
		\theta_1=1,
		$ $
		M=s-b,
		$ $
		\delta=s-b
		$
		and
		$
		c_1=b.
		$
		Theorem~\ref{s2:thm:main}(iii) gives
		$
		0\le s-b<\frac12,
		$
		and (iv) gives
		$
		b\ge0.
		$
	\end{proof}
	
	\begin{remark} 
		Theorem~\ref{s2:thm:main} lists conditions that every valid inequality of the form \eqref{s2:eq:main} must satisfy. We do not claim that they are sufficient. Establishing the inequality for parameters satisfying them requires a separate argument (weighted interpolation, Caffarelli--Kohn--Nirenberg or Pitt-type estimates in each variable), and, depending on the class of functions on which the inequality is posed, further restrictions near the coordinate hyperplanes
		 $\{x_k=0\}$
		may appear. The cases $Q=2$ correspond to
	 		$\det M=0$
	 		with
	$M\theta=0,$
		i.e.\ to Hardy--type (uncertainty-principle) inequalities.
	\end{remark}

	\section{Compact embedding}\label{sec-compact}
	
	Let $n\ge2$, $m\ge1$, $s,r>0$, $\gamma\ge0$, and
	$$
	X=\left\{
	u:\|u\|_X<\infty
	\right\},
	\qquad
	\|u\|_X^2
	=
	\|u\|_2^2+\|D_x^su\|_2^2+\||x|^\gamma D_y^ru\|_2^2,
	$$
	is a Hilbert space. 	
We say	$u$ is \emph{radial in $x$} if
	$$
	u(Qx,y)=u(x,y)
	\quad\text{a.e.}
	$$
	for every $Q\in O(n)$, and \emph{radial in $y$} similarly. We denote the corresponding closed subspaces by
	$X_{\rm rad,x}$ and $X_{\rm rad,x,y}$.
	
	\medskip
	
	\noindent
	\textbf{Hypothesis (H).}
	There are
	$
	Q_0=n+\frac{m(s+\gamma)}{r}>2s
	$
	and
	$
	B=\frac{2Q_0}{Q_0-2s}
	$
	with
	$
	\|v\|_{L^{B}(\mathbb R^{n+m})}
	\le C\|v\|_X
	$
	for all $v\in X$.
	
	Before proceeding to the main compactness results, we first collect some essential properties of the fractional Sobolev norm in the $x$-variable, namely multiplier estimates, translation bounds, and standard embeddings.
	
	\begin{lemma}
		\label{lemma0}    
		
		Let
		$
		\|f\|_{H^s_x}^2
		=
		\iint(1+|\xi|^2)^s|\mathcal F_xf(\xi,y)|^2\,d\xi\,dy.
		$
		Then
		$
		\|f\|_{H^s_x}^2
		\simeq
		\|f\|_2^2+\|D_x^sf\|_2^2.
		$
		
		(a) Multipliers.		
		If $\eta_1\in\mathcal S(\mathbb R^n)$, then
		$
		\|\eta_1f\|_{H^s_x}
		\le C_{\eta_1,s}\|f\|_{H^s_x}.
		$
		
		(b) Translations.		
		Let
		$
		\sigma=\min\{s,1\},
		$ $
		\sigma'=\min\{r,1\}.
		$
		Then
		$
		\|\tau_{(h,0)}f-f\|_2
		\le
		2|h|^\sigma\|f\|_{H^s_x},
		$
		and
		$
		\|\tau_{(0,k)}f-f\|_2^2
		\le
		4|k|^{2\sigma'}
		\left(
		\|f\|_2^2
		+
		\iint|\eta|^{2r}
		|\mathcal F_yf|^2
		\right).
		$
		
		(c) Sobolev in $x$.		
		There are $q>2$ and $C$ such that
		$
		\|f\|_{L^q(\mathbb R^n)}
		\le
		C
		\left(
		\|f\|_2^2+\|D_x^sf\|_2^2
		\right)^{1/2}
		$
		for $f\in H^s(\mathbb R^n)$.
	\end{lemma}
	
	\begin{proof}
		
		For (a), by Peetre's inequality,
		$$
		1+|\xi|^2
		\le
		2(1+|\xi-\zeta|^2)(1+|\zeta|^2),
		$$
		and hence
		$$
		(1+|\xi|^2)^{s/2}
		\le
		2^{s/2}
		(1+|\xi-\zeta|^2)^{s/2}
		(1+|\zeta|^2)^{s/2}.
		$$
		Since
		$ 
		\mathcal F_x(\eta_1f)
		=
		(2\pi)^{-n/2}\widehat{\eta_1}*\mathcal F_xf,
		$  
		Young's inequality gives
		$$
		\left\|
		(1+|\cdot|^2)^{s/2}
		\mathcal F_x(\eta_1f)
		\right\|_2
		\le
		C
		\left\|
		(1+|\cdot|^2)^{s/2}\widehat{\eta_1}
		\right\|_1
		\left\|
		(1+|\cdot|^2)^{s/2}\mathcal F_xf
		\right\|_2.
		$$
		Applying this for a.e. $y$ and integrating in $y$ proves (a).
		
		For (b), we use
		$$
		|e^{i\vartheta}-1|
		\le
		\min\{2,|\vartheta|\}
		\le
		2^{1-\sigma}|\vartheta|^\sigma.
		$$
		Consequently,
		$$
		|e^{ih\cdot\xi}-1|^2
		\le
		4|h|^{2\sigma}|\xi|^{2\sigma}
		\le
		4|h|^{2\sigma}(1+|\xi|^{2s}),
		$$
		because $\sigma\le s$ whenever $s<1$, while for $s\ge1$,
		$|\xi|^2\le1+|\xi|^{2s}$.
		the Plancherel identity in $x$ gives the first estimate. The $y$-case is identical with $r$ in place of $s$.
		
		For (c), let
		$
		q''=\frac{2q}{q-2}.
		$
		Choose $q>2$ sufficiently close to $2$ that
		$
		q''>\frac ns.
		$
		Since
		$
		\frac1{q'}=\frac12+\frac1{q''},
		$
		The Hausdorff--Young and the Hölder inequalities give
		$$
		\|f\|_q
		\le
		\|\widehat f\|_{q'}
		\le
		\left\|
		(1+|\xi|^{2s})^{1/2}\widehat f
		\right\|_2
		\left\|
		(1+|\xi|^{2s})^{-1/2}
		\right\|_{L^{q''}}.
		$$
		The last factor is finite because $sq''>n$. Therefore
		$$
		\|f\|_q
		\le
		C
		\left(
		\|f\|_2^2+\|D_x^sf\|_2^2
		\right)^{1/2}.
		$$
		This proves the lemma.
	\end{proof}
	
	A central challenge in proving compactness is ruling out the concentration of mass near the origin. The following lemma achieves this by providing uniform $L^2$ and $L^p$ control on thin cylinders around $x=0$, which holds independently of any radiality assumptions.
	\begin{lemma}\label{lemma1}

		Let
		$
		q,$ $
		\alpha=n\left(\frac12-\frac1q\right)>0
		$
		be as in Lemma \ref{lemma0}(c). For all $\delta>0$ and $u\in X$,
		with no radiality assumption and with $y$ ranging over all of
		$\mathbb R^m$,
		$$
		\|u\|_{L^2(\bb_\delta^n\times\mathbb R^m)}
		\le
		C\delta^\alpha\|u\|_X.
		$$
		Consequently, for $2<p<B$ there are $\beta>0$ and $C$ such that
		$$
		\|u\|_{L^p(\bb_\delta^n\times\mathbb R^m)}
		\le
		C\delta^\beta\|u\|_X.
		$$
	\end{lemma}
	\begin{proof}
		
		For a.e. $y$, the  Hölder inequality in $x$ gives
		$$
		\int_{\bb_\delta^n}|u(x,y)|^2\,dx
		\le
		|\bb_\delta^n|^{1-2/q}
		\|u(\cdot,y)\|_{L^q_x}^2.
		$$
		By Lemma \ref{lemma0}(c),
		$$
		\|u(\cdot,y)\|_{L^q_x}^2
		\le
		C
		\left(
		\|u(\cdot,y)\|_2^2
		+
		\|D_x^su(\cdot,y)\|_2^2
		\right).
		$$
		Integrating in $y$ yields
		$$
		\int_{\bb_\delta^n\times\mathbb R^m}|u|^2
		\le
		C\delta^{n(1-2/q)}\|u\|_X^2.
		$$
		Hence,
		$
		\|u\|_{L^2(\bb_\delta^n\times\mathbb R^m)}
		\le
		C\delta^\alpha\|u\|_X.
		$
				For $p$, choose $\theta\in(0,1)$ such that
		$
		\frac1p
		=
		\frac\theta2+\frac{1-\theta}{B}.
		$
		Then the  Hölder inequality and (H) give
		$$
		\|u\|_{L^p(E)}
		\le
		\|u\|_{L^2(E)}^\theta
		\|u\|_{L^{B}(E)}^{1-\theta}
		\le
		C\delta^{\alpha\theta}\|u\|_X
		$$
		for $E=\bb_\delta^n\times\mathbb R^m$.
		Set
		$
		\beta=\alpha\theta.
		$
	\end{proof}
	
	To exploit radial symmetry for decay at infinity, we will need to smoothly localize our functions. The next technical lemma shows that the fractional Sobolev norm is well-behaved under an appropriate bounded partition of unity.
	\begin{lemma}\label{lemma2} 
		
		Let $d\ge1$ and $t>0$. Let
		$$
		\|v\|_{H^t(\mathbb R^d)}^2
		=
		\int_{\mathbb R^d}
		(1+|\xi|^2)^t|\widehat v(\xi)|^2\,d\xi
		$$
		and $k=\lceil t\rceil$.
		Let $\{\chi_j\}_{j\in J}\subset C^\infty(\mathbb R^d)$ satisfy
		$ 
		0\le\chi_j\le1,
		$ $
		\|\partial^\beta\chi_j\|_\infty\le L
		\quad (|\beta|\le k),
		$ 
		and
		$
		\sum_j\mathbf 1_{\operatorname{supp}\chi_j}\le M
		$
		everywhere. Then
		$
		\sum_j\|\chi_jv\|_{H^t}^2
		\le
		C\|v\|_{H^t}^2,
		$ $
		C=C(d,t,M,L).
		$
	\end{lemma}
	
	\begin{proof}
		
		(i) The case $0<t<1$
		
		The Gagliardo seminorm
		$$
		[v]_t^2
		=
		\iint
		\frac{|v(z)-v(z')|^2}
		{|z-z'|^{d+2t}}
		\,dz\,dz'
		$$
		satisfies
		$$
		[v]_t^2
		=
		c_{d,t}
		\int|\xi|^{2t}|\widehat v(\xi)|^2\,d\xi.
		$$
		Indeed, by the Plancherel identity,
		$$
		\int|v(z+h)-v(z)|^2\,dz
		=
		\int|e^{ih\cdot\xi}-1|^2|\widehat v(\xi)|^2\,d\xi,
		$$
		and the identity follows after the change of variables
		$h\mapsto h/|\xi|$.
		
		It therefore suffices to bound
		$
		\sum_j\|\chi_jv\|_2^2
		+
		\sum_j[\chi_jv]_t^2.
		$
				First,
		$
		\sum_j\|\chi_jv\|_2^2
		=
		\int|v|^2\sum_j\chi_j^2
		\le
		M\|v\|_2^2.
		$
		
		Next,
		$$
		\chi_j(z)v(z)-\chi_j(z')v(z')
		=
		\chi_j(z)(v(z)-v(z'))
		+
		v(z')(\chi_j(z)-\chi_j(z')).
		$$
		The first term, after summation in $j$, contributes at most
		$
		2M[v]_t^2.
		$
		
		For the second term, fix $z,z'$ and put $h=z-z'$. If
		$\chi_j(z)\ne\chi_j(z')$, then $z$ or $z'$ belongs to
		$\operatorname{supp}\chi_j$, so at most $2M$ indices contribute.
		Moreover,
		$$
		|\chi_j(z)-\chi_j(z')|
		\le
		\min\{L|h|,1\}.
		$$
		Therefore
		$$
		\sum_j|\chi_j(z)-\chi_j(z')|^2
		\le
		2M\min\{L^2|h|^2,1\}.
		$$
		Consequently, the second term is bounded by
		$
		4M\|v\|_2^2
		\int_{\mathbb R^d}
		\frac{\min\{L^2|h|^2,1\}}
		{|h|^{d+2t}}\,dh.
		$
		Using polar coordinates,
		$$
		\int_{\mathbb R^d}
		\frac{\min\{L^2|h|^2,1\}}
		{|h|^{d+2t}}\,dh
		\le
		C_{d,t}L^{2t}.
		$$
				This proves the result for $0<t<1$.
		
		(ii) General $t$.
				Let $k=\lceil t\rceil$. Leibniz' rule gives
		$
		\partial^\beta(\chi_jv)
		=
		\sum_{\gamma\le\beta}
		\binom{\beta}{\gamma}
		\partial^{\beta-\gamma}\chi_j
		\,\partial^\gamma v.
		$
		Since every derivative of $\chi_j$ vanishes off
		$\operatorname{supp}\chi_j$,
		$
		\sum_j
		|\partial^{\beta-\gamma}\chi_j|^2
		|\partial^\gamma v|^2
		\le
		L^2M|\partial^\gamma v|^2.
		$
		Hence
		$
		T:v\mapsto(\chi_jv)_j
		$
		is bounded map
		$
		L^2\to\ell^2(J;L^2)
		$
		and
		$
		H^k\to\ell^2(J;H^k),
		$
		uniformly in $J$.
		
		By complex interpolation,
		$
		[L^2,H^k]_{t/k}=H^t
		$
		and
		$
		[\ell^2(J;A_0),\ell^2(J;A_1)]_\theta
		=
		\ell^2(J;[A_0,A_1]_\theta).
		$
		Therefore
		$
		T:H^t\to\ell^2(J;H^t)
		$
		is bounded, which proves the claim.
	\end{proof} 
	
	\emph{Remark.}
	Part (i) is completely elementary. For $t\ge1$ the proof uses standard interpolation theory.
	
	Building toward our decay estimates, we now construct a specific angular partition of unity on spherical shells. This construction relies on rotations and is perfectly tailored to interact with radial functions.
	
	\begin{lemma}\label{lemma3}

		Let $d\ge2$. There are $c,C,M>0$ and $C_\beta$, depending only on $d$, such that for every $R\ge2$ there exist
		$K_R$ with
		$
		cR^{d-1}\le K_R\le CR^{d-1},
		$
		rotations $Q_1,\dots,Q_{K_R}\in SO(d)$, and
		$\chi_R^0\in C_c^\infty(\mathbb R^d)$ such that
		$
		\chi_{R,j}:=\chi_R^0\circ Q_j^{-1}
		$
		satisfy:
		
		\begin{enumerate}
			\item
			$
			0\le\chi_{R,j}\le1;
			$
			
			\item
			$
			\|\partial^\beta\chi_{R,j}\|_\infty
			\le C_\beta,
			$
			independently of $R,j$;
			
			\item
			$
			\sum_j
			\mathbf 1_{\operatorname{supp}\chi_{R,j}}
			\le M;
			$
			
			\item every $z$ with $R<|z|<2R$ has some $j$ with
			$
			\chi_{R,j}(z)=1.
			$
		\end{enumerate}
		
	\end{lemma}
	\begin{proof}

		Fix
		$
		\vartheta\in C_c^\infty\left(\frac12,3\right),
		$  $
		0\le\vartheta\le1,
		$ $
		\vartheta=1\text{ on }[1,2],
		$
		and
		$
		\varphi\in C^\infty([0,\infty);[0,1]),
		$ $
		\varphi=1\text{ on }[0,1],
		$ $
		\varphi=0\text{ on }[2,\infty).
		$
		Put
		$$
		\chi_R^0(z)
		=
		\vartheta(|z|/R)
		\varphi\big(R|\widehat z-e_d|\big),
		\qquad
		\widehat z=\frac z{|z|}.
		$$
		This is smooth because
		$\vartheta(|z|/R)$ vanishes near $z=0$.
		Let $\{\theta_j\}_{j\le K_R}$ be a maximal
		$1/R$-separated subset of $S^{d-1}$. By maximality, the caps
		$
		B(\theta_j,1/R)\cap S^{d-1}
		$
		cover $S^{d-1}$. By separation, the caps of radius $1/(2R)$
		are disjoint. Since the area of a spherical cap of radius $t$
		satisfies
		$
		|B(\theta,t)\cap S^{d-1}|
		\simeq t^{d-1}
		$
		for $t\le2$, comparison of areas gives
		$
		cR^{d-1}\le K_R\le CR^{d-1}.
		$
		
		Choose $Q_j\in SO(d)$ with
		$
		Q_je_d=\theta_j.
		$
		Then
		$$
		\chi_{R,j}(z)
		=
		\chi_R^0(Q_j^{-1}z)
		=
		\vartheta(|z|/R)
		\varphi(R|\widehat z-\theta_j|).
		$$
		
		If $R<|z|<2R$, then $\vartheta(|z|/R)=1$. By the covering
		property, some $j$ satisfies
		$
		|\widehat z-\theta_j|<1/R,
		$
		and hence
		$
		\varphi(R|\widehat z-\theta_j|)=1.
		$
		This proves (4).
		
		If
		$
		z\in
		\operatorname{supp}\chi_{R,j}
		\cap
		\operatorname{supp}\chi_{R,k},
		$
		then
		$
		|\widehat z-\theta_j|,
		|\widehat z-\theta_k|
		\le\frac2R,
		$
		and hence
		$
		|\theta_j-\theta_k|
		\le\frac4R.
		$
		The disjoint caps of radius $1/(2R)$ around such $\theta_k$
		are contained in a cap of radius $9/(2R)$ around $\theta_j$.
		Comparing areas gives the uniform overlap bound (3).
		
		Finally, derivatives of $\chi_R^0$ are uniformly bounded in
		$R$. Indeed, writing
		$
		\chi_R^0(z)=F_R(z/R),
		$
		all derivatives of $F_R$ on the fixed annulus
		$1/2<|w|<3$ grow at most like $R^{|\beta|}$, while
		$
		\partial_z^\beta
		=
		R^{-|\beta|}\partial_w^\beta.
		$
		Thus
		$
		|\partial^\beta\chi_R^0(z)|
		\le C_\beta,
		$
		and the same holds for $\chi_{R,j}$ because rotations preserve
		derivative norms.
	\end{proof}

	We can now establish a Strauss-type radial estimate. By combining the preceding partition of unity with our fractional localization lemma, we show that functions in $X$ that are radial in $x$ exhibit uniform mass decay as $|x| \to \infty$.
	\begin{lemma}\label{lemma4} 
		
		Let $n\ge2$, and $2\le p\le B$. There is $C$ such that
		for all $R\ge2$ and all $u\in X_{\rm rad,x}$,
		$$
		\int_{\{R<|x|<2R\}\times\mathbb R^m}
		|u|^p\,dx\,dy
		\le
		CR^{-\beta_x}\|u\|_X^p,
		$$
		where
		$
		\beta_x
		=
		\frac{(n-1)(p-2)}2.
		$
		Consequently, for $2<p\le B$ and $R\ge2$,
		\begin{equation}
			\int_{\{|x|>R\}\times\mathbb R^m}
			|u|^p\,dx\,dy
			\le
			C'R^{-\beta_x}\|u\|_X^p.
			\label{1-eq4.1}
		\end{equation}
	\end{lemma}
	
	\begin{proof}

		Take $\chi_{R,j}(x)$ from Lemma \ref{lemma3} with $d=n$, viewed as
		functions of $(x,y)$.
		
		Step 1: localization in $X$
		
		We have
		$$
		\sum_j\|\chi_{R,j}u\|_2^2
		\le
		M\|u\|_2^2.
		$$
		
		Since $\chi_{R,j}$ is independent of $y$, then
		$
		\mathcal F_y(\chi_{R,j}u)
		=
		\chi_{R,j}\mathcal F_yu,
		$
		and therefore
		$$
		\sum_j
		\||x|^\gamma D_y^r(\chi_{R,j}u)\|_2^2
		\le
		M
		\||x|^\gamma D_y^ru\|_2^2.
		$$
		
		For the $x$-part, apply Lemma \ref{lemma2} with $d=n$ and $t=s$ to
		$u(\cdot,y)$ for a.e. $y$, and integrate in $y$. Together with
		Lemma \ref{lemma0},
		\begin{equation}
			\sum_j\|\chi_{R,j}u\|_X^2
			\le
			C_1\|u\|_X^2,
			\label{1-eq4.2}
		\end{equation}
		where $C_1$ is independent of $R$.
		
		Step 2: 		
		Let
		$
		T_j(x,y)=(Q_jx,y).
		$
		Since $u$ is radial in $x$,
		$
		u\circ T_j=u,
		$
		and
		$
		\chi_{R,j}\circ Q_j=\chi_R^0.
		$
		Thus
		$
		(\chi_{R,j}u)\circ T_j
		=
		\chi_R^0u.
		$
		By invariance of $\|\cdot\|_X$,
		$
		\|\chi_{R,j}u\|_X
		=
		\|\chi_R^0u\|_X.
		$
		Hence \eqref{1-eq4.2} gives
		$
		K_R\|\chi_R^0u\|_X^2
		\le
		C_1\|u\|_X^2,
		$
		and therefore
		$
		\|\chi_{R,j}u\|_X^2
		\le
		\frac{C_1}{K_R}\|u\|_X^2.
		$
		
		Step 3: 
		By (S),
		$
		\|\chi_{R,j}u\|_p^p
		\le
		CK_R^{-p/2}\|u\|_X^p.
		$
		By Lemma \ref{lemma3}(4),
		$ 
		|u|^p
		\le
		\sum_j|\chi_{R,j}u|^p
		$ 
		on
		$\{R<|x|<2R\}\times\mathbb R^m$.
		Consequently,
		\[
				\begin{split}
			\int_{\{R<|x|<2R\}\times\mathbb R^m}|u|^p
		&	\le
			\sum_j\|\chi_{R,j}u\|_p^p \le
			CK_R^{1-p/2}\|u\|_X^p =
			CK_R^{-(p-2)/2}\|u\|_X^p \le
			CR^{-\frac{(n-1)(p-2)}2}\|u\|_X^p.
		\end{split}
\]
		
		For \eqref{1-eq4.1}, split
		$ 
		\{|x|>R\}
		=
		\bigcup_{k=0}^\infty
		\{2^kR<|x|<2^{k+1}R\}
		$ 
		up to null spheres. Summing gives
		$ 
		\sum_{k=0}^\infty
		C(2^kR)^{-\beta_x}
		=
		C'R^{-\beta_x},
		$ 
		which is finite because $\beta_x>0$ for $p>2$.
	\end{proof}

 When the functions are additionally radial in the $y$-variable, an analogous decay estimate holds as $|y|\to \infty$. The presence of the spatial weight $|x|^\gamma$ requires us to carefully localize in $y$ while keeping $x$ restricted to a bounded set.
	\begin{lemma}\label{lemma5} 
		Let $m\ge2$, $2\le p\le B$, and $M\ge1$. There is
		$C_M$ such that for $R\ge2$ and every $u\in X$ radial in $y$,
		$$
		\int_{\{|x|<M\}\times\{R<|y|<2R\}}
		|u|^p\,dx\,dy
		\le
		C_MR^{-\beta_y}\|u\|_X^p,
		$$
		where
		$ 
		\beta_y
		=
		\frac{(m-1)(p-2)}2.
		$ 
		For $p>2$,
		$$
		\int_{\{|x|<M\}\times\{|y|>R\}}
		|u|^p\,dx\,dy
		\le
		C_M'R^{-\beta_y}\|u\|_X^p.
		$$
	\end{lemma}
	\begin{proof}

		Let
		$ 
		\zeta\in C_c^\infty(\mathbb R^n),
		$ $
		0\le\zeta\le1,
		$ $
		\zeta=1\text{ on }\bb_M^n,
	$ and $
		\operatorname{supp}(\zeta)\subset \bb_{2M}^n,
		$ 
		and put
		$ 
		w=\zeta u.
		$ 
		Then $w$ is radial in $y$ and
		$ 
		\|w\|_X\le C_M\|u\|_X.
		$ 
		Indeed,
		$$
		\|D_x^sw\|_2
		\le
		C_\zeta\|u\|_{H_x^s}
		\le
		C_M\|u\|_X
		$$
		by Lemma \ref{lemma0}(a), while
		$ 
		|x|^\gamma D_y^sw
		=
		\zeta |x|^\gamma D_y^ru
		$ 
		and hence
		$ 
		\||x|^\gamma D_y^rw\|_2
		\le
		\||x|^\gamma D_y^ru\|_2.
		$ 
		
		Localization in $y$.
				Take
		$ 
		\psi_j(y)=\chi_{R,j}(y)
		$ 
		from Lemma \ref{lemma3} with $d=m$.
				First,
		$ 
		\sum_j\|\psi_jw\|_2^2
		\le
		M\|w\|_2^2.
		$ 
				Since $\psi_j$ is independent of $x$,
		$ 
		D_x^s(\psi_jw)
		=
		\psi_jD_x^sw,
		$ 
		and therefore
		$ 
		\sum_j\|D_x^s(\psi_jw)\|_2^2
		\le
		M\|D_x^sw\|_2^2.
		$ 
				For a.e. $x$, Lemma \ref{lemma2} with $d=m$ and $t=r$ gives
		$$
		\sum_j
		\|\psi_jw(x,\cdot)\|_{H_y^r}^2
		\le
		C
		\|w(x,\cdot)\|_{H_y^r}^2.
		$$
		Multiplying by $|x|^{2\gamma}$ and integrating in $x$ gives
		$$
		\sum_j
		\||x|^\gamma D_y^r(\psi_jw)\|_2^2
		\le
		C
		\left(
		\||x|^\gamma D_y^rw\|_2^2
		+
		(2M)^{2\gamma}\|w\|_2^2
		\right).
		$$
		Thus
		$  
		\sum_j\|\psi_jw\|_X^2
		\le
		C_M\|w\|_X^2.
		$ 
		The truncation is essential here: the zeroth-order term produced by
		the localization in $y$ contains
		$|x|^{2\gamma}\|w(x,\cdot)\|_2^2$, which is controlled because
		$|x|\le2M$ on the support of $w$.
		
		Now let
		$ 
		T_j(x,y)=(x,Q_jy).
		$ 
		Since $w$ is radial in $y$,
		$
		w\circ T_j=w.
		$
		Repeating the rotation and covering argument of Lemma \ref{lemma4} gives
		$$
		\int_{\mathbb R^n\times\{R<|y|<2R\}}
		|w|^p\,dx\,dy
		\le
		C_MR^{-\beta_y}\|w\|_X^p.
		$$
		Since $w=u$ on $\bb_M^n\times\mathbb R^m$,
		$$
		\int_{\{|x|<M\}\times\{R<|y|<2R\}}
		|u|^p\,dx\,dy
		\le
		C_MR^{-\beta_y}\|u\|_X^p.
		$$
		The tail estimate follows by dyadic summation.
	\end{proof}

Having established uniform control both near the origin and at spatial infinity, we must now obtain compactness in the intermediate region. The next lemma utilizes the Riesz--Kolmogorov theorem to prove local compactness on bounded domains that are strictly separated from $x=0$.
	\begin{lemma}\label{lemma6} 
		
		Let
		$
		0<\delta<M<\infty,
		\qquad
		T>0,
		$
		and
		$
		U
		=
		\{\delta<|x|<M\}\times \bb_T^m.
		$
		If $(u_j)$ is bounded in $X$, then $(u_j|_U)$ is relatively
		compact in $L^2(U)$. If $2\le p< B$ and (H) holds, then
		$(u_j|_U)$ is relatively compact in $L^p(U)$.
	\end{lemma}
	
	\begin{proof}

		Choose
		$
		\eta_1\in C_c^\infty(\mathbb R^n)
		$
		such that
		$
		0\le\eta_1\le1,
		\qquad
		\eta_1=1
		\quad\text{on }\{\delta\le|x|\le M\},
		$
		and
		$
		\operatorname{supp}\eta_1
		\subset
		\left\{\frac\delta2<|x|<2M\right\}.
		$
		Choose
		$
		\eta_2\in C_c^\infty(\mathbb R^m)
		$
		such that
		$
		\eta_2=1$ on $\bb_T^m,
		$ $
		\operatorname{supp}(\eta_2)\subset \bb_{T+1}^m.
		$
		Put
		$ 
		f_j=\eta_1u_j,
		\qquad
		g_j=\eta_2f_j.
		$ 
		These functions are bounded in $L^2$ and have support in a fixed
		compact set.
				For $x$-translations,
		$$
		\tau_{(h,0)}g_j-g_j
		=
		\eta_2(\tau_hf_j-f_j).
		$$
		Hence, by Lemma \ref{lemma0}(b) and (a),
		$$
		\|\tau_{(h,0)}g_j-g_j\|_2
		\le
		2|h|^\sigma\|f_j\|_{H_x^s}
		\le
		C|h|^\sigma\|u_j\|_X.
		$$
		
		For $y$-translations,
		$$
		\tau_{(0,k)}g_j-g_j
		=
		(\tau_k\eta_2)(\tau_kf_j-f_j)
		+
		(\tau_k\eta_2-\eta_2)f_j.
		$$
		Therefore
		$$
		\|\tau_{(0,k)}g_j-g_j\|_2
		\le
		\|\tau_kf_j-f_j\|_2
		+
		\|\nabla\eta_2\|_\infty|k|\|f_j\|_2.
		$$
		By Lemma \ref{lemma0}(b),
		$$
		\|\tau_kf_j-f_j\|_2^2
		\le
		4|k|^{2\sigma'}
		\left(
		\|u_j\|_2^2
		+
		\left(\frac\delta2\right)^{-2\gamma}
		\||x|^\gamma D_y^ru_j\|_2^2
		\right).
		$$
		Indeed, on $\operatorname{supp}\eta_1$,
		$
		|x|\ge\frac\delta2,
		$
		so the weight is bounded below. No commutator term occurs because
		$\eta_1$ is independent of $y$.
		
		Thus
		$$
		\sup_j\|\tau_zg_j-g_j\|_2\to0
		\qquad\text{as }z\to0.
		$$
		Since the $g_j$ have uniformly compact support, the
		Riesz--Kolmogorov theorem implies that $(g_j)$ is relatively compact
		in $L^2(\mathbb R^{n+m})$. Since
		$
		g_j=u_j\quad\text{on }U,
		$
		the restrictions $(u_j|_U)$ are relatively compact in $L^2(U)$.
		
		For $L^p$, let $\mathcal B$ denote the set of restrictions to $U$
		of functions in the unit ball of $X$. It is totally bounded in
		$L^2(U)$, and by (H),
		$ 
		\|f\|_{L^{ B}(U)}\le C
		$ 
		uniformly for $f\in\mathcal B$. Hence, for $f,g\in\mathcal B$,
		$$
		\|f-g\|_p
		\le
		\|f-g\|_2^\theta
		\|f-g\|_{ B}^{1-\theta}
		\le
		C\|f-g\|_2^\theta,
		$$
		where
		$
		\frac1p
		=
		\frac\theta2+\frac{1-\theta}{ B}.
		$
		Thus $\mathcal B$ is totally bounded in $L^p(U)$.
	\end{proof}

 We are now ready to prove our first compact embedding. By gluing together the control near $x=0$ (Lemma \ref{lemma1}), the decay at $|x| \to \infty$ (Lemma \ref{lemma4}), and the local compactness in the intermediate region (Lemma \ref{lemma6}), we obtain compactness for $x$-radial functions on domains bounded in $y$.
	
	\begin{theorem}
		Assume (H), $n\ge2$, $m\ge1$, and
		$
		2<p< B,
		$ $
		T>0.
		$
		Then the restriction map
		$
		X_{\rm rad,x}
		\to
		L^p(\mathbb R^n\times \bb_T^m)
		$
		is compact.
				In particular, every bounded sequence in $X_{\rm rad,x}$ has a
		subsequence converging in
		$ 
		L^p(\mathbb R^n\times \bb_T^m).
		$ 
		If
		$ 
		u_j\rightharpoonup u
		\quad\text{weakly in }X_{\rm rad,x},
		$ 
		then
		$ 
		u_j\to u
		\quad\text{in }L^p(\mathbb R^n\times \bb_T^m).
		$ 
	\end{theorem}
	
\begin{proof}
	
	Let $\mathcal B$ be the unit ball of $X_{\rm rad,x}$ and let
	$\varepsilon>0$.
  By \eqref{1-eq4.1}, choose $M\ge2$ such that
		$$
		\sup_{u\in\mathcal B}
		\|u\|_{L^p(\{|x|>M\}\times\mathbb R^m)}
		<
		\frac{\varepsilon}{3}.
		$$
		This uses radiality in $x$ and $n\ge2$.
  By Lemma \ref{lemma1}, choose
		$\delta\in(0,M)$ such that
		$$
		\sup_{u\in\mathcal B}
		\|u\|_{L^p(\bb_\delta^n\times\mathbb R^m)}
		<
		\frac{\varepsilon}{3}.
		$$
  Let
		$ 
		U
		=
		\{\delta<|x|<M\}\times \bb_T^m.
		$ 
		By Lemma \ref{lemma6}, $\mathcal B|_U$ is totally bounded in $L^p(U)$.
		Let
		$ 
		f_1,\dots,f_N
		$ 
		be an $\varepsilon/3$-net, extended by zero to
		$\mathbb R^n\times \bb_T^m$.
 For $u\in\mathcal B$, choose $i$ such that
	$$
	\|u-f_i\|_{L^p(U)}
	<
	\frac{\varepsilon}{3}.
	$$
	Up to null sets,
	$$
	\mathbb R^n\times \bb_T^m
	=
	U
	\cup
	(\bb_\delta^n\times \bb_T^m)
	\cup
	(\{|x|\ge M\}\times \bb_T^m).
	$$
	Therefore
	$ 
	\|u-f_i\|_{L^p(\mathbb R^n\times \bb_T^m)}
	<
	\varepsilon.
	$ 
	Thus the image of $\mathcal B$ is totally bounded in
	$L^p(\mathbb R^n\times \bb_T^m)$ and hence relatively compact.
		The weak-convergence statement follows because compact linear
	operators map weakly convergent sequences to strongly convergent
	sequences.
\end{proof}

 Finally, we can extend this result to the entire space $\mathbb R^{n+m}$ by imposing full radiality. By incorporating the uniform $y$-decay from Lemma \ref{lemma5} to control the remaining tails at infinity, we obtain our global compact embedding theorem.
	
	\begin{theorem}
		Assume (H),
		$ 
		n\ge2,
		$ $
		m\ge2,
		$ $
		2<p< B.
		$ 
		Then
		$ 
		X_{\rm rad,x,y}
		\hookrightarrow
		L^p(\mathbb R^{n+m})
		$ 
		is compact.
	\end{theorem}
\begin{proof}
 
	Let $\mathcal B$ be the unit ball of $X_{\rm rad,x,y}$ and let
	$\varepsilon>0$.
	  Choose $M\ge2$ such that
		$$
		\sup_{u\in\mathcal B}
		\|u\|_{L^p(\{|x|>M\}\times\mathbb R^m)}
		<
		\frac{\varepsilon}{4}
		$$
		by \eqref{1-eq4.1}.
	  Choose $\delta\in(0,M)$ such that
		$$
		\sup_{u\in\mathcal B}
		\|u\|_{L^p(\bb_\delta^n\times\mathbb R^m)}
		<
		\frac{\varepsilon}{4}
		$$
		by Lemma \ref{lemma1}.
	  With $M$ fixed, choose $T\ge2$ such that
		$$
		\sup_{u\in\mathcal B}
		\|u\|_{L^p(\{|x|<M\}\times\{|y|>T\})}
		<
		\frac{\varepsilon}{4}
		$$
		by Lemma \ref{lemma5}, using radiality in $y$ and $m\ge2$.
	  Let
		$ 
		U
		=
		\{\delta<|x|<M\}\times \bb_T^m.
		$ 
		By Lemma \ref{lemma6}, $\mathcal B|_U$ has an
		$\varepsilon/4$-net
		$ 
		f_1,\dots,f_N
		$ 
		in $L^p(U)$. Extend these functions by zero to
		$\mathbb R^{n+m}$.
	Up to null sets,
$$\mathbb R^{n+m}
={}
U
\cup
(\bb_\delta^n\times\mathbb R^m)
\cup
(\{|x|\ge M\}\times\mathbb R^m)
\cup
(\{\delta\le|x|<M\}\times\{|y|\ge T\}).$$
	Hence, for every $u\in\mathcal B$, there is an $i$ such that
	$$
	\|u-f_i\|_{L^p(\mathbb R^{n+m})}
	<
	\varepsilon.
	$$
	Therefore $\mathcal B$ is totally bounded in
	$L^p(\mathbb R^{n+m})$, proving compactness.
\end{proof}

	\section{Strauss' lemma}\label{sec-strauss}
	Let $u=u(x,y)$, $x\in\mathbb R^n$, $y\in\mathbb R^m$, and put
	\[
	M_1:=\|D_x^su\|_{L^2(\mathbb R^{n+m})},\qquad
	M_2:=\||x|^bD_y^ru\|_{L^2(\mathbb R^{n+m})},\qquad
	M_0:=\|u\|_{L^2(\mathbb R^{n+m})}.
	\]
	Let $\tilde u(x,\eta)$ be the Fourier transform in $y$ only. By the Plancherel identity (constants omitted),
	\[
	\|D_x^su\|_2^2=\!\int\! h(\eta)^2d\eta,\quad
	\||x|^bD_y^ru\|_2^2=\!\int\! g(\eta)^2d\eta,
	\]
	where
	\[
	h(\eta)=\|D_x^s\tilde u(\cdot,\eta)\|_{L^2_x},\qquad
	g(\eta)=|\eta|^r\,\||x|^b\tilde u(\cdot,\eta)\|_{L^2_x}.
	\]
	``Radial in $x$'' means $u(x,y)=u(|x|,y)$, and similarly for $y$.
	Throughout, $1/2<s<n/2$, $b>0$, $m\ge2$, and
	\[
	a_0:=\frac{2b+1}{2(s+b)}\in(0,1).
	\]

	To begin our analysis of radial functions, we first establish a foundational estimate for the Fourier transform of radial distributions, controlling the standard spherical kernel.
	\begin{lemma} \label{lem:kernel}
		Let $\Phi_d(t)=\int_{S^{d-1}}e^{it\omega_1}d\omega$. Then $|\Phi_d(t)|\le C\min(1,t^{-(d-1)/2})$.
		If $f$ is radial on $\mathbb R^d$, then
		\[
		f(x)=c\int_0^\infty \hat f(\rho)\,\Phi_d(\rho|x|)\,\rho^{d-1}d\rho ,
		\qquad
		|f(x)|\le C\int_0^\infty \|\hat f(\rho)\|_H \min\!\big(1,(\rho|x|)^{-\frac{d-1}2}\big)\rho^{d-1}d\rho .
		\]
	\end{lemma}

	With these kernel bounds in hand, we can now derive the classical Strauss-type pointwise estimates. This lemma quantifies the radial decay of functions using only their fractional Sobolev and $L^2$ norms.
	
	\begin{lemma}[Strauss, and a three-norm version]\label{lem:strauss}
		Let $f$ be radial on $\mathbb R^n$, $\tfrac12<s<\tfrac n2$. Then
		\begin{align}
			&\text{(i)}\quad |x|^{\frac n2-s}\,|f(x)|\lesssim \|D^sf\|_2,\\
			&\text{(ii)}\quad |x|^{\frac{n-1}2}\,|f(x)|\lesssim \|f\|_2^{\,1-\frac1{2s}}\|D^sf\|_2^{\,\frac1{2s}} .
		\end{align}
	\end{lemma}
	
	\begin{proof}
		(i) By Lemma \ref{lem:kernel} and the Cauchy--Schwarz inequality,
		\[
		|f(x)|\lesssim \|D^sf\|_2\Big(\int_0^\infty\rho^{-2s}\min\big(1,(\rho|x|)^{-(n-1)}\big)\rho^{n-1}d\rho\Big)^{1/2}.
		\]
		The integral is
		\[
		\int_0^{1/|x|}\rho^{n-1-2s}d\rho+|x|^{-(n-1)}\int_{1/|x|}^\infty\rho^{-2s}d\rho
		=\Big(\tfrac1{n-2s}+\tfrac1{2s-1}\Big)|x|^{2s-n},
		\]
		which is finite exactly when $\frac12<s<\frac n2$.
		
		(ii) Let $D=\|D^sf\|_2$, $E=\|f\|_2$, $N_*=(D/E)^{1/s}$. Split $\hat f$ into $\rho\le N_*$ and $\rho>N_*$.
		
		\emph{Case $N_*|x|\ge1$.} By Cauchy--Schwarz, the low part is $\lesssim E\,(N_*|x|^{-(n-1)})^{1/2}$ and the high part is
		$\lesssim D\,|x|^{-\frac{n-1}2}N_*^{\frac12-s}$. Both equal $|x|^{-\frac{n-1}2}E^{1-\frac1{2s}}D^{\frac1{2s}}$.
		
		\emph{Case $N_*|x|<1$.} Use the bound $1$ for the kernel on $\rho\le 1/|x|$. This gives low part $\lesssim N_*^{n/2}E$ and high part $\lesssim |x|^{s-\frac n2}D$.
		Relative to the right-hand side of (ii) these are $(N_*|x|)^{\frac{n-1}2}\le1$ and $(N_*|x|)^{s-\frac12}\le 1$ respectively.
	\end{proof}

	In our setting, we must also account for the spatial weight $\vert{}x\vert{}^b$. The following Heisenberg--Pitt inequality allows us to interpolate between the unweighted $L^2$ norm, the fractional derivative, and the weighted mass.
	
	\begin{lemma}[Heisenberg--Pitt type]\label{lem:pitt}
		For $0<s<n/2$, $b>0$:
		$\|f\|_2\lesssim\|D^sf\|_2^{\frac b{s+b}}\,\||x|^bf\|_2^{\frac s{s+b}}$.
	\end{lemma}
	\begin{proof}
		For $R>0$, by Sobolev ($\frac12-\frac1{p_0}=\frac sn$) and the  Hölder inequality on $\bb_R$,
		$\|f\|_2\le R^{-b}\||x|^bf\|_2+\|f\|_{L^2(\bb_R)}\lesssim R^{-b}\||x|^bf\|_2+R^{s}\|D^sf\|_2$.
		Choose $R=(\||x|^bf\|_2/\|D^sf\|_2)^{1/(s+b)}$.
	\end{proof}

	By combining the standard Strauss estimate with the preceding interpolation inequality, we obtain a weighted radial bound. This seamlessly integrates the fractional derivative and the spatial weight into a single pointwise decay estimate.

	\begin{lemma}[weighted radial Strauss]\label{lem:weighted}
		Let $f$ be radial on $\mathbb R^n$ (possibly $H$-valued), $\frac12<s<\frac n2$, $b>0$. Then for every $a\in[a_0,1]$,
		\[
		|x|^{\frac n2-as+(1-a)b}\,|f(x)|\;\lesssim\;\|D^sf\|_2^{\,a}\,\||x|^bf\|_2^{\,1-a}.
		\]
	\end{lemma}
	\begin{proof}
		For $a=a_0$, insert Lemma \ref{lem:pitt} into Lemma \ref{lem:strauss}(ii).
		The exponent of $\||x|^bf\|_2$ is $\frac s{s+b}(1-\frac1{2s})=\frac{2s-1}{2(s+b)}=1-a_0$,
		and $\frac n2-a_0s+(1-a_0)b=\frac{n-1}2$.
		For $a=1$ this is Lemma \ref{lem:strauss}(i).
		For $a\in(a_0,1)$ take the corresponding geometric mean of the two pointwise inequalities; the weights and exponents are affine in $a$.
	\end{proof}
	
	We are now positioned to state the main decay result of this section. By exploiting radial symmetry in both the $x$ and $y$ variables, this theorem yields sharp, scale-invariant pointwise decay estimates across the entire space $\mathbb R^{n+m}$.
	\begin{theorem}\label{thm:main-r}
		Let $n\ge 2$, $m\ge2$, $\frac12<s<\frac n2$, $b>0$, $r>0$, and let $u$ be radial in $x$ and radial in $y$. Let $a\in[a_0,1]$ satisfy
		\begin{equation}\label{cond}
			\frac12<(1-a)\,r<\frac m2 .
		\end{equation}
		Then
		\begin{equation}\label{main}
			|x|^{A}\,|y|^{B}\,|u(x,y)|\;\le\; C\,\|D_x^su\|_2^{\,a}\;\||x|^bD_y^ru\|_2^{\,1-a},
			\qquad
			A=\frac n2-as+(1-a)b,\quad B=\frac m2-(1-a)r .
		\end{equation}
		The set of admissible $a$ is nonempty iff $s+b<r(2s-1)$.
	\end{theorem}
	
	\begin{remark}[scaling]
		Replacing $u$ by $u(\lambda x,\mu y)$ multiplies the left side of \eqref{main} by $\lambda^{-A}\mu^{-B}$ and the right side by
		$\lambda^{a(s-\frac n2)-(1-a)(b+\frac n2)}\mu^{-\frac{am}2+(1-a)(r-\frac m2)}$.
		Equality of the two factors holds if and only if $A,B$ are as in \eqref{main}. So the decay rates are forced by scaling.
	\end{remark}
	
	\begin{proof}
		Write $\sigma=|\eta|$ and $u(x,y)=c\int e^{iy\cdot\eta}\tilde u(x,\eta)d\eta$.
		Since $\tilde u(x,\cdot)$ is radial, Lemma \ref{lem:kernel} gives
		\[
		|u(x,y)|\lesssim\int_{\mathbb R^m} w_y(\eta)\,|\tilde u(x,\eta)|\,d\eta,\qquad
		w_y(\eta)=\min\big(1,(|\eta||y|)^{-\frac{m-1}2}\big).
		\]
		For fixed $\eta$, $\tilde u(\cdot,\eta)$ is radial in $x$. Lemma \ref{lem:weighted} gives
		\[
		|x|^{A}|\tilde u(x,\eta)|\lesssim h(\eta)^a\big(|\eta|^{-r}g(\eta)\big)^{1-a}.
		\]
		Hence
		\[
		|x|^A|u(x,y)|\lesssim\int F_y(\eta)\,h(\eta)^ag(\eta)^{1-a}d\eta,\qquad F_y(\eta)=w_y(\eta)|\eta|^{-(1-a)r}.
		\]
		By the Cauchy--Schwarz, and then the  Hölder inequalities ($\int h^{2a}g^{2(1-a)}\le(\int h^2)^a(\int g^2)^{1-a}$),
		\[
		|x|^A|u(x,y)|\lesssim \|F_y\|_{L^2(d\eta)}\,M_1^{a}M_2^{1-a}.
		\]
		Put $\kappa=(1-a)r$. Then
		\[
		\|F_y\|_2^2\simeq\int_0^{1/|y|}\sigma^{m-1-2\kappa}d\sigma+|y|^{-(m-1)}\!\int_{1/|y|}^\infty\sigma^{-2\kappa}d\sigma
		\simeq |y|^{2\kappa-m},
		\]
		which is finite iff $\frac12<\kappa<\frac m2$, i.e.\ \eqref{cond}. This gives $\|F_y\|_2\simeq|y|^{-B}$.
		Finally, \eqref{cond} is solvable with $1-a\in(0,1-a_0]$ iff $\frac1{2r}<1-a_0=\frac{2s-1}{2(s+b)}$, i.e.\ $s+b<r(2s-1)$.
	\end{proof}

	\begin{proof}[Second proof of Theorem \ref{thm:main-r}]
		\emph{Step 1 (Bernstein in $y$).} By Lemma \ref{lem:kernel} and  the Cauchy--Schwarz inequality on $\{|\eta|\sim2^k\}$,
		\[
		|u_k(x,y)|\lesssim 2^{\frac{km}2}(1+2^k|y|)^{-\frac{m-1}2}\,G_k(x),\qquad G_k(x):=\|u_k(x,\cdot)\|_{L^2_y}.
		\]
		\emph{Step 2 (radial lemma with values in $H=L^2_y$).} The function $x\mapsto u_k(x,\cdot)\in H$ is radial in $x$, and Lemma \ref{lem:weighted} holds for $H$-valued functions with the same proof (replace $|\cdot|$ by $\|\cdot\|_H$ in Lemma \ref{lem:kernel}). So
		\[
		|x|^{A}G_k(x)\lesssim C_k^{\,a}(2^{-kr}B_k)^{1-a}\le 2^{-k(1-a)r}M_1^{a}M_2^{1-a}.
		\]
		\emph{Step 3 (sum in $k$).} With $B=\frac m2-(1-a)r$,
		\[
		|x|^A|u(x,y)|\lesssim M_1^aM_2^{1-a}\sum_{k}2^{kB}(1+2^k|y|)^{-\frac{m-1}2}.
		\]
		Terms with $2^k|y|\le1$ form a geometric series if $B>0$ (i.e.\ $(1-a)r<\frac m2$), with sum $\simeq|y|^{-B}$.
		Terms with $2^k|y|>1$ form a geometric series if $B<\frac{m-1}2$ (i.e.\ $(1-a)r>\frac12$), with sum $\simeq|y|^{-B}$.
		This is \eqref{main}.
	\end{proof}
	
	\begin{remark}
		Step 3 only uses $\sup_kC_k^a B_k^{1-a}$, so the same argument gives the slightly stronger statement with $M_1^aM_2^{1-a}$ replaced by $\sup_k C_k^a B_k^{1-a}$, i.e.\ with $\ell^\infty_k$ norms instead of $\ell^2_k$. One can also decompose in $x$-frequency. This is not done above because $P^x_j$ does not commute with $|x|^b$.
	\end{remark}

	It is natural to ask what happens if radial symmetry is only assumed in one of the variables. If the function is radial in $x$ but arbitrary in $y$, we lose pointwise decay in $y$, but we can still recover mixed-norm estimates with uniform decay in the $x$-variable.

	\begin{theorem}[radial in $x$ only]\label{thm:xrad}
		Let $u$ be radial in $x$ and arbitrary in $y$. Let $\frac12<s<\frac n2$, $b>0$, $a\in[a_0,1]$ and $\kappa=(1-a)r<\frac m2$. Then, with $A$ as in \eqref{main},
		\[
		\sup_{x}\;|x|^{A}\,\big\|D_y^{(1-a)r}u(x,\cdot)\big\|_{L^2_y}\lesssim M_1^{a}M_2^{1-a}.
		\]
		Consequently, for $\frac1q=\frac12-\frac{(1-a)r}m$ (and $a<1$),
		\[
		\sup_{x}\;|x|^{A}\,\|u(x,\cdot)\|_{L^q_y(\mathbb R^m)}\lesssim M_1^{a}M_2^{1-a}.
		\]
	\end{theorem}
	\begin{proof}
		For each $\eta$, Lemma \ref{lem:weighted} applied to $\tilde u(\cdot,\eta)$ gives
		$|x|^A|\eta|^{(1-a)r}|\tilde u(x,\eta)|\lesssim h(\eta)^ag(\eta)^{1-a}$.
		Take the $L^2_\eta$ norm and use the  Hölder inequality: $\int h^{2a}g^{2(1-a)}\le(\int h^2)^a(\int g^2)^{1-a}$.
		The second statement follows from the Sobolev embedding $\dot H^{\kappa}(\mathbb R^m)\subset L^q$, $0<\kappa<\frac m2$.
	\end{proof}
	Pointwise decay in $y$ is lost: without radial symmetry in $y$ there is no $L^2\to L^\infty$ gain at the scale-invariant level, since $\int|\eta|^{-2\kappa}d\eta$ diverges at both ends.

	Conversely, if radial symmetry holds in $y$ but not in $x$, we maintain the pointwise decay in $y$ while sacrificing it in $x$. Instead, we obtain a weighted Lebesgue bound in the $x$-direction.
	\begin{theorem}[radial in $y$ only]\label{thm:yrad}
		Let $u$ be radial in $y$ and arbitrary in $x$. Let $0<s<\frac n2$, $b>0$, $a\in[0,1)$ with $\frac12<(1-a)r<\frac m2$, and
		$\frac1{p}=\frac12-\frac{as}n$. Then
		\[
		\sup_{y}\;|y|^{\frac m2-(1-a)r}\,\big\|\,|x|^{(1-a)b}u(\cdot,y)\big\|_{L^p(\mathbb R^n_x)}\lesssim M_1^{a}M_2^{1-a}.
		\]
	\end{theorem}
	\begin{proof}
		For $f=\tilde u(\cdot,\eta)$, the  Hölder inequality with $\frac1p=\frac a{p_0}+\frac{1-a}2$, $\frac1{p_0}=\frac12-\frac sn$, and the Sobolev inequality $\|f\|_{p_0}\lesssim\|D^sf\|_2$ give
		\[
		\||x|^{(1-a)b}f\|_p\le\|f\|_{p_0}^{a}\,\||x|^bf\|_2^{1-a}\lesssim\|D^sf\|_2^{a}\,\||x|^bf\|_2^{1-a}.
		\]
		By Lemma \ref{lem:kernel} in $y$ and Minkowski's inequality,
		$\||x|^{(1-a)b}u(\cdot,y)\|_p\lesssim\int F_y(\eta)h^ag^{1-a}d\eta$.
		Conclude as in the proof of Theorem \ref{thm:main-r}.
	\end{proof}
	Here no $x$-decay is available, only a weighted $L^p_x$ bound, while the full $|y|^{-B}$ decay is kept. The restriction $a\ge a_0$ is not needed, and $s\le\frac12$ is allowed.
	
	Since $\||x|^{b\tau}D_y^{r\tau}u\|_2\le\|u\|_2^{1-\tau}\||x|^bD_y^ru\|_2^{\tau}$ for $\tau\in[0,1]$ (the  Hölder inequalityin $(x,\eta)$ applied to $|x|^{2b\tau}|\eta|^{2r\tau}|\tilde u|^2$), Theorem \ref{thm:main-r} applied with $(b,r)$ replaced by $(b\tau,r\tau)$ gives the following.

	Returning to the fully radial setting, we can refine our main estimate by introducing an additional interpolation parameter. This three-norm inequality is particularly useful for balancing the contributions of the fractional Sobolev, weighted, and standard $L^2$ norms.
	\begin{theorem}\label{thm:three}
		Let $u$ be radial in $x$ and in $y$, $\frac12<s<\frac n2$, $b>0$, $\tau\in(0,1]$, and $a\in\big[\tfrac{2b\tau+1}{2(s+b\tau)},1\big]$ with $\frac12<(1-a)\tau r<\frac m2$. Then
		\[
		|x|^{A_\tau}|y|^{B_\tau}|u(x,y)|\lesssim \|D_x^su\|_2^{\,a}\,\||x|^bD_y^ru\|_2^{\,(1-a)\tau}\,\|u\|_2^{\,(1-a)(1-\tau)},
		\]
		\[
		A_\tau=\frac n2-as+(1-a)\tau b,\qquad B_\tau=\frac m2-(1-a)\tau r .
		\]
	\end{theorem}

	% Packages: amsmath, amssymb, amsthm.  Environments theorem, lemma, proposition, remark, proof assumed.
	
	\section{Compactness via Strauss' lemma}\label{sec-compact-struss}
	
	\paragraph{Fourier transform.} On $\mathbb R^d$ use the unitary convention
	$\hat f(\zeta)=(2\pi)^{-d/2}\int f(z)e^{-iz\cdot\zeta}dz$, so the Plancherel identity holds with constant $1$.
	$\mathcal F_x,\mathcal F_y$ denote partial transforms, $\tilde u(x,\eta)=\mathcal F_yu$, and
	$D^\tau$ is the multiplier $|\zeta|^\tau$. For $u\in L^2(\mathbb R^{n+m})$ put
	\[
	\|D_x^su\|_2=\||\xi|^s\hat u\|_{L^2},\qquad
	\||x|^bD_y^ru\|_2=\||x|^b|\eta|^r\tilde u(x,\eta)\|_{L^2(dx\,d\eta)} ,
	\]
	(finite or $+\infty$), and
	\[
	X=\{u\in L^2(\mathbb R^{n+m}):\ \|u\|_X:=\|D_x^su\|_2+\||x|^bD_y^ru\|_2+\|u\|_2<\infty\}.
	\]
	$u$ is \emph{radial in $x$} if $u(Ox,y)=u(x,y)$ a.e.\ for every $O\in O(n)$, and radial in $y$ similarly.
	$X_{\rm rad}:=\{u\in X:\ u \text{ radial in } x \text{ and in } y\}$.
	
	\paragraph{Hypotheses.}
	\begin{itemize}
		\item[(H)] $n\ge2,\ m\ge2,\ \tfrac12<s<\tfrac n2,\ r>\tfrac12,\ b>0$.
		\item[(E)] $\|u\|_{L^p}\le C_E\|u\|_X$ for all $u\in X$ and $p\in[2,Q]$, where
		\[
		Q=\frac{2(rn+m(b+s))}{m(b+s)+nr-2rs},\qquad\text{i.e.}\qquad \frac1Q=\frac12-\frac{rs}{nr+m(b+s)} .
		\]
	\end{itemize}
	Since $2rs<nr$, the denominator of $Q$ is positive and $2<Q<\infty$.
	Moreover $\frac1Q>\frac12-\frac sn=\frac1{p_0}$, where $p_0:=\frac{2n}{n-2s}$; hence $Q<p_0$.
	Only the case $p=Q$ of (E) will be used (the case $p=2$ is trivial).

	\begin{theorem}\label{thm:compact}
		Assume (H) and (E). For every $p\in(2,Q)$ the embedding $X_{\rm rad}\hookrightarrow L^p(\mathbb R^{n+m})$ is compact:
		every bounded sequence in $X_{\rm rad}$ has a subsequence converging in $L^p$.
	\end{theorem}
	
	\paragraph{Notation.} $C$ denotes constants depending only on $n,m,s,r,b,p,C_E$ (and, where indicated, on auxiliary cut-offs),
	never on the sequence, nor on $R,R',k$. $M$ denotes a bound for the $X$-norm of the functions in question.
	
Below, we provide some tools that are essential in the proof of main theorem.
	
%	\begin{lemma}[Strauss lemma, three-norm form]\label{lem:strauss}
%		Let $d\ge2$, $\tau>\frac12$, and let $f\in L^2(\mathbb R^d)$ be radial with $D^\tau f\in L^2$. Then for a.e.\ $x$,
%		\[
%		|f(x)|\le C(d,\tau)\,|x|^{-\frac{d-1}2}\,\|f\|_2^{1-\frac1{2\tau}}\,\|D^\tau f\|_2^{\frac1{2\tau}} .
%		\]
%	\end{lemma}
%	

	 Transitioning to the proof of compactness, we first need to ensure that our function space $X$ is well-behaved under smooth localization. The following lemma confirms that multiplying by standard cutoff functions preserves the relevant fractional Sobolev bounds.
	
	\begin{lemma}[multipliers]\label{lem:mult}
		Let $\theta\in\mathcal S(\mathbb R^d)$, $\sigma\ge0$, $\langle\zeta\rangle=(1+|\zeta|^2)^{1/2}$. Then
		$\|\langle\zeta\rangle^\sigma\widehat{\theta f}\|_2\le C(\theta,\sigma)\|\langle\zeta\rangle^\sigma\hat f\|_2$. Consequently, for $v\in L^2(\mathbb R^{n+m})$,
		multiplying by $\theta(x)$ (resp.\ $\theta(y)$) maps $\{\|D_x^sv\|_2<\infty\}$ (resp.\ $\{\|D_y^rv\|_2<\infty\}$) into itself with
		\[
		\|D_x^s(\theta v)\|_2\le C(\theta,s)\big(\|v\|_2+\|D_x^sv\|_2\big),\qquad
		\|D_y^r(\theta v)\|_2\le C(\theta,r)\big(\|v\|_2+\|D_y^rv\|_2\big).
		\]
		Multiplication by a bounded function of $y$ commutes with $D_x^s$ in the sense $\|D_x^s(\chi(y)v)\|_2\le\|\chi\|_\infty\|D_x^sv\|_2$, and multiplication by a bounded function of $x$ commutes with $D_y^r$: $D_y^r(\psi(x)v)=\psi(x)D_y^rv$.
	\end{lemma}
	
	\begin{proof}
		$\widehat{\theta f}=(2\pi)^{-d/2}\hat\theta*\hat f$ and $\langle\zeta\rangle^\sigma\le2^{\sigma/2}\langle\zeta-\zeta'\rangle^\sigma\langle\zeta'\rangle^\sigma$
		(from $1+|a+b|^2\le2(1+|a|^2)(1+|b|^2)$). Young's inequality $L^1*L^2\subset L^2$ gives the first claim, with
		$C=(2\pi)^{-d/2}2^{\sigma/2}\|\langle\cdot\rangle^\sigma\hat\theta\|_{L^1}$. The second follows slice-wise (the other variable is a parameter),
		using $|\zeta|^\sigma\le\langle\zeta\rangle^\sigma\le C_\sigma(1+|\zeta|^\sigma)$. The last two statements hold because
		$\mathcal F_x(\chi(y)v)=\chi(y)\mathcal F_xv$ and $\mathcal F_y(\psi(x)v)=\psi(x)\mathcal F_yv$.
	\end{proof}
	
	To extract convergent subsequences, we must establish the basic functional analytic properties of our working space. This lemma verifies that $X$ has a complete Hilbert structure and that radiality is preserved under weak limits.
	\begin{lemma} \label{lem:hilbert}
		$X$ is complete, equivalent to a Hilbert space, and $X_{\rm rad}$ is a closed (hence weakly closed) subspace. Every bounded sequence in $X_{\rm rad}$ has a subsequence that converges weakly in $X$ to some $u\in X_{\rm rad}$, and then also weakly in $L^2$.
	\end{lemma}
	
	\begin{proof}
		Define $Tu=(u,\ |\xi|^s\hat u,\ |x|^b|\eta|^r\tilde u)\in L^2(\mathbb R^{n+m})^3$. Then $\|Tu\|\le\|u\|_X\le\sqrt3\|Tu\|$.
		If $Tu_j\to(f_0,f_1,f_2)$, then $u_j\to f_0=:u$ in $L^2$, $\hat u_j\to\hat u$ and $\tilde u_j\to\tilde u$ in $L^2$ (unitarity), and
		along a subsequence a.e., so $f_1=|\xi|^s\hat u$ and $f_2=|x|^b|\eta|^r\tilde u$ a.e. Hence $T(X)$ is closed in a Hilbert space, so $X$ is complete.
		Radiality is preserved under $L^2$ limits, and $\|\cdot\|_2\le\|\cdot\|_X$, so $X_{\rm rad}$ is closed in $X$. Bounded sequences in a Hilbert space have weakly convergent subsequences,
		closed subspaces are weakly closed, and $u\mapsto\langle u,\phi\rangle_{L^2}$ is $X$-continuous.
	\end{proof}
	
	Our compactness argument relies on splitting the space into distinct geometric regions and controlling the $L^p$ norm in each. To systematically glue these bounds together, we formulate a convenient interpolation inequality that leverages the uniform control of the $X$-norm.
	\begin{lemma} \label{lem:interp}
		Fix $p_1\in(2,Q)$ and $p\in(2,Q)$. There is $\mu=\mu(p,p_1)\in(0,1]$ such that for every measurable $\Omega\subset\mathbb R^{n+m}$ and $v\in X$,
		\[
		\|v\|_{L^p(\Omega)}\le C\,M^{1-\mu}\|v\|_{L^{p_1}(\Omega)}^{\mu},\qquad M\ge\|v\|_X .
		\]
	\end{lemma}
	
	\begin{proof}
		If $p\le p_1$, interpolate $L^p$ between $L^2$ and $L^{p_1}$, and use $\|v\|_{L^2(\Omega)}\le M$.
		If $p>p_1$, interpolate between $L^{p_1}$ and $L^Q$ and use $\|v\|_{L^Q}\le C_EM$. The exponent on $\|v\|_{L^{p_1}(\Omega)}$ is positive in both cases.
	\end{proof}
	
	\paragraph{Choice of $p_1$.}
	Let $\sigma:=\min\{r,m/4\}\in(0,m/2)$ and $p_\sigma:=\frac{2m}{m-2\sigma}$. Fix once and for all
	\begin{equation}\label{p1}
		p_1:=2+\tfrac12\min\Big\{\,Q-2,\ \tfrac{4s}n,\ \tfrac{4r}m,\ p_\sigma-2\Big\}\in(2,Q).
	\end{equation}
	Thus $p_1\le 2+\frac{4s}n$, $p_1\le2+\frac{4r}m$ and $p_1\le p_\sigma$.

	Let $v\in X$ with $\|v\|_X\le M$.
		We now systematically bound the mass of our sequence in these different spatial regions. The first proposition controls the behavior near the singular set $\{x=0\}$, leveraging the fractional $x$-derivative independently of any radial symmetry.
	\begin{proposition}[near $\{x=0\}$; no symmetry needed]\label{prop:A}
		For $p\in(2,Q)$ let $\lambda\in(0,1)$ be defined by $\frac1p=\frac{1-\lambda}2+\frac\lambda Q$. Then for all $\varepsilon>0$,
		\[
		\|v\|_{L^2(\{|x|<\varepsilon\})}\le C\varepsilon^sM,\qquad
		\|v\|_{L^p(\{|x|<\varepsilon\})}\le C\varepsilon^{s(1-\lambda)}M .
		\]
	\end{proposition}
	
	\begin{proof}
		For a.e.\ $y$, $v(\cdot,y)\in L^2_x$ with $D^s_xv(\cdot,y)\in L^2_x$. By the  Hölder inequality and the Sobolev inequality,
		\[
		\int_{|x|<\varepsilon}|v|^2dx\le|\bb_\varepsilon|^{1-\frac2{p_0}}\|v(\cdot,y)\|_{p_0}^2\le C\varepsilon^{2s}\|D_x^sv(\cdot,y)\|_2^2 .
		\]
		Integrate in $y$ and use the Plancherel identity in $x$. For the second claim, the  Hölder inequality gives
		$\|v\|_{L^p(\Omega)}\le\|v\|_{L^2(\Omega)}^{1-\lambda}\|v\|_{L^Q(\Omega)}^\lambda$, and $\|v\|_Q\le C_EM$.
	\end{proof}

	Away from the origin, we can begin to exploit the radiality of the sequence. By applying our previously established radial estimates, we show that the $L^{p_1}$ norm decays uniformly as $|x| \to \infty$.
	\begin{proposition}[tail in $x$; uses radiality in $x$]\label{prop:B}
		Let $v\in X$ be radial in $x$, $\|v\|_X\le M$. Then for every $R>0$,
		\[
		\int_{\{|x|>R\}}|v|^{p_1}\,dx\,dy\le C\,M^{p_1}\,R^{-\kappa_1},\qquad
		\kappa_1=\tfrac{(n-1)(p_1-2)}2+b\,\alpha>0,\quad \alpha:=\tfrac{m(p_1-2)}{2r}\in(0,2].
		\]
	\end{proposition}
	
	\begin{proof}
		\emph{Step 1: decay of $g(x):=\|v(x,\cdot)\|_{L^2_y}$.}
		For a.e.\ $\eta$, $f_\eta:=\tilde v(\cdot,\eta)\in L^2_x$ is radial (a countable dense set of rotations suffices, by continuity of rotations in $L^2$),
		with $E_\eta=\|f_\eta\|_2$, $D_\eta=\|D^s f_\eta\|_2$, and by the Plancherel identity $\int E_\eta^2d\eta=\|v\|_2^2$, $\int D_\eta^2d\eta=\|D_x^sv\|_2^2$.
		Lemma \ref{lem:strauss} ($d=n\ge2$, $\tau=s>\frac12$) gives, for a.e.\ $(x,\eta)$ (Fubini),
		$|\tilde v(x,\eta)|\le C|x|^{-\frac{n-1}2}E_\eta^{1-\frac1{2s}}D_\eta^{\frac1{2s}}$. Hence for a.e.\ $x$, by the Plancherel identity in $y$ and the  Hölder inequality in $\eta$ with exponents $\frac{2}{2-1/s}$ and $2s$,
		\[
		g(x)^2=\int|\tilde v(x,\eta)|^2d\eta\le C|x|^{-(n-1)}\!\int E_\eta^{2-\frac1s}D_\eta^{\frac1s}d\eta
		\le C|x|^{-(n-1)}\|v\|_2^{2-\frac1s}\|D_x^sv\|_2^{\frac1s}.
		\]
		So $g(x)\le CM|x|^{-\frac{n-1}2}$ for a.e.\ $x$.
		
		\emph{Step 2: Gagliardo--Nirenberg in $y$.}
		Let $k(x):=\|D_y^rv(x,\cdot)\|_{L^2_y}$. Then $\int|x|^{2b}k(x)^2dx\le M^2$, so $k(x)<\infty$ for a.e.\ $x$, and
		$\int_{|x|>R}k^2dx\le R^{-2b}M^2$. Since $p_1\le p_\sigma$, Lemma \ref{lem:GN} with $d=m$, $\tau=\sigma<\rho=r$ gives for a.e.\ $x$ 	from \eqref{p1} that
		\[
		\int|v(x,y)|^{p_1}dy\le C\,k(x)^{\alpha}g(x)^{p_1-\alpha},\qquad \alpha=p_1\cdot\tfrac mr\big(\tfrac12-\tfrac1{p_1}\big)=\tfrac{m(p_1-2)}{2r}\in(0,2].
		\]

		\emph{Step 3.} Write $g^{p_1-\alpha}=g^{p_1-2}g^{2-\alpha}$ and use the  Hölder inequality with exponents $\frac2\alpha,\frac2{2-\alpha}$ (trivial if $\alpha=2$):
		\[
		\int_{|x|>R}\!\!\int|v|^{p_1}dy\,dx\le C\sup_{|x|>R}g^{p_1-2}\Big(\int_{|x|>R}k^2\Big)^{\frac\alpha2}\Big(\int g^2\Big)^{1-\frac\alpha2}
		\le C\big(MR^{-\frac{n-1}2}\big)^{p_1-2}\big(MR^{-b}\big)^{\alpha}M^{2-\alpha}.  
		\]
	\end{proof}
	
A similar tail estimate is needed in the $y$-direction. Provided we stay strictly away from $\{x=0\}$ to avoid the weight singularity, radial symmetry in $y$ guarantees uniform mass decay as $|y| \to \infty$.
	\begin{proposition}\label{prop:C}
		Let $v\in X$ be radial in $y$, $\|v\|_X\le M$, and $\varepsilon\in(0,1]$. Then there is $C_\varepsilon$ such that for all $R'>0$,
		\[
		\int_{\{|x|\ge\varepsilon,\ |y|>R'\}}|v|^{p_1}dx\,dy\le C_\varepsilon M^{p_1}R'^{-\frac{(m-1)(p_1-2)}2}.
		\]
	\end{proposition}
	
	\begin{proof}
		\emph{Step 0: cut-off.} Choose radial $\phi\in C_c^\infty(\mathbb R^n)$, $0\le\phi\le1$, $\phi=1$ on $\{|x|\le\frac12\}$, $\operatorname{supp}\phi\subset\{|x|<1\}$,
		and set $\phi_\varepsilon=\phi(\cdot/\varepsilon)$, $\psi_\varepsilon=1-\phi_\varepsilon$, $w:=\psi_\varepsilon(x)v$. Then $w=v$ on $\{|x|\ge\varepsilon\}$, $w=0$ on $\{|x|\le\varepsilon/2\}$, $w$ is radial in $y$, and
		\[
		\|w\|_2\le M,\qquad
		\|D_y^rw\|_2=\|\psi_\varepsilon D_y^rv\|_2\le(\varepsilon/2)^{-b}\||x|^bD_y^rv\|_2\le(\varepsilon/2)^{-b}M,
		\]
		\[
		\|D_x^sw\|_2\le\|D_x^sv\|_2+\|D_x^s(\phi_\varepsilon v)\|_2\le C_\varepsilon M\quad\text{(Lemma \ref{lem:mult})}.
		\]
		Here the weight $|x|^b$ has been converted into a plain $y$-derivative bound, at the price of a constant depending on $\varepsilon$.
		
		\emph{Step 1: decay of $g(y):=\|w(\cdot,y)\|_{L^2_x}$.}
		For a.e.\ $\xi$, $y\mapsto\mathcal F_xw(\xi,y)$ is radial in $y$ with $E_\xi=\|\mathcal F_xw(\xi,\cdot)\|_{L^2_y}$, $D_\xi=\|D_y^r\mathcal F_xw(\xi,\cdot)\|_{L^2_y}$, and $\int E_\xi^2=\|w\|_2^2$, $\int D_\xi^2=\|D_y^rw\|_2^2$.
		Lemma \ref{lem:strauss} ($d=m\ge2$, $\tau=r>\frac12$) and the  Hölder inequality in $\xi$ (as in Prop.\ \ref{prop:B}) give, for a.e.\ $y$,
		\[
		g(y)\le C|y|^{-\frac{m-1}2}\|w\|_2^{1-\frac1{2r}}\|D_y^rw\|_2^{\frac1{2r}}\le C_\varepsilon M|y|^{-\frac{m-1}2}.
		\]
		
		\emph{Step 2: Gagliardo--Nirenberg in $x$.}
		Let $h(y)=\|D_x^sw(\cdot,y)\|_{L^2_x}$, so $\int h^2dy=\|D_x^sw\|_2^2\le C_\varepsilon^2M^2$ and $\int g^2dy=\|w\|_2^2$.
		Put $e_1=\frac{n(p_1-2)}{2s}$. By \eqref{p1}, $e_1\le2$, and $\theta:=\frac ns(\frac12-\frac1{p_1})=\frac{e_1}{p_1}<1$, so $p_1<p_0$ and Lemma \ref{lem:GN} ($d=n$, $\tau=s$) gives
		$\int|w(x,y)|^{p_1}dx\le Ch^{e_1}g^{p_1-e_1}$ for a.e.\ $y$.
		
		\emph{Step 3.} Using the  Hölder inequality with exponents $\frac2{e_1},\frac2{2-e_1}$ on $g^{p_1-e_1}=g^{p_1-2}g^{2-e_1}$,
		\[
		\int_{|y|>R'}\!\!\int|w|^{p_1}dx\,dy\le C\sup_{|y|>R'}g^{p_1-2}\Big(\int h^2\Big)^{\frac{e_1}2}\Big(\int g^2\Big)^{1-\frac{e_1}2}
		\le C_\varepsilon M^{p_1}R'^{-\frac{(m-1)(p_1-2)}2}.
		\]
		Since $|v|=|w|$ on $\{|x|\ge\varepsilon\}$, the claim follows.
	\end{proof}
	
	The compact counterpart is formulated here.
	\begin{proposition}\label{prop:D}
		Let $(v_k)\subset X$ with $\|v_k\|_X\le M$ and $v_k\rightharpoonup0$ weakly in $L^2$. Let $0<\varepsilon<R<\infty$, $R'>0$, and
		$\Omega=\{\varepsilon\le|x|\le R,\ |y|\le R'\}$. Then $\|v_k\|_{L^p(\Omega)}\to0$ for every $p\in(2,Q)$.
	\end{proposition}
	
	\begin{proof}
		Choose $\psi\in C_c^\infty(\mathbb R^n\setminus\{0\})$ with $\psi=1$ on $\{\varepsilon\le|x|\le R\}$ and $\chi\in C_c^\infty(\mathbb R^m)$ with $\chi=1$ on $\{|y|\le R'\}$.
		Set $\Theta=\psi(x)\chi(y)$ and $w_k=\Theta v_k$, with compact support in $K=\operatorname{supp}\Theta$. Let $c_\psi=\min_{\operatorname{supp}\psi}|x|>0$.
		
		\emph{Uniform regularity.}
		$\|w_k\|_2\le\|\Theta\|_\infty M$. For $D_x^s$: $\|D_x^s(\chi v_k)\|_2\le\|\chi\|_\infty M$ and then, by Lemma \ref{lem:mult} (multiplication by $\psi(x)$),
		$\|D_x^sw_k\|_2\le C M$. For $D_y^r$: $\|D_y^r(\psi v_k)\|_2=\|\psi D_y^rv_k\|_2\le\|\psi\|_\infty c_\psi^{-b}M$ and $\|\psi v_k\|_2\le\|\psi\|_\infty M$, so by Lemma \ref{lem:mult} (multiplication by $\chi(y)$), $\|D_y^rw_k\|_2\le CM$. Hence, with $\zeta=(\xi,\eta)$,
		\[
		\int|\xi|^{2s}|\hat w_k|^2+\int|\eta|^{2r}|\hat w_k|^2\le CM^2,
		\qquad
		\int_{|\zeta|>2N}|\hat w_k|^2\le CM^2\big(N^{-2s}+N^{-2r}\big)
		\]
		(because $|\zeta|>2N$ forces $|\xi|>N$ or $|\eta|>N$).
		
		\emph{Low frequencies.} For each fixed $\zeta$, $\hat w_k(\zeta)=(2\pi)^{-\frac{n+m}2}\int v_k\,\Theta e^{-iz\cdot\zeta}dz\to0$, since $v_k\rightharpoonup0$ in $L^2$ and $\Theta e^{-iz\cdot\zeta}\in L^2$.
		Also $|\hat w_k(\zeta)|\le(2\pi)^{-\frac{n+m}2}\|\Theta\|_2M$. By dominated convergence on the ball $\{|\zeta|\le2N\}$, $\int_{|\zeta|\le2N}|\hat w_k|^2\to0$.
		
Finally, for every $N$, $\limsup_k\|w_k\|_2^2\le CM^2(N^{-2s}+N^{-2r})$; letting $N\to\infty$, $\|w_k\|_2\to0$.
		Since $\|w_k\|_Q\le\|\Theta\|_\infty C_EM$, interpolation gives $\|w_k\|_p\le\|w_k\|_2^{1-\lambda}\|w_k\|_Q^\lambda\to0$, and $\|v_k\|_{L^p(\Omega)}\le\|w_k\|_p$ because $\Theta=1$ on $\Omega$.
	\end{proof}
	
\begin{proof}[Proof of Theorem \ref{thm:compact}.]
		Let $(u_k)\subset X_{\rm rad}$, $\|u_k\|_X\le M_0$, and fix $p\in(2,Q)$.
	By Lemma \ref{lem:hilbert} we may pass to a subsequence with $u_k\rightharpoonup u$ weakly in $X$ and in $L^2$, where $u\in X_{\rm rad}$, $\|u\|_X\le M_0$ (weak lower semicontinuity of the norm in the Hilbert space $T(X)$, up to the equivalence constant $\sqrt3$; we simply set $M:=2\sqrt3M_0$).
	Put $v_k:=u_k-u$. Then $v_k\in X_{\rm rad}$, $\|v_k\|_X\le M$, and $v_k\rightharpoonup0$ in $L^2$. By (E), $u\in L^p$, so it suffices to show $\|v_k\|_{L^p}\to0$.
	
	Let $\mu=\mu(p,p_1)\in(0,1]$ be as in Lemma \ref{lem:interp} and $\lambda\in(0,1)$ as in Proposition \ref{prop:A}. Given $\delta>0$,
	  choose $\varepsilon\in(0,1]$ with $C\varepsilon^{s(1-\lambda)}M\le\delta$. By Prop.\ \ref{prop:A}, $\|v_k\|_{L^p(\Omega_1)}\le\delta$ for all $k$, where $\Omega_1=\{|x|<\varepsilon\}$.
	  Choose $R>\varepsilon$ with $CM^{1-\mu}(MR^{-\kappa_1/p_1}\!\cdot C)^{\mu}\le\delta$. By Prop.\ \ref{prop:B} ($v_k$ is radial in $x$) and Lemma \ref{lem:interp}, $\|v_k\|_{L^p(\Omega_2)}\le\delta$ for all $k$, where $\Omega_2=\{|x|>R\}$.
  With this $\varepsilon$, choose $R'$ so large that $CM^{1-\mu}\big(C_\varepsilon M^{p_1}R'^{-(m-1)(p_1-2)/2}\big)^{\mu/p_1}\le\delta$. By Proposition  \ref{prop:C} ($v_k$ is radial in $y$) and Lemma \ref{lem:interp}, $\|v_k\|_{L^p(\Omega_3)}\le\delta$ for all $k$, where $\Omega_3=\{|x|\ge\varepsilon,\ |y|>R'\}$.
	  By Proposition \ref{prop:D}, $\|v_k\|_{L^p(\Omega_4)}\to0$ where $\Omega_4=\{\varepsilon\le|x|\le R,\ |y|\le R'\}$.
 
	Since $\mathbb R^{n+m}=\Omega_1\cup\Omega_2\cup\Omega_3\cup\Omega_4$,
	\[
	\limsup_{k\to\infty}\|v_k\|_{L^p(\mathbb R^{n+m})}\le\sum_{j=1}^4\limsup_k\|v_k\|_{L^p(\Omega_j)}\le3\delta .
	\]
	As $\delta>0$ is arbitrary, $v_k\to0$ in $L^p$, i.e.\ $u_k\to u$ in $L^p$.  
\end{proof}

	%==========================================================================
	\section{The weighted Gagliardo--Nirenberg inequality of Grushin type \\ under radial symmetry } \label{sec-radialgrushin}
	%==========================================================================

	Throughout  this section, $m,n\ge1$ are integers, the variable in $\R^n\times\R^m$ is written $(x,y)$, and the exponents $s\in(0,1]$, $r>0$ and the weights $b,\beta\in\R$ are fixed. For a function $f$ on $\R^d$ the Fourier transform is $\widehat f(\xi)=(2\pi)^{-d/2}\int e^{-ix\cdot\xi}f(x)\,dx$, and $D^\sigma$ is the Fourier multiplier with symbol $|\xi|^\sigma$. The operators $D_x^s$ and $D_y^r$ are the multipliers with symbols $|\xi|^s$ and $|\eta|^r$ acting through the partial Fourier transform in $x$ and in $y$, respectively. For $u\in\Sch(\R^{n+m})$ we set
	\begin{equation}\label{s3eq:defAC}
		\cA(u)=\big\||x|^{b}D_y^{r}u\big\|_{L^2(\R^{n+m})},\qquad
		\cC(u)=\big\|D_x^{s}u\big\|_{L^2(\R^{n+m})}.
	\end{equation}
	All norms are understood in $[0,\infty]$. We introduce the quantities
	\begin{equation}\label{s3eq:defparam}
		\varrho=s+\beta,\quad \Gamma=rn+m(b+s),\quad
		\delta=\frac{r\varrho}{\Gamma},\quad
		\tau=\frac{m\varrho}{\Gamma}=\frac{m\delta}{r},\quad
		B=\frac{2\Gamma}{\Gamma-2r\varrho}=\frac{2}{1-2\delta},
	\end{equation}
	so that $B=\dfrac{2\,(rn+m(b+s))}{m(b+s)+r(n-2s-2\beta)}$. The inequality under study is
	\begin{equation}\label{s3eq:main}
		\big\||x|^{\beta}u\big\|_{L^{B}(\R^{n+m})}\le C\,\cA(u)^{\tau}\,\cC(u)^{1-\tau}.
	\end{equation}
	
	We shall consider four classes of functions. The class $\Sch$ consists of all Schwartz functions on $\R^{n+m}$. The class $\Sch_x$ consists of those $u\in\Sch$ that are radial in $x$, that is $u(Qx,y)=u(x,y)$ for all $Q\in O(n)$ and all $(x,y)$. The class $\Sch_y$ consists of those $u\in\Sch$ that are radial in $y$, that is $u(x,Py)=u(x,y)$ for all $P\in O(m)$. The class $\Sch_{xy}=\Sch_x\cap\Sch_y$ consists of the functions radial in both variables. Note that radial in $x$ with $n=1$ means even in $x$.
	
	We will use the following hypotheses on the parameters.
	\begin{gather}
		2s<n, \tag{H0}\label{s3H0}\\
		s+\beta>0, \tag{H1}\label{s3H1}\\
		\Gamma>m(s+\beta), \tag{H2}\label{s3H2}\\
		\Gamma>2r(s+\beta), \tag{H3}\label{s3H3}\\
		\beta\,(rn+ms)\le ms\,b, \tag{H4}\label{s3H4}\\
		\beta\,(r+ms)\le s\big(mb+(n-1)r\big), \tag{R1}\label{s3R1}\\
		\beta\,(m+2r)<mb+r(n-2s). \tag{R2}\label{s3R2}
	\end{gather}
	Hypotheses \eqref{s3H1}--\eqref{s3H3} express that $B\in(2,\infty)$ and $\tau\in(0,1)$ (see Lemma~\ref{s3lem:param} below); \eqref{s3H0} is a limitation of the method; \eqref{s3H4} is the condition $\beta\le\tau b$, and it is the one that is modified by radial symmetry in $x$.

	We begin by formulating our main result, which establishes the valid parameter regimes for the weighted Grushin-type Gagliardo--Nirenberg inequality both in the general non-symmetric case and under partial radial symmetry in the $x$-variable.
	\begin{theorem}\label{s3thm:main}
		Let $m,n\ge1$, $s\in(0,1]$, $r>0$, $b,\beta\in\R$, and assume \eqref{s3H0}--\eqref{s3H3}. Then $\Gamma>0$, $2<B<\infty$ and $0<\tau<1$, and there is a constant $C=C(m,n,s,r,b,\beta)$ such that the following two statements hold.
		
		\smallskip
		\noindent\emph{(a) (no symmetry).} If \eqref{s3H4} holds, then \eqref{s3eq:main} holds for all $u\in\Sch$.
		
		\smallskip
		\noindent\emph{(b) (radial in $x$).} If $n\ge2$ and both \eqref{s3R1} and \eqref{s3R2} hold, then \eqref{s3eq:main} holds for all $u\in\Sch_x$, and in particular for all $u\in\Sch_{xy}$.
	\end{theorem}
	
	Part (a) is the result obtained earlier without symmetry; it is recalled here because its proof is the same as that of part (b), with the classical Hardy--Sobolev inequality in place of its radial extension. The point of the present note is part (b) and the optimality statements below.

	To clarify the analytical advantage gained from radial symmetry, we explicitly compare the improved threshold for the weight exponent $\beta$ against the standard non-symmetric bound.
	\begin{remark} \label{s3rem:gain}
		Let $n\ge2$. Then \eqref{s3H4} implies \eqref{s3R1} and \eqref{s3R2} (Lemma~\ref{s3lem:param}\,(vi)), so (b) is an improvement of (a). The threshold for $\beta$ in \eqref{s3H4} is $\frac{msb}{ms+rn}$, while the threshold in \eqref{s3R1} is $\frac{s(mb+(n-1)r)}{r+ms}$, and a direct computation gives
		\[
		\frac{s\,(mb+(n-1)r)}{r+ms}-\frac{msb}{ms+rn}=\frac{(n-1)\,s\,r\,\Gamma}{(r+ms)(ms+rn)}>0 .
		\]
		For $n=1$ the two thresholds coincide, which is natural: radial means even in that case, and an even function can carry two symmetric bumps at $\pm x_0$, which reproduces the translation obstruction. The values of $B$ and $\tau$, and the conditions \eqref{s3H1}--\eqref{s3H3}, are not affected by radial symmetry.
	\end{remark}

	To complement the sufficiency result of Theorem~\ref{s3thm:main}, we state an optimality proposition showing that the scaling relations, parameter regimes, and symmetry conditions are sharp.
	\begin{proposition}[Optimality]\label{s3prop:opt}
		Let $m,n\ge1$, $s\in(0,1]$, $r>0$ and $b,\beta\in\R$ with $\Gamma>0$. Let $\mathcal Y$ be one of the classes $\Sch$, $\Sch_x$, $\Sch_y$, $\Sch_{xy}$, and suppose that there are $B\in[1,\infty)$, $\tau\in(0,1]$ and $K<\infty$ such that
		\begin{equation}\label{s3eq:ineqK}
			\||x|^\beta u\|_{B}\le K\,\cA(u)^{\tau}\,\cC(u)^{1-\tau}\qquad\text{for all }u\in\mathcal Y\text{ with }\operatorname{supp}u\subset\{x\neq0\}\text{ compact}.
		\end{equation}
		Then the following three statements hold.
		
		\smallskip
		\noindent\emph{(a)} $B$ and $\tau$ are given by \eqref{s3eq:defparam}.
		
		\smallskip
		\noindent\emph{(b)} If $\mathcal Y=\Sch$ or $\mathcal Y=\Sch_y$, then $\tau<1$ and $\beta\le\tau b$, that is \eqref{s3H4} holds.
		
		\smallskip
		\noindent\emph{(c)} If $\mathcal Y=\Sch_x$ or $\mathcal Y=\Sch_{xy}$, then $\tau<1$ and \eqref{s3R1} holds.
		
		\smallskip
		\noindent\emph{(d)} In all four classes, \eqref{s3H1}, \eqref{s3H2} and \eqref{s3H3} hold.
	\end{proposition}
	
	The assumption $\tau\le1$ is made because for $\tau>1$ the right-hand side of \eqref{s3eq:ineqK} contains a negative power of $\cC(u)$, and the upper bounds for $\cC$ used below would not be applicable. In the classes without radial symmetry in $x$ this case is excluded by the modulation $e^{i\lambda x_1}\phi(x)$ with $\lambda\to\infty$, which leaves $\||x|^\beta u\|_B$ and $\cA(u)$ unchanged and makes $\cC(u)\to\infty$.

	By combining the affirmative bounds of Theorem~\ref{s3thm:main} with the sharp necessary conditions of Proposition~\ref{s3prop:opt}, we obtain complete characterizations across the various symmetry classes.
	\begin{corollary}[Consequences]\label{s3cor:opt}
		Assume \eqref{s3H0}--\eqref{s3H3}.
		
		\smallskip
		\noindent\emph{(1) Radial in $y$ only.} In the class $\Sch_y$ the inequality \eqref{s3eq:main} holds if and only if \eqref{s3H4} holds. Thus radial symmetry in $y$ alone gives no improvement whatsoever.
		
		\smallskip
		\noindent\emph{(2) Radial in $x$, and in both variables.} Let $n\ge2$ and $s<\frac12$. In the classes $\Sch_x$ and $\Sch_{xy}$ the inequality \eqref{s3eq:main} holds if and only if \eqref{s3R1} holds; in particular the range of Theorem~\ref{s3thm:main}\,(b) is sharp, and radial symmetry in $y$ in addition to radial symmetry in $x$ gives no further improvement. Together with Proposition~\ref{s3prop:opt}\,(a),(d), this means that, for $\Gamma>0$, $s<\frac12$ and $n\ge2$, the inequality \eqref{s3eq:main} with some $B\in[1,\infty)$ and $\tau\in(0,1]$ holds in $\Sch_x$ (or in $\Sch_{xy}$) if and only if $B,\tau$ are given by \eqref{s3eq:defparam} and \eqref{s3H1}--\eqref{s3H3} and \eqref{s3R1} hold. For $s\ge\frac12$, \eqref{s3R1} is necessary and \eqref{s3R2} together with \eqref{s3R1} is sufficient; whether \eqref{s3R2} is necessary is not decided here.
	\end{corollary}
	
	For $s<\frac12$, hypothesis \eqref{s3R1} implies \eqref{s3R2}, and for $s\ge\frac12$ hypothesis \eqref{s3R2} implies \eqref{s3R1} (Lemma~\ref{s3lem:param}\,(vi)); this is why only one of the two conditions appears in the discussion above.

	%==========================================================================
 
	%==========================================================================

 Moving toward the proof of Theorem~\ref{s3thm:main}, we first establish a reduction identity that allows us to control the full $x$-fractional Sobolev norm through the $H^s$-norm of the partial $L^2_y$-profile.
	
	\begin{lemma}[Gagliardo identity and reduction to $L^2$ in $y$]\label{s3lem:diam}
		\emph{(a)} Let $0<s<1$ and $f\in L^2(\R^n)$. Then
		\[
		\int_{\R^n}\!|\xi|^{2s}|\widehat f(\xi)|^2d\xi=\kappa_{n,s}^{-1}\iint\frac{|f(x+h)-f(x)|^2}{|h|^{n+2s}}\,dx\,dh,
		\qquad\kappa_{n,s}=\int_{\R^n}\!\frac{|e^{ih_1}-1|^2}{|h|^{n+2s}}\,dh\in(0,\infty),
		\]
		with both sides possibly infinite.
		
		\smallskip
		\noindent\emph{(b)} Let $0<s\le1$, $u\in\Sch(\R^{n+m})$ and $U_2(x)=\|u(x,\cdot)\|_{L^2(\R^m)}$. Then $U_2\in H^s(\R^n)$ and
		\[
		\|D^sU_2\|_{L^2(\R^n)}\le\|D_x^su\|_{L^2(\R^{n+m})}.
		\]
		If $u\in\Sch_x$ then $U_2$ is radial.
	\end{lemma}
	
	\begin{proof}
		(a) By the Plancherel identity and Tonelli, the double integral equals $\int|\widehat f(\xi)|^2\big(\int|e^{ih\cdot\xi}-1|^2|h|^{-n-2s}dh\big)d\xi$. By rotation invariance and the substitution $h\mapsto h/|\xi|$ the inner integral equals $\kappa_{n,s}|\xi|^{2s}$, and $\kappa_{n,s}$ is finite and positive because $|e^{ih_1}-1|^2\lesssim\min(|h|^2,1)$ and $0<s<1$.
		
		(b) Note that $\|U_2\|_2=\|u\|_2$. Assume first $0<s<1$. Apply (a) to $u(\cdot,y)$ for each $y$, integrate in $y$, and use the Plancherel identity in the $x$-variable (for each fixed $y$) and Tonelli:
		\[
		\kappa_{n,s}\|D_x^su\|_2^2=\iint\big\|u(x+h,\cdot)-u(x,\cdot)\big\|_{L^2(\R^m)}^2|h|^{-n-2s}\,dx\,dh .
		\]
		By the triangle inequality in $L^2(\R^m)$ we have $|U_2(x+h)-U_2(x)|\le\|u(x+h,\cdot)-u(x,\cdot)\|_{L^2(\R^m)}$, hence by (a) applied to $U_2$, $\kappa_{n,s}\|D^sU_2\|_2^2\le\kappa_{n,s}\|D_x^su\|_2^2$.
		
		Assume now $s=1$. Since $u\in\Sch$, the map $x\mapsto u(x,\cdot)\in L^2(\R^m)$ is $C^1$ with bounded derivative, so $U_2$ is Lipschitz and differentiable almost everywhere. At a point of differentiability and for a unit vector $e$,
		\[
		|\partial_eU_2(x)|\le\lim_{t\to0}\Big\|\frac{u(x+te,\cdot)-u(x,\cdot)}{t}\Big\|_{L^2(\R^m)}=\|\partial_eu(x,\cdot)\|_{L^2(\R^m)}\le\big\||\nabla_xu|(x,\cdot)\big\|_{L^2(\R^m)} .
		\]
		Taking the supremum over $e$ and integrating in $x$ gives $\|\nabla U_2\|_2\le\|\nabla_xu\|_2$. A Lipschitz function in $L^2$ with gradient in $L^2$ belongs to $H^1$, and $\|D^1f\|_2=\|\nabla f\|_2$ by the Plancherel identity.
		
		Finally, if $u(Qx,y)=u(x,y)$ for all $Q\in O(n)$, then $U_2(Qx)=U_2(x)$.
	\end{proof}

	The next result is the radial extension of Lemma~\ref{lem:HS}. It allows $a<0$, that is, growing weights on the left-hand side. It is the radial weighted Sobolev inequality of \cite{DrelichmanDuran11,DDD12}, where it is used together with a Strauss-type lemma for radial functions.

	The core estimate needed to exploit radial symmetry in $x$ is the following weighted Sobolev inequality, which accommodates growing weights for radially symmetric functions.
	\begin{lemma}[Radial weighted Sobolev inequality {\cite{DrelichmanDuran11,DDD12}}]\label{s3lem:radSob}
		Let $d\ge2$, $0<\sigma<\frac d2$, $c>-2\sigma$ with $c(1-2\sigma)\le2\sigma(d-1)$, and $q=\frac{2(d+c)}{d-2\sigma}$ (so that $q>2$). Then
		\[
		\big\||x|^{c/q}u\big\|_{L^q(\R^d)}\le C\,\|D^{\sigma}u\|_{L^2(\R^d)}\qquad\text{for every radial }u\in C_c^\infty(\R^d).
		\]
	\end{lemma}
	
	We now rewrite this inequality in the parametrization used in the present note and extend it by density from radial $C_c^\infty$ to radial $H^s$.

	Rewriting this estimate in our parameterization and extending it by density to $H^s(\mathbb{R}^n)$, we deduce a general radial Hardy--Sobolev inequality tailored to our setting.
	\begin{corollary}[Radial Hardy--Sobolev inequality]\label{s3cor:HSrad}
		Let $n\ge2$, $0<s<\frac n2$, and let $a\in\R$ satisfy
		\begin{equation}\label{s3eq:arange}
			s-\frac n2<a<s\qquad\text{and}\qquad a\ge-(n-1)s .
		\end{equation}
If $\frac1g=\frac12-\frac{s-a}{n}$, then $g\in(2,\infty)$ and
		\[
		\big\||x|^{-a}f\big\|_{L^g(\R^n)}\le C\,\|D^sf\|_{L^2(\R^n)}\qquad\text{for every radial }f\in H^s(\R^n).
		\]
	\end{corollary}
	
	\begin{proof}
		We have $\frac1g\in(0,\frac12)$ by \eqref{s3eq:arange}, so $g\in(2,\infty)$ and $n-2s+2a>0$; in fact $g=\frac{2n}{n-2s+2a}$. Apply Lemma~\ref{s3lem:radSob} with $d=n$, $\sigma=s$, $q=g$ and $c=-ag$, so that $|x|^{c/q}=|x|^{-a}$. We check the hypotheses.
		
		First, the relation $g=\frac{2(n+c)}{n-2s}$ holds. Indeed, from $g(n-2s+2a)=2n$ we get $ag=\frac12\big(2n-g(n-2s)\big)$, hence $n+c=n-ag=\frac12g(n-2s)$.
		
		Second, $c>-2s$, which means $ag<2s$. Since $g=\frac{2n}{n-2s+2a}>0$, the inequality $ag<2s$ is equivalent to $na<s(n-2s+2a)$, that is $a(n-2s)<s(n-2s)$, that is $a<s$.
		
		Third, $c(1-2s)\le2s(n-1)$. We distinguish cases. If $c\le0$ and $s\le\frac12$, the left-hand side is $\le0$. If $c\le0$ and $s>\frac12$, then $|c|<2s$ by the second point, so $c(1-2s)=|c|(2s-1)<2s(2s-1)<2s(n-1)$, because $2s<n$. If $c>0$ and $s\ge\frac12$, the left-hand side is $\le0$. Finally, let $c>0$ (that is $a<0$) and $s<\frac12$, and write $\alpha=-a>0$. Then $c=\alpha g$ with $g=\frac{2n}{n-2s-2\alpha}$, and multiplying the inequality by $n-2s-2\alpha>0$ it becomes
		\begin{align*}
			2n\alpha(1-2s)\le2s(n-1)(n-2s-2\alpha)
			&\iff\alpha\big(n-2ns+2s(n-1)\big)\le s(n-1)(n-2s)\\
			&\iff\alpha(n-2s)\le s(n-1)(n-2s),
		\end{align*}
		that is $\alpha\le(n-1)s$, which is the assumption $a\ge-(n-1)s$.
		
		Thus Lemma~\ref{s3lem:radSob} applies, and gives the inequality for radial $f\in C_c^\infty(\R^n)$. To pass to radial $f\in H^s(\R^n)$ we show that radial functions in $C_c^\infty$ are dense in the radial functions of $H^s$. Let $\eta_R(x)=\eta(|x|/R)$ with a radial cut-off $\eta\in C^\infty$, $\eta=1$ on $[0,\frac12]$, $\eta=0$ on $[1,\infty)$. We claim $\eta_Rf\to f$ in $H^s$ as $R\to\infty$. Clearly $\|(1-\eta_R)f\|_2\to0$. For $0<s<1$, by Lemma~\ref{s3lem:diam}\,(a) it suffices to treat the seminorm of $(1-\eta_R)f$, and
		\[
		|(1-\eta_R)f(x)-(1-\eta_R)f(z)|\le|f(x)-f(z)|\,|1-\eta_R(x)|+|f(z)|\,|\eta_R(x)-\eta_R(z)| .
		\]
		The integral of the square of the first term against $|x-z|^{-n-2s}$ tends to zero by dominated convergence, since $1-\eta_R(x)\to0$ pointwise and $|f(x)-f(z)|^2|x-z|^{-n-2s}\in L^1$. The second term gives
		\begin{multline*}
			\iint\frac{|f(z)|^2|\eta_R(x)-\eta_R(z)|^2}{|x-z|^{n+2s}}\,dx\,dz\le\|f\|_2^2\int\frac{\min(\|\eta_R'\|_\infty^2|h|^2,1)}{|h|^{n+2s}}\,dh\\
			=C\|\eta_R'\|_\infty^{2s}\|f\|_2^2\le CR^{-2s}\|f\|_2^2\to0 .
		\end{multline*}
		For $s=1$ the claim follows from $\nabla((1-\eta_R)f)=(1-\eta_R)\nabla f-f\nabla\eta_R$. After the cut-off we mollify with a radial mollifier; $\rho_\epsilon*(\eta_Rf)\to\eta_Rf$ in $H^s$ by the Plancherel identity and dominated convergence, and the result is radial, smooth and compactly supported. Hence radial $C_c^\infty$ functions are dense in the radial functions of $H^s$. If $f_k\to f$ in $H^s$ with $f_k$ radial in $C_c^\infty$ and (after a subsequence) almost everywhere, Fatou's lemma gives $\||x|^{-a}f\|_g\le\liminf_k\||x|^{-a}f_k\|_g\le C\lim_k\|D^sf_k\|_2=C\|D^sf\|_2$.
	\end{proof}
	%==========================================================================
 
	%==========================================================================
	Before embarking on the main proof, we establish the algebraic properties of our scaling parameters and prove key equivalences between the parameter conditions and weight signs.
	\begin{lemma}\label{s3lem:param}
		Assume \eqref{s3H1}--\eqref{s3H3}. Then $\Gamma>m\varrho>0$, $0<\delta<\frac12$, $0<\tau<1$ and $2<B<\infty$. Define
		\begin{equation}\label{s3eq:defa}
			a=\frac{\tau b-\beta}{1-\tau},\qquad\frac1g=\frac12-\frac{\delta}{1-\tau}\in\R .
		\end{equation}
		Then the following statements hold.
		\begin{gather}
			\beta=\tau b-(1-\tau)a,\qquad (1-\tau)(s-a)=n\delta, \tag{i}\\
			\frac1B=\frac\tau2+\frac{1-\tau}{g},\qquad\frac1g=\frac12-\frac{s-a}{n}, \tag{ii}\\
			a<s, \tag{iii}\\
			a\ge0\iff\text{\eqref{s3H4}}, \tag{iv}\\
			a\ge-(n-1)s\iff\text{\eqref{s3R1}},\qquad \frac1g>0\iff a>s-\frac n2\iff\text{\eqref{s3R2}}, \tag{v}\\
			\text{\eqref{s3H4}}\Rightarrow\text{\eqref{s3R1} and \eqref{s3R2}},\quad s<\tfrac12\ \&\ \text{\eqref{s3H0}}\Rightarrow(\text{\eqref{s3R1}}\Rightarrow\text{\eqref{s3R2}}),\quad s\ge\tfrac12\ \&\ \text{\eqref{s3H0}}\Rightarrow(\text{\eqref{s3R2}}\Rightarrow\text{\eqref{s3R1}}). \tag{vi}
		\end{gather}
	\end{lemma}
	
	\begin{proof}
		By \eqref{s3H1} and \eqref{s3H2}, $\Gamma>m\varrho>0$, hence $\tau=m\varrho/\Gamma\in(0,1)$ and $\delta=r\tau/m>0$; \eqref{s3H3} is equivalent to $\delta<\frac12$, and $B=2/(1-2\delta)\in(2,\infty)$.
		
		The first identity in (i) is the definition of $a$. For the second, using $\tau\Gamma=m\varrho$ and $\Gamma-m(b+s)=rn$,
		\[
		(1-\tau)(s-a)=s(1-\tau)-\tau b+\beta=\varrho-\tau(b+s)=\varrho\,\frac{\Gamma-m(b+s)}{\Gamma}=\frac{\varrho\,rn}{\Gamma}=n\delta .
		\]
		For (ii), $\frac1B=\frac12-\delta$, and therefore $\frac\tau2+\frac{1-\tau}g=\frac\tau2+\frac{1-\tau}2-\delta=\frac12-\delta=\frac1B$; the second identity in (ii) is the second identity in (i) inserted in \eqref{s3eq:defa}. Statement (iii) follows from (i) since $\delta>0$ and $1-\tau>0$.
		
		For (iv), $a\ge0$ means $\tau b\ge\beta$, that is $m\varrho b\ge\beta\Gamma$, that is $msb+m\beta b\ge\beta rn+\beta mb+\beta ms$, which is \eqref{s3H4}.
		
		For the first equivalence in (v), note that $(1-\tau)\Gamma=\Gamma-m\varrho=rn+mb-m\beta>0$. Multiplying $a+(n-1)s\ge0$ by this positive number and using $a(1-\tau)=\tau b-\beta$ and $\tau\Gamma=m\varrho$ we obtain
		\[
		\begin{aligned}
			(a+(n-1)s)(1-\tau)\Gamma&=m\varrho b-\beta\Gamma+(n-1)s(rn+mb-m\beta)\\
			&=msb-\beta rn-\beta ms+(n-1)s\,rn+(n-1)s\,mb-(n-1)s\,m\beta\\
			&=n\Big[s\big(mb+(n-1)r\big)-\beta(r+ms)\Big],
		\end{aligned}
		\]
		so $a\ge-(n-1)s$ is equivalent to \eqref{s3R1}. For the second statement in (v), $\frac1g>0$ is equivalent to $2\delta<1-\tau$, that is $2r\varrho<\Gamma-m\varrho$, that is $(m+2r)(s+\beta)<rn+mb+ms$, that is \eqref{s3R2}; and $\frac1g>0$ is equivalent to $s-a<\frac n2$ by (ii).
		
		For (vi), if \eqref{s3H4} holds then $a\ge0>s-\frac n2$ by \eqref{s3H0} (so \eqref{s3R2} holds) and $a\ge0\ge-(n-1)s$ (so \eqref{s3R1} holds). If $s<\frac12$, then $-(n-1)s-(s-\frac n2)=n(\frac12-s)>0$, so $a\ge-(n-1)s$ implies $a>s-\frac n2$. If $s\ge\frac12$, then $(s-\frac n2)+(n-1)s=n(s-\frac12)\ge0$, so $a>s-\frac n2$ implies $a>-(n-1)s$.
	\end{proof}
	
	%==========================================================================
 \bigskip
	%==========================================================================
With our analytical preliminaries and parameter equivalences established, we proceed to prove the main theorem by combining a fiberwise Gagliardo--Nirenberg estimate in $y$ with a weighted Sobolev step in $x$.

	\begin{proof}[Proof of Theorem~\ref{s3thm:main}.]
	
	The first assertions ($\Gamma>0$, $2<B<\infty$, $0<\tau<1$) are contained in Lemma~\ref{s3lem:param}. Let $u\in\Sch(\R^{n+m})$, $u\ne0$. Then $\cA(u),\cC(u)>0$, since $D_y^ru=0$ or $D_x^su=0$ would force $\widehat u=0$ almost everywhere. If $\cA(u)=\infty$ there is nothing to prove, so we assume $\cA=\cA(u)<\infty$ and write $\cC=\cC(u)$. We use the numbers $a$ and $g$ of \eqref{s3eq:defa}. In part (a) we assume \eqref{s3H4}, so that $a\ge0$ by Lemma~\ref{s3lem:param}\,(iv), and $0\le a<s$ by (iii). In part (b) we assume $n\ge2$, \eqref{s3R1} and \eqref{s3R2}; then by Lemma~\ref{s3lem:param}\,(iii) and (v), $s-\frac n2<a<s$ and $a\ge-(n-1)s$, that is, \eqref{s3eq:arange} holds. In both parts $\frac1g=\frac12-\frac{s-a}n\in(0,\frac12)$, so $2<g<\infty$.

	\paragraph{Step 1: Gagliardo--Nirenberg inequality in $y$.}
	For $x\in\R^n$ let
	\[
	F(x)=\|D_y^ru(x,\cdot)\|_{L^2(\R^m)},\qquad U_2(x)=\|u(x,\cdot)\|_{L^2(\R^m)},\qquad U_B(x)=\|u(x,\cdot)\|_{L^B(\R^m)} .
	\]
	By Tonelli, $\||x|^bF\|_{L^2(\R^n)}=\cA$ and $\||x|^\beta U_B\|_{L^B(\R^n)}=\||x|^\beta u\|_{L^B(\R^{n+m})}$. For fixed $x$ we apply Lemma~\ref{lem:GN} to $u(x,\cdot)\in\Sch(\R^m)$ with $d=m$, $\sigma=r$, $p=B$. The hypotheses hold: $B\in(2,\infty)$, and $\frac1B>\frac12-\frac rm$ is equivalent to $\delta<\frac rm$, that is to $\tau<1$. The exponent is $\nu=\frac mr(\frac12-\frac1B)=\frac{m\delta}r=\tau$. Hence
	\begin{equation}\label{s3eq:step1}
		U_B(x)\le C\,F(x)^{\tau}\,U_2(x)^{1-\tau}\qquad\text{for every }x\in\R^n.
	\end{equation}
	
	\paragraph{Step 2: Hölder's inequality in $x$.}
	Multiplying \eqref{s3eq:step1} by $|x|^\beta$ and using $\beta=\tau b-(1-\tau)a$ (Lemma~\ref{s3lem:param}\,(i)), we obtain for $x\ne0$
	\[
	|x|^{\beta}U_B(x)\le C\,\big(|x|^{b}F(x)\big)^{\tau}\,\big(|x|^{-a}U_2(x)\big)^{1-\tau}.
	\]
	Since $\frac1B=\frac\tau2+\frac{1-\tau}g$ by Lemma~\ref{s3lem:param}\,(ii), Hölder's inequality with the exponents $\frac2\tau$ and $\frac g{1-\tau}$ yields
	\begin{equation}\label{s3eq:step2}
		\||x|^{\beta}u\|_{L^B(\R^{n+m})}\le C\,\||x|^bF\|_{L^2(\R^n)}^{\tau}\,\||x|^{-a}U_2\|_{L^g(\R^n)}^{1-\tau}=C\,\cA^{\tau}\,\||x|^{-a}U_2\|_{L^g(\R^n)}^{1-\tau}.
	\end{equation}
	Here $a$ may be negative in part (b); then $|x|^{-a}=|x|^{|a|}$ is a growing weight, and nothing in this step depends on the sign of $a$.
	
	\paragraph{Step 3: Hardy--Sobolev inequality in $x$.}
	By Lemma~\ref{s3lem:diam}\,(b), $U_2\in H^s(\R^n)$ and $\|D^sU_2\|_2\le\cC$. In part (a) we apply Lemma~\ref{lem:HS}, whose hypotheses ($0\le a<s<\frac n2$) hold, to obtain
	\begin{equation}\label{s3eq:step3}
		\||x|^{-a}U_2\|_{L^g(\R^n)}\le C\,\|D^sU_2\|_{L^2(\R^n)}\le C\,\cC .
	\end{equation}
	In part (b), $u\in\Sch_x$, so $U_2$ is radial by Lemma~\ref{s3lem:diam}\,(b), and we apply Corollary~\ref{s3cor:HSrad}, whose hypotheses ($n\ge2$, $0<s<\frac n2$ and \eqref{s3eq:arange}) hold; this again gives \eqref{s3eq:step3}. Note that it is exactly here that the radial symmetry is used.

	Inserting \eqref{s3eq:step3} into \eqref{s3eq:step2} yields $\||x|^\beta u\|_B\le C\,\cA^{\tau}\,\cC^{1-\tau}$, which is \eqref{s3eq:main}. The constants depend only on $m,n,s,r,b,\beta$. Since $\Sch_{xy}\subset\Sch_x$, part (b) also applies to functions radial in both variables.

\end{proof}
	We now present the proofs of the optimality statements and their consequences, constructing explicit test functions designed to isolate the fundamental scaling and geometric obstructions.
	%==========================================================================
	\section*{Proof of Proposition~\ref{s3prop:opt} and Corollary~\ref{s3cor:opt}}\label{s3sec:nec}
	%==========================================================================
	
	\paragraph{Test functions.}
	All the test functions below are products $u(x,y)=\phi(x)\psi(y)$ with $\phi\in C_c^\infty(\R^n\setminus\{0\})$ and $\psi\in C_c^\infty(\R^m)$, $\phi,\psi\ne0$. The function $u$ belongs to the class $\mathcal Y$ provided $\phi$ is radial when $\mathcal Y\in\{\Sch_x,\Sch_{xy}\}$ and $\psi$ is radial when $\mathcal Y\in\{\Sch_y,\Sch_{xy}\}$. For such $u$ we have
	\begin{equation}\label{s3eq:prod}
		\||x|^\beta u\|_B=\||x|^\beta\phi\|_B\|\psi\|_B,\quad\cA(u)=\||x|^b\phi\|_2\|D^r\psi\|_2,\quad\cC(u)=\|D^s\phi\|_2\|\psi\|_2,
	\end{equation}
	and all these quantities are finite and positive, for all real $b,\beta$, because $\operatorname{supp}\phi\subset\{x\ne0\}$ is compact.
	
	\paragraph{Proof of (a): scaling.}
	For $\lambda,\mu>0$ let $u_{\lambda,\mu}(x,y)=u(\lambda x,\mu y)$; it belongs to the same class as $u$, and has compact support in $\{x\neq0\}$ if $u$ has. A change of variables gives
	\begin{gather*}
		\||x|^\beta u_{\lambda,\mu}\|_B=\lambda^{-\beta-\frac nB}\mu^{-\frac mB}\||x|^\beta u\|_B,\qquad
		\cA(u_{\lambda,\mu})=\lambda^{-b-\frac n2}\mu^{r-\frac m2}\cA(u),\\
		\cC(u_{\lambda,\mu})=\lambda^{s-\frac n2}\mu^{-\frac m2}\cC(u).
	\end{gather*}
	Inserting this in \eqref{s3eq:ineqK} we get, for all $\lambda,\mu>0$,
	\[
	\lambda^{-\beta-\frac nB}\mu^{-\frac mB}\le K'\,\lambda^{e_1}\mu^{e_2},\qquad e_1=-\tau\big(b+\tfrac n2\big)+(1-\tau)\big(s-\tfrac n2\big),\quad e_2=\tau r-\tfrac m2,
	\]
	with $K'=K\cA(u)^\tau\cC(u)^{1-\tau}/\||x|^\beta u\|_B$ independent of $\lambda,\mu$. Letting $\lambda\to0$ and $\lambda\to\infty$ (with $\mu=1$) forces $-\beta-\frac nB=e_1$, and letting $\mu\to0$ and $\mu\to\infty$ forces $-\frac mB=e_2$. With $\delta':=\frac12-\frac1B$ these two relations read
	\begin{equation}\label{s3eq:scal}
		\delta'=\frac{\tau r}m,\qquad n\delta'=\varrho-\tau(b+s),
	\end{equation}
	because $e_1=-\tau b+(1-\tau)s-\frac n2$ and $-\beta-\frac nB=-\beta-\frac n2+n\delta'$. Substituting the first relation in the second gives $\tau\big(\frac{rn}m+b+s\big)=\varrho$, that is $\tau=m\varrho/\Gamma$ (recall $\Gamma>0$), and then $\delta'=r\varrho/\Gamma=\delta$ and $B=2/(1-2\delta)$. This is \eqref{s3eq:defparam}.
	
	\paragraph{Proof of (b): concentration at a point $x_0\ne0$.}
	Let $\mathcal Y=\Sch$ or $\Sch_y$. Fix $x_0\ne0$, $\phi\in C_c^\infty(B(0,1))$, $\psi\in C_c^\infty(\R^m)$ (radial if $\mathcal Y=\Sch_y$), and for $0<\varepsilon<|x_0|/2$ let $\phi_\varepsilon(x)=\phi((x-x_0)/\varepsilon)$ and $u_\varepsilon=\phi_\varepsilon\otimes\psi\in\mathcal Y$. On $\operatorname{supp}\phi_\varepsilon$ we have $\frac12|x_0|\le|x|\le\frac32|x_0|$, so by \eqref{s3eq:prod}
	\[
	\||x|^\beta u_\varepsilon\|_B\ge c\,\varepsilon^{n/B},\qquad\cA(u_\varepsilon)\le C\,\varepsilon^{n/2},\qquad\cC(u_\varepsilon)=\varepsilon^{\frac n2-s}\|D^s\phi\|_2\|\psi\|_2,
	\]
	the last identity being exact by scaling. Since $\tau\in(0,1]$, inequality \eqref{s3eq:ineqK} gives $\varepsilon^{n/B}\lesssim\varepsilon^{\tau\frac n2+(1-\tau)(\frac n2-s)}=\varepsilon^{\frac n2-(1-\tau)s}$ for all small $\varepsilon$, hence $\frac nB\ge\frac n2-(1-\tau)s$, that is
	\begin{equation}\label{s3eq:pointcond}
		n\delta'\le(1-\tau)s .
	\end{equation}
	Since $\delta'=\tau r/m>0$, this forces $\tau<1$. By the second relation in \eqref{s3eq:scal}, $n\delta'=\beta+s-\tau b-\tau s$, so \eqref{s3eq:pointcond} reads $\beta+s-\tau b-\tau s\le s-\tau s$, that is $\beta\le\tau b$. Multiplying by $\Gamma>0$ and using $\tau\Gamma=m\varrho$, this is $\beta\Gamma\le m(s+\beta)b$, that is $\beta(rn+mb+ms)\le msb+m\beta b$, which is \eqref{s3H4}.
	
	\paragraph{Proof of (c): the thin shell.}
	Let $\mathcal Y=\Sch_x$ or $\Sch_{xy}$. Fix $\phi_0\in C_c^\infty((-1,1))$, $\phi_0\ne0$, and $\psi\in C_c^\infty(\R^m)$ (radial if $\mathcal Y=\Sch_{xy}$), and for $0<\varepsilon<\frac12$ let
	\[
	\Phi_\varepsilon(x)=\phi_0\Big(\frac{|x|-1}{\varepsilon}\Big),\qquad u_\varepsilon=\Phi_\varepsilon\otimes\psi .
	\]
	Then $\Phi_\varepsilon$ is smooth, radial, and supported in the annulus $S_\varepsilon=\{1-\varepsilon\le|x|\le1+\varepsilon\}$, whose volume is $\le C\varepsilon$; thus $u_\varepsilon\in\mathcal Y$. We estimate the three factors in \eqref{s3eq:prod}. First, since $|x|\in[\frac12,\frac32]$ on $S_\varepsilon$ and $\int|\Phi_\varepsilon|^B\,dx=|S^{n-1}|\int_0^\infty t^{\,n-1}|\phi_0(\tfrac{t-1}\varepsilon)|^B\,dt\ge c\,\varepsilon$,
	\[
	\||x|^\beta\Phi_\varepsilon\|_B\ge c\,\varepsilon^{1/B},\qquad\||x|^b\Phi_\varepsilon\|_2\le C\,\varepsilon^{1/2}.
	\]
	Second, we claim $\|D^s\Phi_\varepsilon\|_2\le C\varepsilon^{\frac12-s}$. For $s=1$ this is clear, since $|\nabla\Phi_\varepsilon|\le\|\phi_0'\|_\infty\varepsilon^{-1}$ on a set of volume $\le C\varepsilon$. For $0<s<1$ we use Lemma~\ref{s3lem:diam}\,(a). Since $||x|-|z||\le|x-z|$, the function $\Phi_\varepsilon$ is Lipschitz with constant $\|\phi_0'\|_\infty/\varepsilon$. In the region $|h|<\varepsilon$ of the double integral, the integrand vanishes unless $x$ lies in the annulus $S_{2\varepsilon}$, whose volume is $\le C\varepsilon$, so
	\[
	\iint_{|h|<\varepsilon}\frac{|\Phi_\varepsilon(x+h)-\Phi_\varepsilon(x)|^2}{|h|^{n+2s}}\,dx\,dh\le C\varepsilon\cdot\varepsilon^{-2}\int_{|h|<\varepsilon}|h|^{2-n-2s}\,dh=C\varepsilon^{1-2s},
	\]
	because $s<1$. In the region $|h|\ge\varepsilon$ we use $|\Phi_\varepsilon(x+h)-\Phi_\varepsilon(x)|^2\le2|\Phi_\varepsilon(x+h)|^2+2|\Phi_\varepsilon(x)|^2$, and obtain a bound $4\|\Phi_\varepsilon\|_2^2\int_{|h|\ge\varepsilon}|h|^{-n-2s}dh\le C\varepsilon\cdot\varepsilon^{-2s}$. Adding the two regions gives the claim.
	
	Hence $\cA(u_\varepsilon)\le C\varepsilon^{1/2}$ and $\cC(u_\varepsilon)\le C\varepsilon^{\frac12-s}$, and \eqref{s3eq:ineqK} gives, for small $\varepsilon$ and $\tau\in(0,1]$,
	\[
	\varepsilon^{1/B}\lesssim\varepsilon^{\frac\tau2+(1-\tau)(\frac12-s)}=\varepsilon^{\frac12-(1-\tau)s}.
	\]
	Therefore $\frac1B\ge\frac12-(1-\tau)s$, that is
	\begin{equation}\label{s3eq:shellcond}
		\delta'\le(1-\tau)s .
	\end{equation}
	Again $\delta'=\tau r/m>0$ forces $\tau<1$. Using $\delta'=r\varrho/\Gamma$ and $1-\tau=(\Gamma-m\varrho)/\Gamma$ (recall $\Gamma>0$), \eqref{s3eq:shellcond} becomes $r\varrho\le s(\Gamma-m\varrho)=s(rn+mb-m\beta)$, that is $rs+r\beta\le srn+smb-sm\beta$, that is $\beta(r+ms)\le s(mb+(n-1)r)$. This is \eqref{s3R1}.
	
	\paragraph{Proof of (d).}
	By (a), $\tau=m\varrho/\Gamma$ and $B=2/(1-2\delta)$ with $\delta=r\varrho/\Gamma$, and by (b) or (c) we have $\tau<1$. Since $\tau>0$ and $\Gamma>0$ we get $\varrho>0$, which is \eqref{s3H1}; $\tau<1$ means $\Gamma>m\varrho$, which is \eqref{s3H2}; and $B<\infty$ means $\frac1B=\frac12-\delta>0$, that is $\Gamma>2r\varrho$, which is \eqref{s3H3}.
	
	\paragraph{Proof of Corollary~\ref{s3cor:opt}.}
	Assume \eqref{s3H0}--\eqref{s3H3}; then $\Gamma>0$, $2<B<\infty$ and $0<\tau<1$ by Lemma~\ref{s3lem:param}. (1) In the class $\Sch_y$, if \eqref{s3H4} holds then \eqref{s3eq:main} holds for all $u\in\Sch$ by Theorem~\ref{s3thm:main}\,(a), in particular for $u\in\Sch_y$. Conversely, if \eqref{s3eq:main} holds in $\Sch_y$, then \eqref{s3eq:ineqK} holds with $\mathcal Y=\Sch_y$, and Proposition~\ref{s3prop:opt}\,(b) gives \eqref{s3H4}. (2) Let $n\ge2$ and $s<\frac12$. If \eqref{s3R1} holds, then \eqref{s3R2} holds by Lemma~\ref{s3lem:param}\,(vi), and Theorem~\ref{s3thm:main}\,(b) gives \eqref{s3eq:main} on $\Sch_x$ and $\Sch_{xy}$. Conversely, \eqref{s3eq:main} on $\Sch_x$ or $\Sch_{xy}$ implies \eqref{s3eq:ineqK}, hence \eqref{s3R1} by Proposition~\ref{s3prop:opt}\,(c). For $s\ge\frac12$, necessity of \eqref{s3R1} is Proposition~\ref{s3prop:opt}\,(c), and sufficiency of \eqref{s3R1} and \eqref{s3R2} is Theorem~\ref{s3thm:main}\,(b). \qed
	
	%==========================================================================
 \bigskip
	%==========================================================================
	
	The effect of radial symmetry can be understood from the proof of Theorem~\ref{s3thm:main} and from the optimality statements. Radial symmetry in $x$ acts on the $x$-step, where the quantity $U_2(x)=\|u(x,\cdot)\|_{L^2_y}$ must be controlled in a weighted Lebesgue space by $\|D^sU_2\|_2$ alone. Without symmetry the weight $|x|^{-a}$ must satisfy $a\ge0$, and this is exactly the translation (more precisely, point-concentration) obstruction \eqref{s3eq:pointcond}. For radial functions translations are not available; the only remaining obstruction in the range of validity of the radial Hardy--Sobolev inequality is the thin shell of Proposition~\ref{s3prop:opt}\,(c), whose condition $a\ge-(n-1)s$ coincides with the condition $c(1-2s)\le2s(n-1)$ of Lemma~\ref{s3lem:radSob} when $s<\frac12$. This explains why, for $s<\frac12$, the radial result is sharp. The same thin-shell test applied directly to Lemma~\ref{s3lem:radSob} (with $\||x|^{c/q}\Phi_\varepsilon\|_q\gtrsim\varepsilon^{1/q}$ and $\|D^s\Phi_\varepsilon\|_2\lesssim\varepsilon^{\frac12-s}$) shows that its condition $c(1-2s)\le2s(n-1)$ cannot be relaxed.
	
	Radial symmetry in $y$ alone acts on the $y$-step, which is the Gagliardo--Nirenberg inequality in the unweighted space $L^B(\R^m)$ of Lemma~\ref{lem:GN}. Radial symmetry gives improvements for weighted norms (Strauss-type decay and weighted Sobolev inequalities), but for unweighted norms the range of Lemma~\ref{lem:GN} cannot be enlarged, and, more to the point, Proposition~\ref{s3prop:opt}\,(a),(b) show that all the necessary conditions are already forced by test functions that can be taken radial in $y$. Weights in $y$ would change this picture, but they require a different structure of the inequality, since the weight $|x|^b$ in $\cA(u)$ is a weight in $x$.
	
	When both variables are radial, the range of parameters is that of the $x$-radial case, by Corollary~\ref{s3cor:opt}\,(2). The additional symmetry removes the translations in $y$ as well; this is relevant for compactness questions, but it does not change the range of parameters for which \eqref{s3eq:main} holds.
	
	\begin{remark} 
		Hypothesis \eqref{s3H0}, $2s<n$, is used in Lemmas~\ref{lem:HS} and~\ref{s3lem:radSob} and cannot be dropped in the present argument, even for radial functions: for $n=2$, $s=1$ the homogeneous radial Hardy--Sobolev inequality fails. For $s\ge\frac12$ the gap between the necessary condition \eqref{s3R1} and the sufficient pair \eqref{s3R1}--\eqref{s3R2} corresponds to the endpoint $g=\infty$ of Corollary~\ref{s3cor:HSrad}, which is a radial Strauss-type pointwise estimate $|f(x)|\lesssim|x|^{-(n-2s)/2}\|D^sf\|_2$ ($\frac12<s<\frac n2$).  The endpoint $c(1-2s)=2s(n-1)$ of Lemma~\ref{s3lem:radSob} (equality in \eqref{s3R1} for $s<\frac12$) is used as stated in \cite{DrelichmanDuran11,DDD12}; if only the strict inequality is available, replace $\le$ by $<$ in \eqref{s3R1} in Theorem~\ref{s3thm:main}\,(b).
	\end{remark}

	\section{Minimizers}\label{sec-minimizers}

	Let $m,n\ge1$, $s,r\in(0,1]$, $b\ge0$, and write
	\begin{equation}\label{em-eq:D}
		D:=rn+m(b+s),\qquad B:=\frac{2D}{D-2sr}=\frac{2(rn+m(b+s))}{m(b+s)+r(n-2s)},\qquad t:=\frac{sm}{D}.
	\end{equation}
	Then $0<t<1$ and the following identities, which are used repeatedly, hold:
	\begin{equation}\label{em-eq:ids}
		1-t=\frac{rn+mb}{D},\qquad 1-\frac2B=\frac{2sr}{D},\qquad m\Big(1-\frac2B\Big)=2rt,\qquad \frac{D}{B}=\frac{D-2sr}{2}.
	\end{equation}
	(For the third identity: $m(1-2/B)=2msr/D=2rt$ because $t=sm/D$.)
	
  We assume $(m,n,s,r)\neq(1,1,1,1)$. Since $rn+ms\ge r+s\ge 2sr$ with equality throughout only if $m=n=1$ and $s=r=1$, this guarantees $D-2sr>0$, hence $2<B<\infty$. Moreover $B$ is decreasing in $b$, so $B\le B_0:=B|_{b=0}<\infty$. We put $B':=B/(B-2)$, the conjugate exponent of $B/2$, and $\varrho:=2/B\in(0,1)$.
	
  For $s\in(0,1)$ and a measurable $u$ on $\R^{n+m}$ set
	\[
	E(u):=c_{n,s}\int_{\R^m}\!\iint_{\R^n\times\R^n}\frac{|u(x,y)-u(x',y)|^2}{|x-x'|^{n+2s}}\dd x\dd x'\dd y ,
	\]
	and for $s=1$ let $E(u):=\int|\nabla_xu|^2$. For $r\in(0,1)$ put
	\[
	F(u):=c_{m,r}\int_{\R^n}|x|^{2b}\iint_{\R^m\times\R^m}\frac{|u(x,y)-u(x,y')|^2}{|y-y'|^{m+2r}}\dd y\dd y'\dd x ,
	\]
	and for $r=1$ let $F(u):=\int|x|^{2b}|\nabla_yu|^2$ (the gradient being the distributional gradient on $\{x\neq0\}$). We write $G(u)$ for $F(u)$ with $b=0$. The constants $c_{n,s}$, $c_{m,r}$ are normalised so that, for $w\in L^2(\R^{n+m})$ and the unitary Fourier transform $\hat w(\xi,\eta)$,
	\begin{equation}\label{em-eq:fourier}
		E(w)=\int|\xi|^{2s}|\hat w|^2\dd\xi\dd\eta,\qquad G(w)=\int|\eta|^{2r}|\hat w|^2\dd\xi\dd\eta .
	\end{equation}
	This is possible because, by the Plancherel identity, $\int|w(x+z,y)-w(x,y)|^2\dd x\dd y=\int|e^{iz\cdot\xi}-1|^2|\hat w|^2\dd\xi\dd\eta$, and $\int|e^{iz\cdot\xi}-1|^2|z|^{-n-2s}\dd z=\kappa_{n,s}|\xi|^{2s}$ with $\kappa_{n,s}\in(0,\infty)$ by rotation invariance and homogeneity; take $c_{n,s}=\kappa_{n,s}^{-1}$, and likewise in the $y$-variable. Finally put
	\[
	\Nn(u):=\|u\|_B^B,\qquad J(u):=\frac{E(u)^{1-t}F(u)^{t}}{\|u\|_B^{2}},\qquad J_*:=\inf\{J(u):u\in V\setminus\{0\}\},
	\]
	where $V$ is the space defined next.
	
	\paragraph{The space $V$.} Let $V:=\{u\in L^B(\R^{n+m}):E(u)+F(u)<\infty\}$ with $\|u\|_V^2:=E(u)+F(u)$, and let $V_0$ be the analogous space with $b=0$.
	
	\paragraph{Hypotheses.} What we prove is the existence of optimizers and the compactness of minimizing sequences; the inequality itself is an input, exactly as the Sobolev inequality is an input in Theorems~1 and~3 of the paper. We assume:
	
	\medskip
	\noindent\textbf{(H$_b$)} there is $C_b<\infty$ with $\|u\|_B\le C_b\,E(u)^{\frac{1-t}2}F(u)^{\frac t2}$ for all $u\in V$;
	
	\smallskip
	\noindent\textbf{(H$_0$)} (only needed if $b>0$) there is $C_0<\infty$ with $\|v\|_{B_0}\le C_0\,E(v)^{\frac{1-t_0}2}G(v)^{\frac{t_0}2}$ for all compactly supported $v\in V_0$, where $t_0=\frac{sm}{rn+ms}$.
	
	\medskip
	Under (H$_b$) we have $J_*\ge C_b^{-2}>0$.
	
	\paragraph{Symmetries.} For $c,\lambda,\mu>0$ and $y_0\in\R^m$ let $(T_{c,\lambda,\mu,y_0}u)(x,y):=c\,u(\lambda x,\mu(y-y_0))$. Let $\G$ be the group generated by these maps if $b>0$, and by these maps together with translations in $x$ if $b=0$. By Lemma~\ref{em-lem:scaling} below, $J(gu)=J(u)$ for all $g\in\G$.

 With the functional framework, hypotheses, and symmetries firmly established, we can now state our main result regarding the strict realization of the infimum $J_*$. 
	\begin{theorem}\label{em-thm:main}
		Assume (H$_b$) and, if $b>0$, also (H$_0$). Then $J_*$ is attained. More precisely, if $(\tilde u_k)\subset V\setminus\{0\}$ satisfies $J(\tilde u_k)\to J_*$, then there are $g_k\in\G$ and a subsequence such that $g_k\tilde u_k\to u_0$ strongly in $V$, where $u_0\in V$ satisfies $J(u_0)=J_*$.
	\end{theorem}

Before analyzing minimizing sequences, we must first establish the fundamental topological structure and weak lower semicontinuity properties of our working space $V$.
	
	\begin{lemma}[The space $V$]\label{em-lem:V}
		Assume (H$_b$). Then $\|\cdot\|_V$ is a norm and $(V,\|\cdot\|_V)$ is a Hilbert space continuously embedded in $L^B$. The forms $E,F$ are induced by bounded symmetric bilinear forms $\langle\cdot,\cdot\rangle_E$, $\langle\cdot,\cdot\rangle_F$ on $V$ with $|\langle u,v\rangle_E|\le\|u\|_V\|v\|_V$ and likewise for $F$. Consequently, if $u_k\rightharpoonup u_0$ weakly in $V$ and $v_k:=u_k-u_0$, then
		\begin{equation}\label{em-eq:hilbertsplit}
			E(u_k)=E(u_0)+E(v_k)+o(1),\qquad F(u_k)=F(u_0)+F(v_k)+o(1),
		\end{equation}
		and in particular $E(u_0)\le\liminf E(u_k)$, $F(u_0)\le\liminf F(u_k)$.
	\end{lemma}
	
	\begin{proof}
		By (H$_b$) and the weighted arithmetic--geometric mean inequality $E^{1-t}F^t\le(1-t)E+tF\le E+F$ we get $\|u\|_B\le C_b\|u\|_V$, so $\|u\|_V=0$ forces $u=0$. The forms $E$ and $F$ are quadratic forms, being (for fractional order) integrals of squares of linear expressions in $u$, or (for order one) squares of an $L^2$-norm of a derivative; hence they come from symmetric bilinear forms, and the bound $|\langle u,v\rangle_E|\le E(u)^{1/2}E(v)^{1/2}$ is the Cauchy--Schwarz inequality.
		
		For completeness, let $(u_j)$ be Cauchy in $V$. Then it is Cauchy in $L^B$, so $u_j\to u$ in $L^B$ and a subsequence converges almost everywhere. In the fractional case the integrands defining $E(u_i-u_j)$ and $F(u_i-u_j)$ converge almost everywhere on the relevant product spaces as $i\to\infty$ along this subsequence, so Fatou's lemma gives $E(u-u_j)\le\liminf_iE(u_i-u_j)$ and the same for $F$. Given $\varepsilon>0$ the right-hand sides are less than $\varepsilon$ for $j$ large, hence $u\in V$ and $u_j\to u$ in $V$. For $s=1$, $\nabla_xu_j$ is Cauchy in $L^2$, hence converges to some $g$, and since $u_j\to u$ in $L^1_{\rm loc}$ we get $\nabla_xu=g$ in the sense of distributions. For $r=1$ we argue in the same way on the open set $\{x\neq0\}$, where $|x|^b$ is locally bounded above and below.
		
		For (\ref{em-eq:hilbertsplit}) we write, for each of the two forms, $E(u_k)=E(u_0)+E(v_k)+2\langle u_0,v_k\rangle_E$. Since $v\mapsto\langle u_0,v\rangle_E$ is a bounded linear functional on $V$ and $v_k\rightharpoonup0$, the last term tends to $0$. The lower semicontinuity follows because $E(v_k)\ge0$.
	\end{proof}

To carefully handle the nonlinear structure of the functional and the fractional exponents that appear in our sub-additivity estimates, we will rely on the following sequence of elementary algebraic inequalities.
	\begin{lemma} \label{em-lem:elem}
		Let $t\in(0,1)$ and $\varrho\in(0,1)$.
		\begin{enumerate}
			\item[] \textup{(i)} If $A_1,A_2,B_1,B_2\ge0$ then $(A_1+A_2)^{1-t}(B_1+B_2)^t\ge A_1^{1-t}B_1^t+A_2^{1-t}B_2^t$.
			\item[] \textup{(ii)} If $\alpha\in(0,1)$ then $\alpha^\varrho+(1-\alpha)^\varrho>1$.
			\item[] \textup{(iii)} For every $\delta\in(0,1]$ there is $\gamma=\gamma(\delta,\varrho)>0$ such that $a^\varrho+b^\varrho\ge(1+\gamma)(a+b)^\varrho$ for all $a,b\in[\delta,1]$.
		\end{enumerate}
	\end{lemma}
	
%	\begin{proof}
%		For (i), apply H\"older's inequality for sums with exponents $1/(1-t)$ and $1/t$ to the sequences $(A_i^{1-t})_{i=1,2}$ and $(B_i^t)_{i=1,2}$: it gives $\sum_iA_i^{1-t}B_i^t\le(\sum_iA_i)^{1-t}(\sum_iB_i)^t$. For (ii), since $\varrho<1$ and $\alpha\in(0,1)$ we have $\alpha^\varrho>\alpha$ and $(1-\alpha)^\varrho>1-\alpha$; add. For (iii), if $a,b\in[\delta,1]$ then $\lambda:=a/(a+b)$ lies in the compact interval $I_\delta:=[\delta/(1+\delta),1/(1+\delta)]\subset(0,1)$ and $a^\varrho+b^\varrho=(a+b)^\varrho(\lambda^\varrho+(1-\lambda)^\varrho)$. By (ii) the continuous function $\lambda\mapsto\lambda^\varrho+(1-\lambda)^\varrho-1$ is positive on $I_\delta$, hence bounded below by some $\gamma>0$ there, and $\gamma$ depends only on $\delta$ and $\varrho$.
%	\end{proof}
	
	We shall also use the Brezis--Lieb lemma   in the following form: if $f_k\to f$ almost everywhere and $\sup_k\|f_k\|_{L^p}<\infty$ for some $0<p<\infty$, then $\|f_k\|_p^p-\|f_k-f\|_p^p\to\|f\|_p^p$.
	
	%====================================================================
 
	%====================================================================
	A crucial preliminary step in applying the concentration-compactness principle is identifying exactly how the norms and fractional energy forms behave under the symmetry group $\mathcal{G}$.
	\begin{lemma} \label{em-lem:scaling}
		For $v=T_{c,\lambda,\mu,y_0}u$ one has
		\begin{align}
			\Nn(v)&=c^B\lambda^{-n}\mu^{-m}\Nn(u),\label{em-eq:scN}\\
			E(v)&=c^2\lambda^{2s-n}\mu^{-m}E(u),\label{em-eq:scE}\\
			F(v)&=c^2\lambda^{-2b-n}\mu^{2r-m}F(u).\label{em-eq:scF}
		\end{align}
		Translations in $y$ (and in $x$ if $b=0$) leave $\Nn,E,F$ unchanged. In particular $J(v)=J(u)$.
	\end{lemma}
	
	\begin{proof}
		The formulas for $\Nn$, $E$, $F$ follow by the changes of variables $(x,y)\mapsto(\lambda x,\mu y)$, $(x',y)\mapsto(\lambda x',\mu y)$ and $(x,y')\mapsto(\lambda x,\mu y')$ respectively, keeping track of the Jacobians and of the homogeneity of the kernels $|x-x'|^{-n-2s}$, $|y-y'|^{-m-2r}$ and of the weight $|x|^{2b}$. For $E$, for example, the two integrations in $x,x'$ produce $\lambda^{-2n}$, the integration in $y$ produces $\mu^{-m}$, and the kernel produces $\lambda^{n+2s}$, so the total factor is $c^2\lambda^{2s-n}\mu^{-m}$. Translations in $y$ commute with all operations that appear in $E$ and $F$ because the weight depends on $x$ only; translations in $x$ do so if $b=0$.
		
		To check that $J$ is invariant, it suffices to see that $J(v)/J(u)=c^{a_1}\lambda^{a_2}\mu^{a_3}$ with $a_1=2(1-t)+2t-2=0$, $a_2=(2s-n)(1-t)-(2b+n)t+\frac{2n}{B}$, $a_3=-m(1-t)+(2r-m)t+\frac{2m}{B}$. Using (\ref{em-eq:ids}) and $t=sm/D$ one finds $a_3=2rt-m(1-2/B)=0$ and
		\[
		a_2=2s-2t(s+b)-n\Big(1-\frac2B\Big)=2s-\frac{2sm(s+b)+2srn}{D}=0 .\qedhere
		\]
	\end{proof}

	Exploiting the multi-parameter scaling invariances established above, we can uniquely normalize any non-zero configuration to simultaneously fix its $L^B$ mass and balance its respective gradient energies.
	\begin{lemma} \label{em-lem:norm}
		For every $u\in V\setminus\{0\}$ there are $c,\lambda>0$ such that $v:=T_{c,\lambda,1,0}u$ satisfies $\Nn(v)=1$ and $E(v)=F(v)=J(u)$.
	\end{lemma}
	
	\begin{proof}
		Put $j:=J(u)$ and write $a=\log c$, $l=\log\lambda$. By (\ref{em-eq:scN}) and (\ref{em-eq:scE}) with $\mu=1$, the two conditions $\Nn(v)=1$ and $E(v)=j$ form the linear system
		\[
		Ba-nl=-\log\Nn(u),\qquad 2a+(2s-n)l=\log j-\log E(u)
		\]
		in the unknowns $(a,l)$. Its determinant is $B(2s-n)+2n=2sB-n(B-2)=\frac{4s(D-rn)}{D-2sr}=\frac{4sm(b+s)}{D-2sr}\neq0$, so it has a unique solution. With this choice, $\Nn(v)=1$ and $E(v)=j$. Since $J(v)=J(u)=j$ by Lemma~\ref{em-lem:scaling}, we have $E(v)^{1-t}F(v)^t=j\,\|v\|_B^2=j$, that is $j^{1-t}F(v)^t=j$, and so $F(v)=j$ because $t>0$.
	\end{proof}

While the preceding normalizations fix the global energy levels, there remains an internal, hidden one-parameter scaling group that leaves all three governing functionals strictly invariant.
	\begin{lemma} \label{em-lem:hidden}
		For $\tau>0$ let $u_\tau(x,y):=\tau^{D/B}u(\tau^rx,\tau^{b+s}y)$, i.e.\ $u_\tau=T_{c,\lambda,\mu,0}u$ with $(c,\lambda,\mu)=(\tau^{D/B},\tau^r,\tau^{b+s})$. Then $\Nn(u_\tau)=\Nn(u)$, $E(u_\tau)=E(u)$ and $F(u_\tau)=F(u)$.
	\end{lemma}
	
	\begin{proof}
		By (\ref{em-eq:scN}) the exponent of $\tau$ in $\Nn(u_\tau)/\Nn(u)$ is $D-rn-m(b+s)=0$. By (\ref{em-eq:scE}) and $2D/B=D-2sr$, the exponent for $E$ is $(D-2sr)+r(2s-n)-m(b+s)=D-rn-m(b+s)=0$. By (\ref{em-eq:scF}) the exponent for $F$ is $(D-2sr)-r(2b+n)+(b+s)(2r-m)=D-rn-m(b+s)=0$.
	\end{proof}
	
	Thus the group has three independent scalings $(c,\lambda,\mu)$ but only two of them move the triple $(\Nn,E,F)$ (the invariance of $J$ is one linear relation among the three logarithms); the third, Lemma~\ref{em-lem:hidden}, fixes the triple and will be used to normalize the concentration scale, exactly as the single dilation is used in (8) of the paper.
	
	\paragraph{The concentration function.} For $x_0\in\R^n$, $y_0\in\R^m$ let $Q_1(x_0,y_0):=\bb_1(x_0)\times \bb_1(y_0)$, the unit box. For $u\in L^B$ and $\rho=(\rho_1,\rho_2)\in(0,\infty)^2$ define
	\[
	M_u(\rho):=\sup_{x_0,y_0}\int_{\bb_{\rho_1}(x_0)\times \bb_{\rho_2}(y_0)}|u|^B .
	\]
	
	By continuously acting on a function with this hidden symmetry, we can rigorously adjust and trap the exact spatial scale at which the sequence concentrates its mass.
	\begin{lemma}[Choice of the concentration scale]\label{em-lem:conc}
		Let $u\in L^B$ with $\|u\|_B=1$. The function $\Phi_u(\tau):=M_u(\tau^r,\tau^{b+s})$ is continuous and nondecreasing on $(0,\infty)$, tends to $0$ as $\tau\to0$ and to $1$ as $\tau\to\infty$, and $\Phi_u(\tau)=\sup_{x_0,y_0}\int_{Q_1(x_0,y_0)}|u_\tau|^B$. In particular, for every $\delta\in(0,1)$ there is $\tau>0$ with $\Phi_u(\tau)=\delta$.
	\end{lemma}
	
	\begin{proof}
		Changing variables $(x,y)\mapsto(\tau^rx,\tau^{b+s}y)$ we get $\int_{Q_1(x_0,y_0)}|u_\tau|^B=\tau^{D-rn-m(b+s)}\int|u|^B$ over the box $\bb_{\tau^r}(\tau^rx_0)\times \bb_{\tau^{b+s}}(\tau^{b+s}y_0)$, and the prefactor equals $1$. Taking the supremum over centers gives the formula for $\Phi_u$. Monotonicity is clear. Since $|u|^B\in L^1$, its integral is uniformly absolutely continuous: for each $\varepsilon>0$ there is $\eta>0$ with $\int_S|u|^B<\varepsilon$ whenever $|S|<\eta$. If $\rho'\to\rho$, the Lebesgue measure of the symmetric difference of two boxes with the same center and radii $\rho,\rho'$ tends to $0$ uniformly in the center, so $|M_u(\rho')-M_u(\rho)|\le\sup_{\rm centers}\int_{\rm sym.\,diff.}|u|^B\to0$; this is the continuity. As $\tau\to0$ the boxes have measure tending to $0$ uniformly, so $\Phi_u(\tau)\to0$ by the same absolute continuity. As $\tau\to\infty$ the boxes centered at the origin exhaust $\R^{n+m}$, so $\Phi_u(\tau)\to\int|u|^B=1$ by monotone convergence. The last assertion is the intermediate value theorem.
	\end{proof}
	
	%====================================================================
	\section*{Localization estimates}
	%====================================================================
	
	For a Lipschitz function $\theta:\R^n\to[0,1]$ with compact support put
	\[
	W_\theta(x):=c_{n,s}\int_{\R^n}\frac{|\theta(x)-\theta(x')|^2}{|x-x'|^{n+2s}}\dd x'\quad(s<1),\qquad W_\theta:=|\nabla\theta|^2\quad(s=1),
	\]
	and for a Lipschitz $\psi:\R^m\to[0,1]$ with compact support define $W_\psi$ in the same way with $(m,r)$ in place of $(n,s)$. We also write $\gamma_u(x):=c_{m,r}\iint|u(x,y)-u(x,y')|^2|y-y'|^{-m-2r}\dd y\dd y'$ (or $\int|\nabla_yu(x,y)|^2\dd y$ if $r=1$), so that $F(u)=\int|x|^{2b}\gamma_u(x)\dd x$ and $G(u)=\int\gamma_u(x)\dd x$.

	In order to rigorously perform localized energy estimates on minimizing sequences, we first need precise bounds on the errors introduced by applying Lipschitz cutoff functions.
	\begin{lemma} \label{em-lem:cutoff}
		Let $\theta,\psi$ be as above and $u\in V$, $w:=\theta(x)\psi(y)u$. Then $W_\theta\in L^1\cap L^\infty(\R^n)$, $W_\psi\in L^1\cap L^\infty(\R^m)$ and
		\begin{align}
			E(w)&\le2E(u)+2\int W_\theta(x)\psi(y)^2|u(x,y)|^2\dd x\dd y,\label{em-eq:cutE}\\
			G(w)&\le2\int\theta(x)^2\gamma_u(x)\dd x+2\int\theta(x)^2W_\psi(y)|u(x,y)|^2\dd x\dd y.\label{em-eq:cutG}
		\end{align}
		In particular, if $\theta=0$ on $\{|x|<\rho\}$ then $\int\theta^2\gamma_u\le\rho^{-2b}F(u)$. Finally, by H\"older's inequality with exponents $B/2$ and $B'$,
		\begin{equation}\label{em-eq:cutholder}
			\int W_\theta\psi^2|u|^2\le\|W_\theta\|_{L^{B'}}\|\psi^2\|_{L^{B'}}\|u\|_B^2,\qquad
			\int\theta^2W_\psi|u|^2\le\|\theta^2\|_{L^{B'}}\|W_\psi\|_{L^{B'}}\|u\|_B^2 .
		\end{equation}
	\end{lemma}
	
	\begin{proof}
		Assume $\operatorname{supp}\theta\subset \bb_\rho(x_0)$ and let $\ell$ be a Lipschitz constant of $\theta$. We have $|\theta(x)-\theta(x')|\le\min(1,\ell|x-x'|)$, so for $s<1$
		\[
		W_\theta(x)\le c_{n,s}\int\frac{\min(1,\ell^2|z|^2)}{|z|^{n+2s}}\dd z=:C_{\ell}<\infty,
		\]
		the integral being finite near $0$ because $2-2s>-0$ and at infinity because $2s>0$. If $|x-x_0|\ge2\rho$ then $\theta(x)=0$ and $W_\theta(x)=c_{n,s}\int\theta(x')^2|x-x'|^{-n-2s}\dd x'\le c_{n,s}|\bb_\rho|(|x-x_0|/2)^{-n-2s}$, which is integrable at infinity. Hence $W_\theta\in L^1\cap L^\infty$. For $s=1$ the claim is obvious. The same argument applies to $W_\psi$.
		
		For (\ref{em-eq:cutE}) fix $y$ and use $|\theta(x)u(x,y)-\theta(x')u(x',y)|^2\le2\theta(x')^2|u(x,y)-u(x',y)|^2+2|\theta(x)-\theta(x')|^2|u(x,y)|^2$. Multiply by $\psi(y)^2\le1$ and the kernel, and integrate. The first term is bounded by $2E(u)$ because $\theta\le1$. In the second term, integrating first in $x'$ produces $W_\theta(x)$. For $s=1$ one uses instead $|\nabla_x(\theta\psi u)|^2\le2\theta^2\psi^2|\nabla_xu|^2+2|\nabla\theta|^2\psi^2u^2$. The proof of (\ref{em-eq:cutG}) is identical with the roles of $x$ and $y$ exchanged: for fixed $x$ one uses $|\psi(y)u(x,y)-\psi(y')u(x,y')|^2\le2\psi(y')^2|u(x,y)-u(x,y')|^2+2|\psi(y)-\psi(y')|^2|u(x,y)|^2$, the factor $\theta(x)^2$ being outside. If $\theta$ vanishes for $|x|<\rho$ then $1\le\rho^{-2b}|x|^{2b}$ on its support, which gives the stated bound. The inequalities (\ref{em-eq:cutholder}) are H\"older's inequality together with Fubini's theorem.
	\end{proof}
	
	The scaling behaviour of the norms entering (\ref{em-eq:cutholder}) is what makes the estimates in Lemma~\ref{em-lem:center} work. If $\psi(y)=\psi_0((y-y_0)/\Lambda)$, then
	\[
	\|\psi^2\|_{L^{B'}}=\Lambda^{m/B'}\|\psi_0^2\|_{L^{B'}},\qquad W_\psi(y)=\Lambda^{-2r}W_{\psi_0}\Big(\frac{y-y_0}\Lambda\Big),\qquad \|W_\psi\|_{L^{B'}}=\Lambda^{-2r+m/B'}\|W_{\psi_0}\|_{L^{B'}} ,
	\]
	and by (\ref{em-eq:ids}), $m/B'=m(1-2/B)=2rt$ and $-2r+m/B'=-2r(1-t)$. Similarly $W_\theta$ and $\|\theta^2\|_{L^{B'}}$ do not change if $\theta_0$ is merely translated.

	Because the fractional energy forms are non-local, standard spatial truncation introduces long-range interaction errors; the following lemma demonstrates how to control these cross-terms to obtain an almost-orthogonal splitting of the energy.
	\begin{lemma} \label{em-lem:split}
		Let $x_c\in\R^n$, let $\phi:[0,\infty)\to[0,\pi/2]$ be Lipschitz with constant $\ell$, $\phi=0$ on $[0,2]$ and $\phi=\pi/2$ on $[4,\infty)$, and set
		\[
		\Theta(x,y):=\phi\big(\max\{|x-x_c|,|y|\}\big),\qquad\chi:=\cos\Theta,\quad\eta:=\sin\Theta,\qquad Q:=\bb_4(x_c)\times \bb_4(0).
		\]
		Then $\chi^2+\eta^2=1$, $\chi u,\eta u\in V$ for every $u\in V$, and
		\begin{align}
			\big|E(u)-E(\chi u)-E(\eta u)\big|&\le C\big(E(u)^{1/2}\|u\|_{L^2(Q)}+\|u\|_{L^2(Q)}^2\big),\label{em-eq:splitE}\\
			\big|F(u)-F(\chi u)-F(\eta u)\big|&\le C\big(F(u)^{1/2}\|u\|_{L^2(Q)}+\|u\|_{L^2(Q)}^2\big),\label{em-eq:splitF}
		\end{align}
		with $C$ depending only on $m,n,s,r,b,\ell$ and $|x_c|$.
	\end{lemma}
	
	\begin{proof}
		Observe first that $\Theta$ is Lipschitz in $x$ and in $y$ separately with constant $\ell$, because $\max\{|x-x_c|,|y|\}$ is $1$-Lipschitz in each variable; consequently $f\in\{\chi,\eta\}$ satisfies $|f(x,y)-f(x',y)|\le\min(1,\ell|x-x'|)$ and $|f(x,y)-f(x,y')|\le\min(1,\ell|y-y'|)$. Moreover, $\Theta=\pi/2$, hence $\chi=0$ and $\eta=1$, whenever $\max\{|x-x_c|,|y|\}\ge4$, that is outside $Q$.
		
		We treat $E$ with $s<1$. Write $\mathcal K(z)=c_{n,s}|z|^{-n-2s}$ and let $\mathcal R:=\{(x,x')\in\R^n\times\R^n:|x-x_c|<|x'-x_c|\}$. For any measurable $v$ the integrand $|v(x,y)-v(x',y)|^2\mathcal K(x-x')$ is symmetric under $x\leftrightarrow x'$ and the set $\{|x-x_c|=|x'-x_c|\}$ is a null set, so
		\[
		E(v)=2\int_{\R^m}\!\iint_{\mathcal R}|v(x,y)-v(x',y)|^2\mathcal K(x-x')\dd x\dd x'\dd y=:2\,\|\delta v\|^2,
		\]
		where $\delta v(x,x',y):=v(x,y)-v(x',y)$ and $\|\cdot\|$ denotes the norm of $L^2(\mathcal R\times\R^m;\mathcal K\dd x\dd x'\dd y)$. For $f\in\{\chi,\eta\}$ and $(x,x')\in\mathcal R$ we split
		\[
		\delta(fu)=f(x',y)\big(u(x,y)-u(x',y)\big)+\big(f(x,y)-f(x',y)\big)u(x,y)=:A_f+\Delta_f .
		\]
		The reason for ordering the pair so that $x$ is the point closer to $x_c$ is the following. Since $\chi^2+\eta^2=1$ at the point $(x',y)$, we have the \emph{exact} identity $\sum_f|A_f|^2=|u(x,y)-u(x',y)|^2$, whence
		\begin{equation}\label{em-eq:Aid}
			\sum_f\|A_f\|^2=\|\delta u\|^2=\tfrac12E(u),\qquad\text{and in particular }\|A_f\|\le\big(\tfrac12E(u)\big)^{1/2}.
		\end{equation}
		On the other hand, if $\Delta_f(x,x',y)\neq0$ then $f(x,y)\neq f(x',y)$, so $\Theta(x,y)\neq\Theta(x',y)$ and at least one of the numbers $\max\{|x-x_c|,|y|\}$, $\max\{|x'-x_c|,|y|\}$ is smaller than $4$. If the second one is, then $|y|<4$ and $|x-x_c|<|x'-x_c|<4$, because $(x,x')\in\mathcal R$; if the first is, this is immediate. In either case $(x,y)\in Q$. Thus $\Delta_f$ is supported where $(x,y)\in Q$, and
		\[
		\|\Delta_f\|^2\le\int_Q|u(x,y)|^2\Big(\int_{\R^n}\min(1,\ell^2|x-x'|^2)\,\mathcal K(x-x')\dd x'\Big)\dd x\dd y= C_\ell\int_Q|u|^2 ,
		\]
		where $C_\ell=c_{n,s}\int\min(1,\ell^2|z|^2)|z|^{-n-2s}\dd z<\infty$. Now $E(fu)=2\|A_f+\Delta_f\|^2$ is finite (so $fu\in V$ as far as $E$ is concerned) and
		\[
		\Big|\sum_fE(fu)-E(u)\Big|=2\Big|\sum_f\big(2\operatorname{Re}\langle A_f,\Delta_f\rangle+\|\Delta_f\|^2\big)\Big|\le2\sum_f\big(2\|A_f\|\|\Delta_f\|+\|\Delta_f\|^2\big),
		\]
		where we used (\ref{em-eq:Aid}) in the form $\sum_f\|A_f\|^2=\frac12E(u)$ to cancel the main terms. Inserting the bounds for $\|A_f\|$ and $\|\Delta_f\|$ gives (\ref{em-eq:splitE}).
		
		For $s=1$ the product rule gives $\sum_f|\nabla_x(fu)|^2=\sum_f\big(f^2|\nabla_xu|^2+2fu\nabla_xf\cdot\nabla_xu+u^2|\nabla_xf|^2\big)=|\nabla_xu|^2+|\nabla_x\Theta|^2u^2$, because $\sum_f2f\nabla f=\nabla(\chi^2+\eta^2)=0$ and $|\nabla\chi|^2+|\nabla\eta|^2=|\nabla\Theta|^2$. Since $|\nabla_x\Theta|\le\ell$ and $\nabla_x\Theta=0$ outside $Q$, (\ref{em-eq:splitE}) follows.
		
		The estimate (\ref{em-eq:splitF}) is proved in the same way, now at fixed $x$ with the pair $(y,y')$ ordered by $|y|<|y'|$ and the kernel $c_{m,r}|y-y'|^{-m-2r}$, the whole expression being multiplied by $|x|^{2b}$. Indeed, if $f(x,y)\ne f(x,y')$ then $|x-x_c|<4$ (otherwise $f(x,\cdot)$ is constant), and at least one of $|y|,|y'|$ is smaller than $4$, hence the smaller one, $|y|$, is; so $\Delta_f$ is again supported where $(x,y)\in Q$. On $Q$ we have $|x|^{2b}\le(|x_c|+4)^{2b}$, which makes the constant depend on $|x_c|$. The case $r=1$ is the same product-rule computation, with the weight $|x|^{2b}$ in front.
	\end{proof}
	%====================================================================
	\subsection*{Compactness and the position of the concentration}
	%====================================================================
	
	Equipped with these localization tools, we now establish a fundamental local compactness property, ensuring that bounded sequences in $V$ are locally precompact in $L^2$ and do not suffer from wild oscillations.
	\begin{lemma}[Local compactness]\label{em-lem:loccomp}
		Let $(u_k)$ be bounded in $V$. Then there are a subsequence and $u_0\in V$ such that $u_k\rightharpoonup u_0$ weakly in $V$, $u_k\to u_0$ in $L^2(K)$ for every bounded set $K\subset\R^{n+m}$, and $u_k\to u_0$ almost everywhere.
	\end{lemma}
	
	\begin{proof}
		Let $\Lambda_0:=\sup_k\|u_k\|_V$. By Lemma~\ref{em-lem:V}, $(u_k)$ is bounded in $L^B$ and, being a bounded sequence in a Hilbert space, has a weakly convergent subsequence, $u_k\rightharpoonup u_0$. Since the embedding $V\hookrightarrow L^B$ is continuous, also $u_k\rightharpoonup u_0$ weakly in $L^B$.
		
		Fix $R\ge1$ and $K_R:=\bb_R\times \bb_R$. We claim that $\{u_k|_{K_R}\}$ is precompact in $L^2(K_R)$. Fix $\varepsilon\in(0,1)$. On $K_R\cap\{|x|<\varepsilon\}$, H\"older's inequality gives
		\begin{equation}\label{em-eq:smallx}
			\int_{K_R\cap\{|x|<\varepsilon\}}|u_k|^2\le|K_R\cap\{|x|<\varepsilon\}|^{1-2/B}\|u_k\|_B^2\le C(R)\,\varepsilon^{2nrs/D}\Lambda_0^2 ,
		\end{equation}
		which is small uniformly in $k$ (recall $1-2/B=2sr/D$). On the rest of $K_R$ we localize: choose Lipschitz $\theta:\R^n\to[0,1]$ with $\theta=1$ on $\{\varepsilon\le|x|\le R\}$ and $\operatorname{supp}\theta\subset\{\varepsilon/2\le|x|\le2R\}$, and Lipschitz $\psi:\R^m\to[0,1]$ with $\psi=1$ on $\bb_R$ and $\operatorname{supp}\psi\subset \bb_{2R}$, and put $w_k:=\theta\psi u_k$. By Lemma~\ref{em-lem:cutoff} and (\ref{em-eq:cutholder}),
		\[
		E(w_k)\le2\Lambda_0^2+2\|W_\theta\|_{B'}\|\psi^2\|_{B'}\|u_k\|_B^2\le C,\qquad
		G(w_k)\le2(\varepsilon/2)^{-2b}\Lambda_0^2+2\|\theta^2\|_{B'}\|W_\psi\|_{B'}\|u_k\|_B^2\le C,
		\]
		with $C$ independent of $k$. Moreover $w_k$ is supported in a fixed compact set and $\|w_k\|_2\le C$ by H\"older's inequality. Using (\ref{em-eq:fourier}) and $|e^{i\sigma}-1|^2\le\min(4,\sigma^2)\le4^{1-s}|\sigma|^{2s}$, we obtain for $h=(h_1,h_2)\in\R^n\times\R^m$
		\[
		\|w_k(\cdot+h)-w_k\|_2\le\|w_k(\cdot+h_1,\cdot)-w_k\|_2+\|w_k(\cdot,\cdot+h_2)-w_k\|_2\le2^{1-s}|h_1|^sE(w_k)^{1/2}+2^{1-r}|h_2|^rG(w_k)^{1/2}.
		\]
		The right-hand side tends to $0$ as $h\to0$ uniformly in $k$, so by the Kolmogorov--Riesz theorem $\{w_k\}$ is precompact in $L^2(\R^{n+m})$. Since $u_k=w_k$ on $K_R\cap\{|x|\ge\varepsilon\}$, (\ref{em-eq:smallx}) shows that $\{u_k|_{K_R}\}$ lies, for every $\varepsilon$, within $C(R)\varepsilon^{nrs/D}\Lambda_0$ of a precompact subset of $L^2(K_R)$. A set with this property for all $\varepsilon$ is totally bounded, which proves the claim.
		
		A diagonal argument over $R=1,2,3,\dots$ and over almost-everywhere convergence yields a subsequence and a function $\tilde u$ with $u_k\to\tilde u$ in $L^2(K_R)$ for all $R$ and almost everywhere. Testing against $\varphi\in C_c^\infty$ and using weak convergence in $L^B$ shows $\tilde u=u_0$.
	\end{proof}

	Finally, to prevent the minimizing sequence from escaping to infinity—which would result in a catastrophic loss of mass—we must guarantee that the spatial center of concentration remains uniformly bounded.
	\begin{lemma} \label{em-lem:center}
		Let $b>0$ and assume (H$_0$). For all $K,\delta>0$ there is $R_0=R_0(K,\delta)$ with the following property: if $u\in V$, $\Nn(u)=1$, $E(u)+F(u)\le K$ and $\int_{Q_1(x_0,y_0)}|u|^B\ge\delta$ for some unit box, then $|x_0|\le R_0$.
	\end{lemma}
	
	\begin{proof}
		Suppose $M:=|x_0|-2\ge1$; we shall bound $M$ from above. Fix Lipschitz functions $\theta_0:\R^n\to[0,1]$ and $\psi_0:\R^m\to[0,1]$ with $\theta_0=1$ on $\bb_1$, $\operatorname{supp}\theta_0\subset \bb_2$, $\psi_0=1$ on $\bb_1$, $\operatorname{supp}\psi_0\subset \bb_2$, independent of everything else. Put
		\[
		\Lambda:=M^\kappa,\quad\kappa:=\frac{b}{r(1-t)},\qquad\theta(x):=\theta_0(x-x_0),\quad\psi(y):=\psi_0\Big(\frac{y-y_0}\Lambda\Big),\qquad v:=\theta\psi u .
		\]
		Since $\Lambda\ge1$ we have $v=u$ on $Q_1(x_0,y_0)$, and $\theta$ vanishes on $\{|x|<M\}$.
		
		By (\ref{em-eq:cutE}), (\ref{em-eq:cutholder}) and the scaling identities after Lemma~\ref{em-lem:cutoff} (using $\|u\|_B=1$), $E(v)\le2K+C_1\Lambda^{2rt}$. By (\ref{em-eq:cutG}) and the second bound in Lemma~\ref{em-lem:cutoff}, $G(v)\le2M^{-2b}K+C_2\Lambda^{-2r(1-t)}$, with $C_1,C_2$ depending only on $\theta_0,\psi_0,m,n,s,r,B$. The choice of $\kappa$ gives $\Lambda^{-2r(1-t)}=M^{-2b}$, and $\Lambda^{2rt}=M^{\kappa\cdot2rt}=M^{2bt/(1-t)}$. Since $M\ge1$ we obtain
		\[
		E(v)\le C\,M^{2bt/(1-t)},\qquad G(v)\le C\,M^{-2b}.
		\]
		In particular $v\in V_0$ is compactly supported, so (H$_0$) applies:
		\[
		\|v\|_{B_0}^2\le C_0^2E(v)^{1-t_0}G(v)^{t_0}\le C\,M^{\frac{2bt(1-t_0)}{1-t}-2bt_0}=C\,M^{-\theta_0},
		\]
		where a short computation with $\frac t{1-t}=\frac{sm}{rn+mb}$ and $1-t_0=\frac{rn}{rn+ms}$ gives
		\[
		\theta_0=\frac{2b^2sm^2}{(rn+mb)(rn+ms)}>0 .
		\]
		On the other hand, since $B\le B_0$, H\"older's inequality on the unit box $Q:=Q_1(x_0,y_0)$ gives $\|v\|_{B_0}\ge\|u\|_{L^{B_0}(Q)}\ge|Q|^{\frac1{B_0}-\frac1B}\|u\|_{L^B(Q)}\ge c_0\,\delta^{1/B}$, with $c_0>0$ depending only on $m,n,B,B_0$. Combining the two bounds yields $c_0\delta^{1/B}\le C^{1/2}M^{-\theta_0/2}$, that is $M\le(C^{1/2}c_0^{-1}\delta^{-1/B})^{2/\theta_0}$. Hence $|x_0|\le R_0:=\max\{3,\,2+(C^{1/2}c_0^{-1}\delta^{-1/B})^{2/\theta_0}\}$.
	\end{proof}
	%====================================================================
 
	%====================================================================
	\begin{proof}[Proof of Theorem~\ref{em-thm:main}]

	Let $(\tilde u_k)\subset V\setminus\{0\}$ with $J(\tilde u_k)\to J_*$. Fix once and for all $N_1\in\N$ such that $\bb_4\times \bb_4\subset\R^n\times\R^m$ can be covered by $N_1$ unit boxes $Q_1(\cdot,\cdot)$, and put $\delta_*:=1/(2N_1)\le\frac12$.
	
	\medskip
	\noindent {Step 0: normalization.} By Lemma~\ref{em-lem:norm} we may replace $\tilde u_k$ by $T_{c_k,\lambda_k,1,0}\tilde u_k$ and assume $\Nn(u_k)=1$ and $E(u_k)=F(u_k)=j_k:=J(\tilde u_k)\to J_*$. Next, Lemma~\ref{em-lem:conc} provides $\tau_k>0$ with $\Phi_{u_k}(\tau_k)=\delta_*$; replacing $u_k$ by $(u_k)_{\tau_k}$ and using Lemmas~\ref{em-lem:hidden} and~\ref{em-lem:conc} we keep $\Nn=1$ and $E=F=j_k$, and we now have
	\begin{equation}\label{em-eq:normalized}
		\sup_{x_0,y_0}\int_{Q_1(x_0,y_0)}|u_k|^B=\delta_*\qquad\text{for all }k.
	\end{equation}
	By (\ref{em-eq:normalized}) there are $(x_k,y_k)$ with $\int_{Q_1(x_k,y_k)}|u_k|^B\ge\delta_*/2$. Translating in $y$ (and also in $x$ if $b=0$) we may assume $y_k=0$ (and $x_k=0$ if $b=0$); translations preserve $\Nn,E,F$ and (\ref{em-eq:normalized}). If $b>0$, Lemma~\ref{em-lem:center} with $K:=\sup_k2j_k$ and $\delta:=\delta_*/2$ shows that $(x_k)$ is bounded, so after passing to a subsequence $x_k\to x_\infty$; if $b=0$ put $x_\infty:=0$. In both cases, set $c:=(x_\infty,0)$. All operations so far are elements of $\G$ composed with the dilation of Lemma~\ref{em-lem:hidden}, which is itself of the form $T_{c,\lambda,\mu,0}$, so $u_k=g_k\tilde u_k$ for some $g_k\in\G$. Since $(u_k)$ is bounded in $V$, Lemma~\ref{em-lem:loccomp} gives, after passing to a further subsequence, $u_k\rightharpoonup u_0$ in $V$, $u_k\to u_0$ in $L^2$ of bounded sets, and almost everywhere.
	
	\medskip
	\noindent {Step 1: $u_0\neq0$.} Suppose, to the contrary, that $u_0=0$. Then $u_k\to0$ in $L^2(\bb_4(x_\infty)\times \bb_4(0))$. Let $\gamma:=\gamma(\delta_*/2,\varrho)>0$ be the constant of Lemma~\ref{em-lem:elem}\,(iii) and set $\varepsilon_0:=1-(1+\gamma)^{-1/\varrho}\in(0,1)$; this number depends on $\delta_*$ and $\varrho$ only.
	
	Fix $L\in\N$ and, for $i=1,\dots,L$, let $\phi_i:[0,\infty)\to[0,\pi/2]$ be the piecewise linear function that vanishes on $[0,2+2(i-1)/L]$, equals $\pi/2$ on $[2+2i/L,\infty)$ and is linear in between; it is Lipschitz with constant $\pi L/4$, equals $0$ on $[0,2]$ and equals $\pi/2$ on $[4,\infty)$. Let $\chi_i,\eta_i$ and $Q=\bb_4(x_\infty)\times \bb_4(0)$ be as in Lemma~\ref{em-lem:split} (with $x_c=x_\infty$ and $\phi=\phi_i$). Put
	\[
	a_k^{(i)}:=\int\chi_i^B|u_k|^B,\qquad b_k^{(i)}:=\int\eta_i^B|u_k|^B,\qquad g_i:=1-\chi_i^B-\eta_i^B .
	\]
	By Lemma~\ref{em-lem:split}, since $E(u_k),F(u_k)$ are bounded and $\|u_k\|_{L^2(Q)}\to0$ (with the constant $C$ depending on $L$ and $x_\infty$ but not on $k$),
	\[
	E(u_k)=E(\chi_iu_k)+E(\eta_iu_k)+o(1),\qquad F(u_k)=F(\chi_iu_k)+F(\eta_iu_k)+o(1)\qquad(k\to\infty).
	\]
	Because $(E,F)\mapsto E^{1-t}F^t$ is uniformly continuous on bounded subsets of $[0,\infty)^2$, Lemma~\ref{em-lem:elem}\,(i) yields
	\[
	j_k=E(u_k)^{1-t}F(u_k)^t\ge E(\chi_iu_k)^{1-t}F(\chi_iu_k)^t+E(\eta_iu_k)^{1-t}F(\eta_iu_k)^t-o(1).
	\]
	By the definition of $J_*$ applied to $\chi_iu_k$ and $\eta_iu_k$, namely $E(fu_k)^{1-t}F(fu_k)^t\ge J_*\|fu_k\|_B^2$, and since $\|\chi_iu_k\|_B^2=(a_k^{(i)})^\varrho$, $\|\eta_iu_k\|_B^2=(b_k^{(i)})^\varrho$, we conclude that
	\begin{equation}\label{em-eq:keyineq}
		\frac{j_k}{J_*}\ge\big(a_k^{(i)}\big)^\varrho+\big(b_k^{(i)}\big)^\varrho-o(1).
	\end{equation}
	We now bound $a_k^{(i)}$ and $b_k^{(i)}$ from both sides, uniformly in $i$ and $L$. Both are at most $\Nn(u_k)=1$ since $0\le\chi_i,\eta_i\le1$. For $k$ large $|x_k-x_\infty|<1$, so $Q_1(x_k,0)\subset \bb_2(x_\infty)\times \bb_2(0)$, where $\chi_i=1$; hence $a_k^{(i)}\ge\int_{Q_1(x_k,0)}|u_k|^B\ge\delta_*/2$. For the outer piece, $\bb_4\times \bb_4$ is covered by $N_1$ unit boxes, each of mass at most $\delta_*$ by (\ref{em-eq:normalized}), so $\int_{Q}|u_k|^B\le N_1\delta_*=\frac12$; since $\eta_i=1$ outside $Q$, $b_k^{(i)}\ge\int_{Q^c}|u_k|^B\ge\frac12\ge\delta_*/2$. Therefore $a_k^{(i)},b_k^{(i)}\in[\delta_*/2,1]$ and Lemma~\ref{em-lem:elem}\,(iii) applies to (\ref{em-eq:keyineq}): since $j_k/J_*\to1$,
	\[
	1+o(1)\ge(1+\gamma)\big(a_k^{(i)}+b_k^{(i)}\big)^\varrho,\qquad\text{i.e.}\qquad a_k^{(i)}+b_k^{(i)}\le1-\varepsilon_0+o(1).
	\]
	Equivalently, $\int g_i|u_k|^B=1-a_k^{(i)}-b_k^{(i)}\ge\varepsilon_0-o(1)$ as $k\to\infty$, for each fixed $i\le L$. Now $g_i\ge0$ because $\cos^B\Theta+\sin^B\Theta\le\cos^2\Theta+\sin^2\Theta=1$ for $B>2$, and $g_i$ is supported in the set where $0<\Theta_i<\pi/2$, which is contained in the annulus $\{2+2(i-1)/L<\max\{|x-x_\infty|,|y|\}<2+2i/L\}$. These annuli are pairwise disjoint for different $i$ and contained in $\{|x-x_\infty|<4,|y|<4\}$. Summing over $i=1,\dots,L$,
	\[
	1\ge\int\sum_{i=1}^Lg_i|u_k|^B\ge L\varepsilon_0-o(1)\qquad(k\to\infty).
	\]
	Letting $k\to\infty$ gives $L\varepsilon_0\le1$. As $L$ is arbitrary and $\varepsilon_0$ does not depend on $L$, this is absurd. Hence $u_0\neq0$.
	
	\medskip
	\noindent {Step 2: $u_0$ is an optimizer and $u_k\to u_0$ strongly in $V$.} Put $v_k:=u_k-u_0$ and $\alpha:=\|u_0\|_B^B$. By Brezis--Lieb (applicable since $u_k\to u_0$ a.e. and $\sup\|u_k\|_B<\infty$),
	\[
	1=\Nn(u_k)=\alpha+\Nn(v_k)+o(1),
	\]
	so $\Nn(v_k)\to1-\alpha$, and $\alpha\in(0,1]$ by Fatou's lemma and Step~1. By Lemma~\ref{em-lem:V}, $E(u_k)=E(u_0)+E(v_k)+o(1)$ and $F(u_k)=F(u_0)+F(v_k)+o(1)$. Applying Lemma~\ref{em-lem:elem}\,(i) and the definition of $J_*$ to $u_0$ and $v_k$ (when $v_k=0$ the corresponding term is absent) we get
	\[
	j_k=E(u_k)^{1-t}F(u_k)^t\ge J(u_0)\|u_0\|_B^2+J(v_k)\|v_k\|_B^2-o(1)\ge J_*\big(\alpha^\varrho+\Nn(v_k)^\varrho\big)-o(1).
	\]
	Letting $k\to\infty$ and using $j_k\to J_*>0$ we obtain $\alpha^\varrho+(1-\alpha)^\varrho\le1$. By Lemma~\ref{em-lem:elem}\,(ii) this is impossible if $\alpha\in(0,1)$, so $\alpha=1$. Hence $\|u_0\|_B=1$ and $\|v_k\|_B\to0$.
	
	By the lower semicontinuity in Lemma~\ref{em-lem:V}, $E(u_0)\le\liminf E(u_k)=J_*$ and $F(u_0)\le J_*$. On the other hand, $J_*\le J(u_0)=E(u_0)^{1-t}F(u_0)^t$ because $\|u_0\|_B=1$. If one of the inequalities $E(u_0)\le J_*$, $F(u_0)\le J_*$ were strict, then, since $J_*>0$ and $0<t<1$, we would have $E(u_0)^{1-t}F(u_0)^t<J_*$, a contradiction. Hence $E(u_0)=F(u_0)=J_*=J(u_0)$, and $u_0$ is an optimizer. Finally, again by Lemma~\ref{em-lem:V},
	\[
	E(v_k)=E(u_k)-E(u_0)+o(1)=j_k-J_*+o(1)\to0,\qquad F(v_k)=F(u_k)-F(u_0)+o(1)\to0,
	\]
	that is $\|u_k-u_0\|_V\to0$.
 
	 We showed that, along a subsequence, $g_k\tilde u_k=u_k\to u_0$ in $V$ with $g_k\in\G$ and $J(u_0)=J_*$. This proves Theorem~\ref{em-thm:main}. 
	 	\end{proof}
	
 \begin{remark}
 	One can follow the ideas of \cite{amin} and obtain the best constant of \eqref{eq:main}. Indeed,
 	\[
 	C=\frac{d^{\frac1B-\frac12}}{\tau^{\frac\tau2}(1-\tau)^{\frac{1-\tau}2}}, 
 	\]
 	where
 	$d=\inf\sett{S(u)=\frac12\|u\|_V^2-\frac{1}{B}\norm{|x|^\beta u}_B^B,\;S'(u)=0,\;u\neq0}$.
 \end{remark}

	%%%%%%%%%%%%%%%%%%%%%%%%%%%%%%%%%%%%%%%%%%%%%

\end{document}